\documentclass[10pt]{amsart}
\usepackage{graphicx}
\usepackage[all]{xy}

\numberwithin{equation}{section}

\newtheorem{thm}{Theorem}[section]
\newtheorem{prop}[thm]{Proposition}
\newtheorem{conj}[thm]{Conjecture}
\newtheorem{cor}[thm]{Corollary}
\newtheorem{lem}[thm]{Lemma}

\theoremstyle{definition}
\newtheorem{rem}[thm]{Remark}
\newtheorem{ex}[thm]{Example}
\newtheorem{defn}[thm]{Definition}

\begin{document}
\bibliographystyle{amsalpha}

\title[Wonder of Sine-Gordon  $Y$-systems]{
Wonder of Sine-Gordon  $Y$-systems}

\author{Tomoki Nakanishi}
\address{\noindent Graduate School of Mathematics, Nagoya University, 
Chikusa-ku, Nagoya,
464-8604, Japan}
\email{nakanisi@math.nagoya-u.ac.jp}

\author{Salvatore Stella} 
	\address{\noindent Department of Mathematics, North Carolina State University, Box 8205, Raleigh, NC 27695-8205, USA}
	\email{sstella@ncsu.edu}

\begin{abstract}
The  sine-Gordon $Y$-systems and the reduced sine-Gordon $Y$-systems
were introduced by Tateo in the 90's in the study of
the integrable deformation of  conformal field theory
by the thermodynamic Bethe ansatz method.
The periodicity property and the dilogarithm identities
concerning these $Y$-systems were conjectured by Tateo,
and only  a part of them have been proved
so far.
In this paper we formulate these $Y$-systems by the
polygon realization of cluster algebras of types $A$ and $D$,
and prove the 
conjectured periodicity and dilogarithm identities
 in full generality.
 As it turns out,
there is  a wonderful interplay among continued fractions,
triangulations of polygons, cluster algebras, and $Y$-systems.
\end{abstract}



\maketitle

\begin{figure}
	\begin{center}
		\includegraphics[scale=.60]{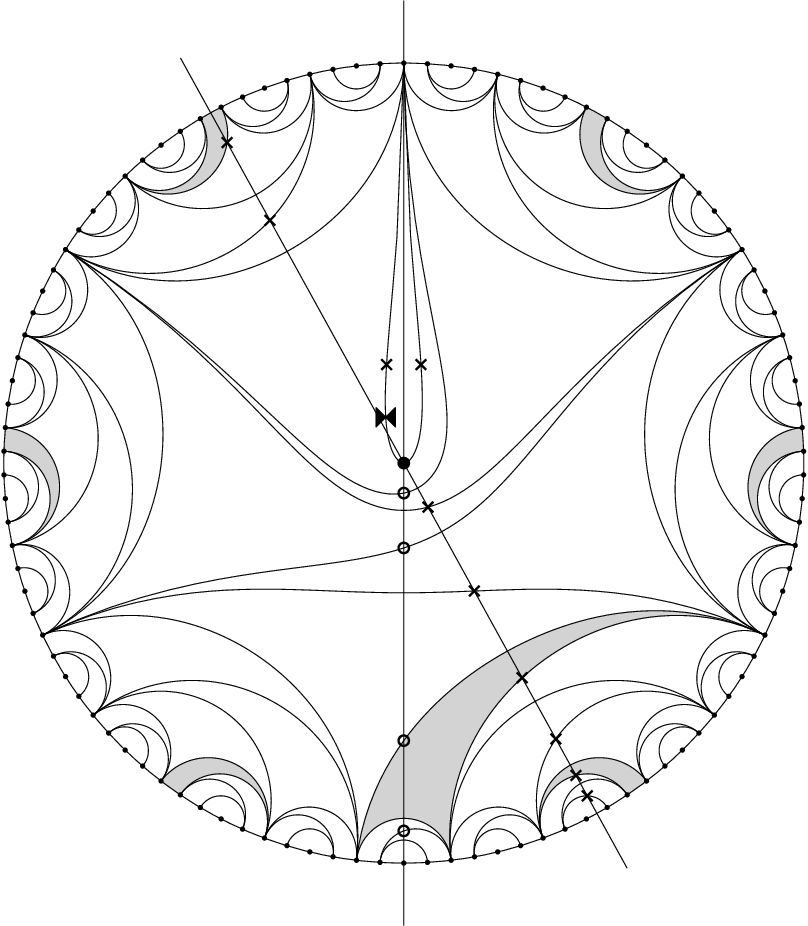}
	\end{center}
\caption{Initial triangulation of 106-gon for
SG $Y$-system $\mathbb{Y}_{\mathrm{SG}}(6,4,3)$.}
\label{fig:106gon}
\end{figure}

\section{Introduction}

In the 90's 
the integrable deformation of conformal field theory
was intensively studied by the thermodynamic Bethe ansatz method.
As a consequence, several
 periodicities of the so-called $Y$-systems
and the associated dilogarithm identities were conjectured.
However, at that time there was no systematic mathematical
framework to prove these conjectures.
As a result, most of  them were left open.
Since the pioneering work by Fomin and Zelevinsky
\cite{Fomin02,Fomin03b},
it has been gradually noticed
that these periodicities of $Y$-systems  in fact come from
 periodicities of cluster algebras.
In this way the  periodicities and 
dilogarithm identities
conjectured for the major family
 of $Y$-systems originated from  quantum affine algebras
were proved using the  cluster algebraic formulation
\cite{Chapoton05, Keller08, Keller10, Nakanishi09, Inoue10a,
Inoue10b}.
Simultaneously, an efficient method of proving periodicities in cluster algebras
has been also
developed by combining
 {\em tropicalization\/}  and {\em categorification\/} technique
 \cite{Inoue10a,Plamondon10b}.

Another interesting family
of $Y$-systems was introduced by Tateo \cite{Tateo95} 
in the same context.
It consists of two subfamilies
 called the {\em sine-Gordon (SG) $Y$-systems\/}
and the {\em reduced sine-Gordon (RSG) $Y$-systems},
respectively.
They are {\em exotic\/} in several  senses.
Firstly, the Lie theoretic interpretation is not entirely clear,
though they are certainly related to the Virasoro algebra,
 $U_q(\widehat{sl}_2)$ at a root of unity, etc.
Secondly,  they are associated with {\em continued fractions}.
Thirdly,  they are systems of functional equations,
some of which appear to be {\em very complicated}.
In the same paper
Tateo also conjectured the periodicity and the dilogarithm identities 
for these $Y$-systems.
Soon later, Gliozzi and Tateo \cite{Gliozzi96} ingeniously
found a general solution for
the RSG $Y$-systems  in terms of
{\em cross-ratios\/},
and thereby proved their periodicity.
However, other conjectures were left open.
The reason for the existence of such a simple solution
for a very complicated system also remained a mystery.

Recently, Tateo and the first author \cite{Nakanishi10b}
formulated a (small) part of the
SG/RSG $Y$-systems by cluster algebras,
then proved their periodicity and the associated
 dilogarithm identities by applying the aforementioned
 tropicalization-categorification method.
However, due to the complexity of these $Y$-systems,
working out the conjecture in full generality by this method
seemed not easy.

In this paper we  prove
Tateo's conjectures on periodicity and dilogarithm
identities of the SG/RSG $Y$-systems {\em in full generality}.
The basic strategy to overcome the above difficulty is the following.
As the referee of the paper  \cite{Nakanishi10b} pointed out,
the cluster algebras for the RSG and SG $Y$-systems
therein are nothing but cluster algebras of types $A$ and $D$, respectively.
 Suppose that it is true in general.
It is well known that the cluster algebras of types $A$ and $D$
admit a {\em surface realization\/} 
 developed by \cite{Gekhtman05,Fomin03a,Fock05,Fomin08,Fomin08b}.
To be more specific, they are realized by
{\em polygons} without puncture for type $A$
and with one puncture for type $D$.
So, we may try to realize these $Y$-systems by polygons.
It turns out that this is possible;
{\em moreover}, the nature of these
$Y$-systems becomes most apparent in the polygon realization.

Let us briefly explain how our method works.
For a given SG/RSG $Y$-system,
the construction of the {\em initial triangulation\/} of a polygon
is our first step,
and it is strongly tied with the theory of continued fractions,
as we should expect.
In fact, this triangulation  is directly associated with
the continued fraction which parametrizes the SG/RSG $Y$-system.
The next key observation is that
the above triangulation has a remarkable {\em quasi-symmetry\/} with
respect to {\em two\/}  axes of the polygon;
moreover,
the mutations realizing the $Y$-system are simply
the {\em reflections\/} with respect to these axes.
For example, Figure \ref{fig:106gon} 
represents the initial triangulation
 (a {\em tagged\/} triangulation in the sense of \cite{Fomin08})
 of
a 106-gon with a puncture corresponding to the SG $Y$-system
associated with the continued fraction
\begin{align}
[3,4,6]:=
\frac{1}{
\displaystyle
3+\frac{1}{
\displaystyle
4+\frac{1}{6}}}
= \frac{25}{81}.
\end{align}
The number 106 comes from $25+81=106$.
Observe that the triangulation is almost symmetric
with respect to the two axes therein;
moreover,
{\em flipping (mutating)\/} the  diagonals which
cross one of  the axes and are not symmetric with respect to it,
we get the reflection.
Composing these two reflections, we get a rotation.
This geometric realization of the SG/RSG $Y$-system
 enables us to prove the conjectured periodicities at once;
 they are nothing but  the {\em full rotations\/}  of the associated polygons.
 From these triangulations,
we can also   reconstruct the cross-ratio
 solution 
 of the RSG $Y$-systems 
 by \cite{Gliozzi96} naturally.
 Furthermore, the conjectured dilogarithm identities
 reduce to a simple geometrical counting problem for
 the triangulations.


Throughout our work
we found a wonderful interplay among
continued fractions, triangulations of polygons,
cluster algebras, and $Y$-systems.
This is perhaps the most important message of the paper.
Also, we ask the reader not  to be discouraged
by the ``horrible appearance'' of the RSG/SG $Y$-systems
in Section \ref{sec:RSGY},
since they are beautiful in nature,
as you  see in Figure \ref{fig:106gon}.

The paper is organized as follows.
In Section 2 we recall some basic properties of continued fractions.
In Section 3 we introduce the RSG/SG $Y$-systems.
In Section 4, using some examples,
we explain in detail
 how to  realize
 the RSG $Y$-systems  by  polygons.
 In particular, we introduce the {\em snapshot method\/}
 to obtain the relations in these $Y$-systems.
Section 5 is the main part of the paper.
To any
 continued fraction
 we construct the associated  triangulation of a polygon,
 and we show that it provides a realization of the 
corresponding RSG $Y$-system.
As a result, we obtain the periodicity of the RSG $Y$-systems,
and also reproduce the solution by Gliozzi-Tateo.
In Section 6 we adapt our construction of the triangulation
for a polygon
with a puncture. Then, we show that
it provides a realization of the corresponding
SG $Y$-system.
As a result, we obtain the periodicity of the SG $Y$-systems.
In Section 7 we prove the  dilogarithm identities conjectured
by Tateo.
In Section 8 we introduce the RSG/SG $T$-systems
accompanying the RSG/SG $Y$-systems.
They share the same periodicity
with the RSG/SG $Y$-systems.

\bigskip
\noindent
{\em Acknowledgements.} 
We thank Hugh Thomas and Dylan Thurston for useful discussion.
We thank MSRI, Berkeley
for financial support and for providing the ideal environment
where this work was done.
The second author is partially supported by A. Zelevinsky's NSF grant
DMS-1103813.

\section{Continued fractions}

Before starting, let us summarize some useful properties
of continued fractions which will be used throughout the paper.
The results are standard in the literature
(e.g., \cite{Wall48}).

First, we fix an arbitrarily  positive integer  $F$.
Then, we fix a sequence of positive integers
$(n_1, \dots, n_F)$ with $n_1\geq 2$.
This is our input data.
The  sequence $(n_1, \dots, n_F)$ 
determines
a rational number $0<\xi <1$
by the continued fraction
\begin{align}
\xi&=[n_F,\dots,n_1]
:=
\frac
{ 1
}
{\displaystyle
n_F+
\frac{
1
}
{\displaystyle n_{F-1}+
\frac{ 1}{
{\displaystyle
\ddots
+
\frac{
1
}
{\displaystyle  n_{1}
}
}
}
}
}.
\end{align}
To make it clear, $[n_1]=1/n_1$.
Conversely, any rational number $0<\xi <1$
is uniquely expressed in this form.
Therefore, the correspondence is one-to-one.
The order of the subscript of $n_i$ is opposite to the standard one,
but this is convenient for our purpose.

For the sequence $(n_1,\dots,n_F)$,
we introduce a family of continued fractions
\begin{align}
\xi_a&=[n_a,\dots,n_1],
\quad
1\leq a\leq F.
\end{align}
We write $\xi_a$ with coprime integers $p_a$, $q_a$ as
\begin{align}
\label{eq:xia}
\xi_a =\frac{p_a}{q_a},
\quad
1\leq a\leq F.
\end{align}
In particular, $(p_1,q_1)=(1,n_1)$.
Since 
$\xi_a=1/(n_a + \xi_{a-1})$ for $2\leq a\leq F$,
we have
\begin{align}
\frac{p_a}
{q_a}
=
\frac{1}
{n_a + \displaystyle \frac{p_{a-1}}{q_{a-1}}
}
=
\frac{q_{a-1}}
{n_a q_{a-1}+p_{a-1}}.
\end{align}
Then, thanks to the coprime property, we have
the relations
\begin{align}
\label{eq:p-q1}
p_a&=q_{a-1},
\quad
2\leq a\leq F,\\
\label{eq:p-q2}
q_a&=
{n_a q_{a-1}+p_{a-1}},
\quad
2\leq a\leq F.
\end{align}
Thus, all $p_a$ and $q_a$ are determined from
the recursion relations,
\begin{align}
\label{eq:xi1}
q_a&=n_aq_{a-1}+q_{a-2},
\quad
2\leq a\leq F,\\
\label{eq:xi2}
p_a&=n_{a-1}p_{a-1}+p_{a-2},
\quad
3\leq a\leq F,
\end{align}
with the initial condition $q_0:=1$, $q_1=n_1$,
$p_1=1$, $p_2=n_1$.

More generally, for
 the same sequence $(n_1$, \dots, $n_F)$
 and for any $k=1,\dots, F$,
we introduce a  family of continued fractions
\begin{align}
\xi^{(k)}_a=[n_a,\dots,n_k],
\quad
 k\leq a\leq F.
\end{align}
Thus, $\xi_a=\xi^{(1)}_a$, though we mainly use
the former notation,
since they are the ``principals''.
We write $\xi^{(k)}_a$ with coprime integers $p^{(k)}_a$, $q^{(k)}_a$ as
\begin{align}
\xi^{(k)}_a =\frac{p^{(k)}_a}{q^{(k)}_a},
\quad
 k\leq a\leq F.
\end{align}
In particular, $(p^{(k)}_k,q^{(k)}_k)=(1,n_k)$.
Then, by the same argument as before,
we have $p^{(k)}_a=q^{(k)}_{a-1}$,
and
all $p^{(k)}_a$ and $q^{(k)}_a$ are determined from
the recursion relations,
\begin{align}
\label{eq:xid1}
q^{(k)}_a&=n_aq^{(k)}_{a-1}+q^{(k)}_{a-2},
\quad
 k+1\leq a\leq F,\\
\label{eq:xid2}
p^{(k)}_a&=n_{a-1}p^{(k)}_{a-1}+p^{(k)}_{a-2},
\quad
 k+2\leq a\leq F,
\end{align}
with the initial condition $q^{(k)}_{k-1}:=1$, $q^{(k)}_k=n_k$,
$p^{(k)}_k=1$, $p^{(k)}_{k+1}=n_k$.
We define  integers
\begin{align}
\label{eq:rad}
r^{(k)} := p^{(k)}_{F}+q^{(k)}_{F},
\quad
 1\leq k\leq F,
\end{align}
and we especially write $r^{(1)}$  as $r$
in accordance with the notation $\xi_F=\xi^{(1)}_F$.

\begin{ex}
\label{ex:743}
We use the following data as a running example throughout the paper:
For $F=3$, $(n_1,n_2,n_3)=(6,4,3)$,  we have 
\begin{alignat}{4}
\xi_1&=
\frac{1}{6},
&\quad
(p_1,q_1)&=(1,6),
&\quad
\\
\label{eq:6-25}
\xi_2&=
\frac{1}{4+\frac{1}{6}}=\frac{6}{25},
&\quad
(p_2,q_2)&=(6,25),
&\quad
\\
\xi_3&=
\frac{1}{3+\frac{6}{25}}=\frac{25}{81},
&\quad
(p_3,q_3)&=(25,81),
&\quad
r&=106,
\\
\xi^{(2)}_2&=
\frac{1}{4},
&\quad
(p^{(2)}_2,q^{(2)}_2)&=(1,4),
&\quad
\\
\xi^{(2)}_3&=
\frac{1}{3+\frac{1}{4}}=\frac{4}{13},
&\quad
(p^{(2)}_3,q^{(2)}_3)&=(4,13),
&\quad
r^{(2)}&=17,\\
\xi^{(3)}_3&=
\frac{1}{3},
&\quad
(p^{(3)}_3,q^{(3)}_3)&=(1,3),
&\quad
r^{(3)}&=4.
\end{alignat}
\end{ex}

The following formulas are well known (e.g., \cite{Wall48}).

\begin{prop}
\label{prop:cf0}
(a). (Fundamental recurrence formulas)
For $k=1,\dots,F-2$,
we have
\begin{align}
\label{eq:rr2}
q^{(k)}_a&=n_kq^{(k+1)}_a+q^{(k+2)}_a,
\quad
 k+2\leq a\leq F,
\\
\label{eq:rr3}
p^{(k)}_a&=n_kp^{(k+1)}_a+p^{(k+2)}_a,
\quad
 k+2\leq a\leq F.
\end{align}
\par
(b). 
For $k=1,\dots,F-1$,
we have
\begin{align}
\label{eq:pq-qp}
q^{(k)}_a p^{(k+1)}_a
-
q^{(k+1)}_a p^{(k)}_a
=(-1)^{a-k+1},
\quad
 k+1\leq a\leq F.
\end{align}
\end{prop}
\begin{proof}
(a).
They are easily proved by
induction on $a$
using \eqref{eq:xid2}.
(b). For each $a$, this is proved  by 
induction on $k$ in the decreasing order
using (a).
\end{proof}

We will also use the following properties later.

\begin{prop}
\label{prop:cf1}
(a). For $k=1,\dots,F$, we have
\begin{align}
\label{eq:rr4}
r^{(k)}&=n_kr^{(k+1)}+r^{(k+2)},
\end{align}
where we set
$r^{(F+1)}=
r^{(F+2)}:=1$.
\par
(b). For $k=1,\dots,F$, we have
\begin{align}
\label{eq:rF3}
r^{(k)}&=q^{(k)}_{a} r^{(a+1)} + p^{(k)}_{a}r^{(a+2)},
\quad
 k\leq a\leq F.
\end{align}
For $a=k$, it reduces to (a).
\par
(c).
For each $a=2,\dots, F$,
the numbers $p_a$ and $p^{(2)}_a$ are coprime.
\par
(d).
The numbers $r$ and $r^{(2)}$ are coprime.
\par
(e). For $a=3,\dots,F$, we have
\begin{align}
\label{eq:pp0}
p_{a-1}p^{(2)}_a
-
p_{a}p^{(2)}_{a-1}
= (-1)^a.
\end{align}
\end{prop}

\begin{proof}
(a). 
 For $k\leq F-2$,
 this is a corollary of 
 Proposition \ref{prop:cf0} (a).
 For $k=F-1, F$,
 this is checked by direct inspection;
 indeed
$r^{(F-1)}=n_{F-1}(n_F+1)+1$
and $r^{(F)}=n_F+1$.
(b).
This can be proved by  induction on $a$
using
\eqref{eq:xid2} and \eqref{eq:rr4}.
(c).
Using  \eqref{eq:rr3} repeatedly, we have
$\mathrm{gcd}(p_a,p^{(2)}_a)
=
\mathrm{gcd}(p^{(3)}_a,p^{(2)}_a)
=\cdots
=
\mathrm{gcd}(p^{(a-1)}_a,p^{(a)}_{a})
=1.
$
The claim also follows from (e) below.
(d).
Using  \eqref{eq:rr4} repeatedly, we have
$\mathrm{gcd}(r^{(k)},r^{(k+1)})
=
\mathrm{gcd}(r^{(k+2)},r^{(k+1)})
=\cdots
=
\mathrm{gcd}(r^{(F)},r^{(F+1)})
=1.
$
(e). This is a special case of 
Proposition \ref{prop:cf0} (b) with $k=1$.
\end{proof}

\section{RSG and SG $Y$-systems}
\label{sec:RSGY}

Here we introduce the RSG and SG $Y$-systems
following
\cite{Tateo95}.
For the background of these equations
 in conformal field theory,
consult \cite{Tateo94,Tateo95}
and \cite[Section 2.3]{Nakanishi10b}.
Then we state the periodicity property
of these $Y$-systems.

\subsection{RSG $Y$-systems}
\label{subsec:RSGY}

We continue to use the sequence
$(n_1,\dots, n_F)$ in the previous section
as input data.
We exclude the case $(n_1)=(2)$ with $F=1$,
because the RSG $Y$-system is empty.

We introduce the notation for the signs
\begin{align}
\label{eq:sign1}
\varepsilon_a:= (-1)^{a-1},
\quad
a=1,\dots, F.
\end{align}

We start with the case $n_1\neq 2$.
Let us introduce the {\em $Y$-variables\/} 
$ Y^{(a)}_m(u)$,
where $u\in \mathbb{Z}$,
$a=1\dots, F$, and
\begin{align}
\label{eq:RSGam}
m=
\begin{cases}
1,\dots, n_1-2&\mbox{if  $a=1$}\\
1,\dots, n_a & \mbox{if $a=2,\dots,F$}.
\end{cases}
\end{align}
Let $X_{\mathrm{RSG}}(n_1,\dots,n_F)$ be the Dynkin diagram
of type $A$ indexed by $(a,m)$ in the above range
as shown in Figure \ref{fig:XA}.
\begin{figure}
$$
\begin{xy}
(0,0)*\cir<2pt>{},
(10,0)*\cir<2pt>{},
(20,0)*\cir<2pt>{},
(30,0)*\cir<2pt>{},
(40,0)*\cir<2pt>{},
(50,0)*\cir<2pt>{},
(70,0)*\cir<2pt>{},
(80,0)*\cir<2pt>{},
(90,0)*\cir<2pt>{},
(-10,5)*{a},
(10,5)*{\overbrace{\phantom{xxxxxxxxx}}^{\displaystyle 1}},
(40,5)*{\overbrace{\phantom{xxxxxxxxx}}^{\displaystyle 2}},
(80,5)*{\overbrace{\phantom{xxxxxxxxx}}^{\displaystyle F}},
(-10,-5)*{m},
(0,-5)*{1},
(10,-5.5)*{\cdots},
(20,-5.5)*{n_1-2},
(30,-5)*{1},
(40,-5.5)*{\cdots},
(50,-5.5)*{n_{2}},
(70,-5)*{1},
(80,-5.5)*{\cdots},
(90,-5.5)*{n_{F}},
\ar@{-} (1,0);(9,0)
\ar@{-} (11,0);(19,0)
\ar@{-} (21,0);(29,0)
\ar@{-} (31,0);(39,0)
\ar@{-} (41,0);(49,0)
\ar@{--} (51,0);(69,0)
\ar@{-} (71,0);(79,0)
\ar@{-} (81,0);(89,0)
\end{xy}
$$
\caption{The diagram $X_{\mathrm{RSG}}(n_1,\dots,n_F)$ for $n_1\neq 2$.}
\label{fig:XA}
\end{figure}
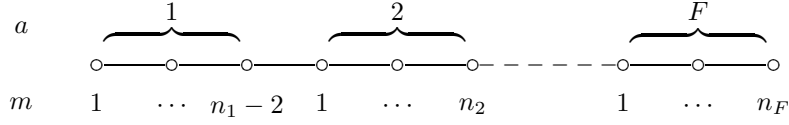

\begin{defn}
\label{defn:RSG}
For $n_1\neq 2$,
the {\em reduced sine-Gordon (RSG) $Y$-system\/}
$\mathbb{Y}_{\mathrm{RSG}}(n_1,\dots,n_F)$ is the following system of
relations:
For a general $(a,m)$ other than
$(2,1)$, $(3,1)$, \dots, $(F,1)$,
\begin{align}
\label{eq:RSG1}
Y^{(a)}_m(u-p_a)Y^{(a)}_m(u+p_a)
&=
\prod_{(b,k)\sim(a,m)}
(1+Y^{(b)}_{k}(u)^{\varepsilon_b})^{\varepsilon_b},
\end{align}
where $(b,k)\sim (a,m)$ means $(b,k)$ is adjacent to
$(a,m)$ in the diagram $X_{\mathrm{RSG}}(n_1,\dots,n_F)$,
and $p_a$ ($1\leq a \leq F$) are the numbers
defined in \eqref{eq:xia}.
Exceptional cases are as follows.
For $(a,m)=(2,1)$,
\begin{align}
\label{eq:RSG5}
\begin{split}
&Y^{(2)}_{1}(u-p_{2})Y^{(2)}_{1}(u+p_{2})\\
&\qquad
=
(1+Y^{(2)}_{2}(u)^{-1})^{-1}
(1+Y^{(1)}_{1}(u))\\
&
\qquad
\quad
\times
\prod_{m=1}^{n_{1}-2}
(
1+Y^{(1)}_m(u -1-m)^{-{1}}
)^{-{1}}
(
1+Y^{(1)}_m(u +1+m)^{-{1}}
)^{-{1}}.
\end{split}
\end{align}
For $(a,m)=(a,1)$ with $a=3,\dots,F$,
\begin{align}
\label{eq:RSG4}
\begin{split}
&Y^{(a)}_{1}(u-p_a)Y^{(a)}_{1}(u+p_a)\\
&\qquad
=
(1+Y^{(a)}_{2}(u)^{\varepsilon_a})^{\varepsilon_a}
(1+Y^{(a-2)}_{n_{a-2}-2\delta_{a3}}(u)^{\varepsilon_{a}})^{\varepsilon_{a}}\\
&
\qquad
\quad
\times
\prod_{m=1}^{n_{a-1}}
(
1+Y^{(a-1)}_m(u -p_a+(n_{a-1}+1-m)p_{a-1})^{\varepsilon_{a}}
)^{\varepsilon_{a}}\\
&
\qquad
\quad
\times
\prod_{m=1}^{n_{a-1}}
(
1+Y^{(a-1)}_m(u +p_a-(n_{a-1}+1-m)p_{a-1})^{\varepsilon_{a}}
)^{\varepsilon_{a}},
\end{split}
\end{align}
where $\delta_{a3}$ is the Kronecker delta.
\end{defn}

\begin{rem}
\label{rem:RSGY}
The variables $Y_1(u)$, $Y_2(u)$, \dots
in \cite{Tateo95} and \cite{Nakanishi10b}
are $Y^{(F)}_{n_F}(u)$, $Y^{(F)}_{n_F-1}(u)$, \dots, $Y^{(F-1)}_{n_{F-1}}(u)$,
\dots, $Y^{(1)}_{1}(u)$  here with a certain rescale of $u$,
and also up to the opposite convention of the
sign $\varepsilon_a$ for even $F$.
\end{rem}

When $F=1$, only the relation
\eqref{eq:RSG1} is involved. 
This is the well known {\em $Y$-system of type $A_{n_1-2}$\/} (with level 2,
in the terminology of \cite{Kuniba94a}).
So, the RSG $Y$-systems are generalizations of the $Y$-systems of type $A$.

When $n_1=2$ with $F\geq 2$,
 we need some modification.
We reset the   $Y$-variables
 $Y^{(a)}_m(u)
 $,
where $u\in \mathbb{Z}$,
$a=2\dots, F$, and
$m=
1,\dots, n_a$.
Let $X_{\mathrm{RSG}}(2,n_2,\dots,n_F)$ be the {\em tadpole\/} diagram
 indexed by $(a,m)$ in the above range
as shown in Figure \ref{fig:XA2}.

\begin{defn}
For $F\geq 2$,
the {\em RSG $Y$-system\/}
$\mathbb{Y}_{\mathrm{RSG}}(2,n_2,\dots,n_F)$ is the following system of
relations:
\begin{itemize}
\item[(i).]
 the relation \eqref{eq:RSG1} for $(a,m)$ other than
$(3,1),\dots, (F,1)$,
where the adjacency diagram
for \eqref{eq:RSG1} is  $X_{\mathrm{RSG}}(2, n_2,\dots,n_F)$,
\item[(ii).]
 the relation \eqref{eq:RSG4} for $(a,1)$
with $a=3,\dots, F$, 
where for $a=3$ the term $Y^{(1)}_0(u)$
in the right hand side of \eqref{eq:RSG4}   is discarded.
\end{itemize}
\end{defn}

\begin{figure}
$$
\begin{xy}
(30,0)*\cir<2pt>{},
(40,0)*\cir<2pt>{},
(50,0)*\cir<2pt>{},
(70,0)*\cir<2pt>{},
(80,0)*\cir<2pt>{},
(90,0)*\cir<2pt>{},
(10,5)*{a},
(40,5)*{\overbrace{\phantom{xxxxxxxxx}}^{\displaystyle 2}},
(80,5)*{\overbrace{\phantom{xxxxxxxxx}}^{\displaystyle F}},
(10,-5)*{m},
(30,-5)*{1},
(40,-5.5)*{\cdots},
(50,-5.5)*{n_{2}},
(70,-5)*{1},
(80,-5.5)*{\cdots},
(90,-5.5)*{n_{F}},
(29.5,1); (29.5,-1) **\crv{(25,7) & (12,0) & (25,-7)};
\ar@{-} (31,0);(39,0)
\ar@{-} (41,0);(49,0)
\ar@{--} (51,0);(69,0)
\ar@{-} (71,0);(79,0)
\ar@{-} (81,0);(89,0)
\end{xy}
$$
\caption{The diagram $X_{\mathrm{RSG}}(2,n_2,\dots,n_F)$.}
\label{fig:XA2}
\end{figure}

When $F=2$, only the relation
\eqref{eq:RSG1} is involved. 
This is the well known {\em $Y$-system of tadpole type $T_{n_2}$} of \cite{Ravanini93}).
So, these RSG $Y$-systems
 are generalizations of the $Y$-systems of tadpole type.

All  main results in the paper are
applicable,
whether $n_1\neq 2$ or $n_1=2$.
To make the description simpler,
from now on we  do not pay a special attention to the 
exceptional case $n_1=2$ when claiming and proving general statements.
The reader can safely concentrate on the case $n_1\neq 2$.

Now let us return to Definition \ref{defn:RSG}.
In the right hand side of \eqref{eq:RSG5}, we have
\begin{align}
u-p_2+ 1\leq u-1-m \leq u-2
\end{align}
thanks to $p_2=n_1$ and $p_1=1$.
Also,
in the right hand side of \eqref{eq:RSG4}, we have
\begin{align}
u -p_a +p_{a-1}\leq u-p_a+(n_{a-1}+1-m)p_{a-1} \leq
 u-p_{a-2}
\end{align}
thanks to the relation $-p_a+n_{a-1}p_{a-1}=-p_{a-2}$ in \eqref{eq:xi2}.

Let $\mathcal{Y}=\mathcal{Y}_{\mathrm{RSG}}(n_1,\dots,n_F)$ be
 the set of all  $Y$-variables of 
 $\mathbb{Y}_{\mathrm{RSG}}(n_1,\dots,n_F)$.
Let $\mathcal{Y}_+$ (resp. $\mathcal{Y}_-$) be the subset of 
$\mathcal{Y}$ consisting of $Y^{(a)}_m(u)$ such that
the integer
\begin{align}
\label{eq:bi1}
\theta^{(a)}_m(u):=u+p_{a+1}-(n_a-m)p_a
\end{align}
is even (resp. odd).
It is easy to check the following  property.

\begin{prop}
\label{prop:bi1}
In each relation of
the $Y$-system $\mathbb{Y}_{\mathrm{RSG}}(n_1,\dots,n_F)$,
if the variables in the left hand side are in $\mathcal{Y}_{+}$
(resp. $\mathcal{Y}_{-}$),
then 
the variables in the right hand side are also in
 $\mathcal{Y}_{+}$
(resp. $\mathcal{Y}_{-}$).
\end{prop}

Therefore,
one can {\em bisect\/} the RSG $Y$-system 
into the one for  $\mathcal{Y}_+$
and  the one for $\mathcal{Y}_-$.
They are equivalent systems related by the shift of
the parameter $u\rightarrow u+1$.
So it is enough to concentrate on the $Y$-systems
for  $\mathcal{Y}_+$.

\subsection{SG $Y$-systems}

Again, we exclude the case $(n_1)=(2)$ with $F=1$,
because the SG $Y$-system is equivalent to
the RSG $Y$-system with $(n_1)=(5)$ with $F=1$.

Let us reset the  $Y$-variables 
$ Y^{(a)}_m(u) $,
where  $u\in \mathbb{Z}$,
$a=1\dots, F$, and
\begin{align}
\label{eq:SGam}
m=
\begin{cases}
\overline{1},\overline{2},0,1,\dots, n_1-2&\mbox{if  $a=1$}\\
1,\dots, n_a & \mbox{if $a=2,\dots,F$}.
\end{cases}
\end{align}
Note that the three indices $\overline{1}$, $\overline{2}$, $0$
are added for $a=1$.
Let $X_{\mathrm{SG}}(n_1,\dots,n_F)$ be the Dynkin diagram
of type $D$ indexed by $(a,m)$ in the above range
as shown in Figure \ref{fig:XD}.

\begin{figure}
$$
\begin{xy}
(-20,5)*\cir<2pt>{},
(-20,-5)*\cir<2pt>{},
(-10,0)*\cir<2pt>{},
(0,0)*\cir<2pt>{},
(10,0)*\cir<2pt>{},
(20,0)*\cir<2pt>{},
(30,0)*\cir<2pt>{},
(40,0)*\cir<2pt>{},
(50,0)*\cir<2pt>{},
(70,0)*\cir<2pt>{},
(80,0)*\cir<2pt>{},
(90,0)*\cir<2pt>{},
(-30,5)*{a},
(0,10)*{\overbrace{\phantom{xxxxxxxxxxxxxxxxxx}}^{\displaystyle 1}},
(40,5)*{\overbrace{\phantom{xxxxxxxxx}}^{\displaystyle 2}},
(80,5)*{\overbrace{\phantom{xxxxxxxxx}}^{\displaystyle F}},
(-30,-5)*{m},
(-10,-5)*{0},
(0,-5)*{1},
(10,-5.5)*{\cdots},
(20,-5.5)*{n_1-2},
(30,-5)*{1},
(40,-5.5)*{\cdots},
(50,-5.5)*{n_{2}},
(70,-5)*{1},
(80,-5.5)*{\cdots},
(90,-5.5)*{n_F},
(-20, 1)*{\overline{1}},
(-20,-9)*{\overline{2}},
\ar@{-} (-11,0.5);(-19,4.5)
\ar@{-} (-11,-0.5);(-19,-4.5)
\ar@{-} (-1,0);(-9,0)
\ar@{-} (1,0);(9,0)
\ar@{-} (11,0);(19,0)
\ar@{-} (21,0);(29,0)
\ar@{-} (31,0);(39,0)
\ar@{-} (41,0);(49,0)
\ar@{--} (51,0);(69,0)
\ar@{-} (71,0);(79,0)
\ar@{-} (81,0);(89,0)
\end{xy}
$$
\caption{The diagram $X_{\mathrm{SG}}(n_1,\dots,n_F)$.}
\label{fig:XD}
\end{figure}

\begin{defn}
The {\em sine-Gordon (SG) $Y$-system\/}
$\mathbb{Y}_{\mathrm{SG}}(n_1,\dots,n_F)$ is the following system of
relations:
\begin{itemize}
\item[(i).] the relation  \eqref{eq:RSG1}
for $(a,m)$ other than
$(2,1)$, $(3,1)$, \dots, $(F,1)$,
where the adjacency diagram
for \eqref{eq:RSG1} is  replaced with
 $X_{\mathrm{SG}}(n_1,\dots,n_F)$,
\item[(ii).] the relation for $(a,m)=(2,1)$, 
\begin{align}
\label{eq:SG5}
\begin{split}
&Y^{(2)}_{1}(u-p_{2})Y^{(2)}_{1}(u+p_{2})\\
&\qquad
=
(1+Y^{(2)}_{2}(u)^{-1})^{-1}
(1+Y^{(1)}_{\overline{1}}(u)^{-1})^{-1}
(1+Y^{(1)}_{\overline{2}}(u)^{-1})^{-1}
\\
&
\qquad
\quad
\times
\prod_{m=0}^{n_{1}-2}
(
1+Y^{(1)}_m(u -1-m)^{-{1}}
)^{-{1}}
(
1+Y^{(1)}_m(u +1+m)^{-{1}}
)^{-{1}},
\end{split}
\end{align}
\item[(iii).] the relation \eqref{eq:RSG4} for $(a,m)=(a,1)$ with $a=3,\dots,F$.
\end{itemize}
 \end{defn}

When $F=1$, only the relation
\eqref{eq:RSG1} is involved.
This is the well known {\em $Y$-system of type $D_{n_1+1}$\/} (with level 2).
So, the SG $Y$-systems are generalizations of the $Y$-systems of type $D$.

As the name suggests,
the RSG $Y$-system
$\mathbb{Y}_{\mathrm{RSG}}(n_1,\dots,n_F)$
is obtained from $\mathbb{Y}_{\mathrm{SG}}(n_1,\dots,n_F)$
by the following reduction \cite{Tateo95}.
The relation \eqref{eq:RSG1} at $(a,m)=(1,0)$ reads
\begin{align}
Y^{(1)}_{0}(u-1)
Y^{(1)}_{0}(u+1)
=
(1+Y^{(1)}_{1}(u))
(1+Y^{(1)}_{\overline{1}}(u))
(1+Y^{(1)}_{\overline{2}}(u)).
\end{align}
Thus, \eqref{eq:SG5} is equivalent to
\begin{align}
\label{eq:SG6}
\begin{split}
&Y^{(2)}_{1}(u-p_{2})Y^{(2)}_{1}(u+p_{2})\\
&\qquad=
Y^{(1)}_{\overline{1}}(u)
Y^{(1)}_{\overline{2}}(u)
(1+Y^{(2)}_{2}(u)^{-1})^{-1}
(1+Y^{(1)}_{1}(u))
\\
&
\qquad
\quad
\times
(1+Y^{(1)}_{0}(u-1))^{-{1}}
(1+Y^{(1)}_{0}(u+1))^{-{1}}
\\
&\qquad
\quad
\times
\prod_{m=1}^{n_{1}-2}
(
1+Y^{(1)}_m(u -1-m)^{-{1}}
)^{-{1}}
(
1+Y^{(1)}_m(u +1+m)^{-{1}}
)^{-{1}}.
\end{split}
\end{align}
Then, under the specialization
\begin{align}
\label{eq:reduction}
Y^{(1)}_{0}(u)=0,
\quad
Y^{(1)}_{\overline{1}}(u)= Y^{(1)}_{\overline{2}}(u)=-1,
\end{align}
$\mathbb{Y}_{\mathrm{SG}}(n_1,\dots,n_F)$
reduces to
$\mathbb{Y}_{\mathrm{RSG}}(n_1,\dots,n_F)$.

Let $\mathcal{Y}=\mathcal{Y}_{\mathrm{SG}}(n_1,\dots,n_F)$ be
 the set of all $Y$-variables of 
 $\mathbb{Y}_{\mathrm{SG}}(n_1,\dots,n_F)$.
Let $\mathcal{Y}_+$ (resp. $\mathcal{Y}_-$) be the subset of 
$\mathcal{Y}$ consisting of $Y^{(a)}_m(u)$ such that
the integer
\begin{align}
\label{eq:bi2}
\theta^{(a)}_m(u):=
\begin{cases}
u+p_{a+1}-(n_a-m)p_a &
(a,m)\neq (1,\overline{1}), (1,\overline{2})\\
u +1 &
(a,m)= (1,\overline{1}), (1,\overline{2})\\
\end{cases}
\end{align}
is even (resp. odd).

Again, it is easy to check the following property.

\begin{prop}
In each relation of
the $Y$-system $\mathbb{Y}_{\mathrm{SG}}(n_1,\dots,n_F)$,
if the variables in the left hand side are in $\mathcal{Y}_{+}$
(resp. $\mathcal{Y}_{-}$),
then 
the variables in the right hand side are also in
 $\mathcal{Y}_{+}$
(resp. $\mathcal{Y}_{-}$).
\end{prop}

So, it is enough to concentrate on
 the SG $Y$-systems for  $\mathcal{Y}_+$.

\subsection{Periodicity}
The first main result of the paper is to prove the following
remarkable periodicity of the RSG and SG $Y$-systems
conjectured by \cite{Tateo95}.

Recall that  $r=r^{(1)}$ is the number defined in \eqref{eq:rad}.

\begin{thm} [\cite{Gliozzi96}]
\label{thm:periodRSG}
The RSG $Y$-system 
$\mathbb{Y}_{\mathrm{RSG}}(n_1,\dots,n_F)$ has the following periodicity.\begin{align}
Y^{(a)}_m(u + 2 r) = Y^{(a)}_m(u).
\end{align}
Furthermore, the above period is the minimal one
except for the trivial case $(n_1)=(3)$ with $F=1$,
where, not $2r$, but $r=4$ is the minimal one.
\end{thm}

\begin{thm} [{\cite{Nakanishi10b} for $F=2$}]
\label{thm:periodSG}The SG $Y$-system 
$\mathbb{Y}_{\mathrm{SG}}(n_1,\dots,n_F)$ has the following periodicity.
\par
(i). If $r$ is even, we have
\begin{align}
Y^{(a)}_m(u + 2r) = Y^{(a)}_m(u).
\end{align}
\par
(ii). If $r$ is odd, we have
\begin{align}
\mbox{(half periodicity)}\quad
Y^{(a)}_m(u + 2r) &= 
\begin{cases}
Y^{(1)}_{\overline{2}}(u) & (a,m)=(1,\overline{1})\\
Y^{(1)}_{\overline{1}}(u) & (a,m)=(1,\overline{2})\\
Y^{(a)}_m(u) & \mbox{otherwise},
\end{cases}\\
\mbox{(full periodicity)}\quad
Y^{(a)}_m(u +  4r) & = Y^{(a)}_m(u).
\end{align}
Furthermore, the above period is the minimal one.
\end{thm}

Theorem \ref{thm:periodRSG} was proved in \cite{Gliozzi96} by solving the
$Y$-system in terms of cross-ratios of $r$ points.
Also, Theorems \ref{thm:periodRSG} and \ref{thm:periodSG} were   proved
by \cite{Nakanishi10b} for $F=2$ by the cluster algebraic 
formulation of the RSG and SG $Y$-systems.

Proofs of Theorems  \ref{thm:periodRSG}
and  \ref{thm:periodSG}
will be given in 
Sections \ref{subsec:periodRSG} and \ref{subsec:periodSG},
respectively.

\begin{rem}
\label{rem:period1}
(a). The statement of periodicity in Theorem \ref{thm:periodSG}
was not correctly stated in \cite{Tateo95} in the case $r$ is even,
 and corrected in the above form in \cite{Nakanishi10b} in the case of $F=2$.
 (For $F=1$, it is well known \cite{Fomin07}.)
\par
(b). 
The following (half) periodicity property of the RSG $Y$-systems for $F=1$
  is well known \cite{Gliozzi96,Frenkel95}):
For $\mathbb{Y}_{\mathrm{RSG}}(n_1)$,
\begin{align}
Y^{(1)}_m(u + r) = Y^{(1)}_{n_1-1-m}(u).
\end{align}
For $n_1=3$, it degenerates to the full periodicity,
so that we have the exception in Theorem \ref{thm:periodRSG}.
\par
(c). Theorem \ref{thm:periodRSG} also follows from 
Theorem \ref{thm:periodSG}
by the reduction \eqref{eq:reduction}.
\end{rem}

\section{Example: RSG $Y$-system $\mathbb{Y}_{\mathrm{RSG}}(6,4,3)$
}
\label{sec:643}

We start from the formulation of the RSG $Y$-systems
by polygons.
It is simpler than the one for the SG $Y$-systems,
because  we do not have to consider punctures.

In this section
we concentrate on  our running example $F=3$, $(n_1,n_2,n_3)=(6,4,3)$
set in Example \ref{ex:743}.
We explain in detail  how to formulate the
RSG $Y$-system 
$\mathbb{Y}_{\mathrm{RSG}}(6,4,3)$
by a $106$-gon, where $106=r$ for $(6,4,3)$,
and also how to prove the periodicity.
We heuristically derive the construction of its polygon realization
in three steps  along the {\em generations}.
Once we understand the essence of this example,
the generalization is not so difficult.

\subsection{Mutation in cluster algebras}
\label{subsec:exchange}

Here we recall
(a part of) the mutation rule  of a cluster algebra
with coefficients.
See \cite{Fomin07} for detail.
Let $(B,y)$ be a labeled $Y$-seed, consisting of 
an exchange matrix $B=(b_{ij})_{i,j\in I}$,
 and a coefficient tuple $y=(y_i)_{i\in I}$,
 where
 coefficients $y_i$ are in the universal semifield
 generated by the initial coefficients.
The mutation $(B',y')=\mu_k(B,y)$ at $k\in I$ is defined by
 the following rule.
  \begin{align}
 b'_{ij}&=
 \begin{cases}
 -b_{ij} & \mbox{$i=k$ or $j=k$}\\
 b_{ij}
 + b_{ik}[b_{kj}]_+
 +[-b_{ik}]_+  b_{kj}  & i,j\neq k,
 \end{cases}\\
   \label{eq:yrel}
 y'_{i}&=
 \begin{cases}
 y_{i}^{-1} &i=k\\
 \displaystyle
 y_{i}
 \frac{(1+ y_k)^{[-b_{ki}]_+ }}
  {(1+ y_k^{-1})^{[b_{ki}]_+}}
 & i \neq k,
 \end{cases}
 \end{align}
where $[a]_+=a$ for $a>0$ and $0$ otherwise.
The relation \eqref{eq:yrel} is called the {\em exchange relation\/}
for coefficients ($y$-variables).

Following the convention of \cite{Fomin08,Fomin08b},
we identify a triangulation of a polygon with
a skew-symmetric matrix $B$ in the following way:
each diagonal is identified with a label $i$ of $B$,
and if two diagonals $i$ and $j$ share
a common triangle,
and $j$ follows $i$ anticlockwise (resp. clockwise),
then $b_{ij}=1$ (resp. $b_{ij}=-1$).
Otherwise, $b_{ij}=0$.
 For example,
for the following triangle
\medskip
\begin{align}
\label{eq:tri1}
\raisebox{-20pt}
{\begin{xy}
(15,-3)*{k},
(5.5,9.3)*{i},
(26,9)*{j},
(6,0); (4,4) **\crv{(5.3,2.3)} ?>*\dir{>};
(26,4); (24,0) **\crv{(24.7,2.3)} ?>*\dir{>}
\ar@{-} (0,0);(30,0)
\ar@{-} (30,0);(15,15)
\ar@{-} (15,15);(0,0)
\end{xy}
}
\end{align}
we have $b_{ki}=1$ and $b_{kj}=-1$.

By the mutation at $k$,
the diagonal $k$ {\em flips} as follows.
\begin{align}
\label{eq:tri2}
\raisebox{0pt}
{
\begin{xy}
(10,-3)*{k},
(47.5,0)*{k},
%
\ar@{-} (0,0);(20,0)
\ar@{-} (20,0);(10,10)
\ar@{-} (10,10);(0,0)
\ar@{-} (10,-10);(0,0)
\ar@{-} (20,0);(10,-10)
\ar@{<->} (25,0);(35,0)
\ar@{-} (50,10);(50,-10)
\ar@{-} (60,0);(50,10)
\ar@{-} (50,10);(40,0)
\ar@{-} (50,-10);(40,0)
\ar@{-} (60,0);(50,-10)
\end{xy}
}
\end{align}
Meanwhile,
the $y$-variables
with the indices $k$, $i$, $j$ for the triangle
  \eqref{eq:tri1} mutate as
\begin{align}
\label{eq:yex}
y'_k=y_k^{-1},
\quad
y'_i=y_i (1+y^{-1}_k)^{-1},
\quad
y'_j = y_j (1+y_k).
\end{align}
We only need the rules \eqref{eq:tri2} and
\eqref{eq:yex} to formulate the RSG $Y$-systems.

\subsection{First generation: $F=1$, $n_1=6$}
\label{subsec:RSG6}
We start from the RSG $Y$-system $\mathbb{Y}_{\mathrm{RSG}}(6)$
corresponding to the input data
 $F=1$, $(n_1)=(6)$,
which is the ``first generation'' of $\mathbb{Y}_{\mathrm{RSG}}(6,4,3)$.
The only relations in $\mathbb{Y}_{\mathrm{RSG}}(6)$
are those in
\eqref{eq:RSG1},
which
are explicitly written as follows:
\begin{align}
\label{eq:Y7}
\begin{split}
Y^{(1)}_1(u-1)Y^{(1)}_1(u+1)
&=1+Y^{(1)}_2(u),\\
Y^{(1)}_2(u-1)Y^{(1)}_2(u+1)
&=(1+Y^{(1)}_1(u))(1+Y^{(1)}_3(u)),\\
Y^{(1)}_3(u-1)Y^{(1)}_3(u+1)
&=(1+Y^{(1)}_2(u))(1+Y^{(1)}_4(u)),\\
Y^{(1)}_4(u-1)Y^{(1)}_4(u+1)
&=1+Y^{(1)}_3(u).
\end{split}
\end{align}
As already mentioned in Section \ref{subsec:RSGY},
this is the well known $Y$-system of type $A_4$.
The underlying cluster algebra is of type $A_4$,
and it is realized by a $7$-gon.
Note that $7=r$ for $(n_1)=6$.

\begin{figure}
	\begin{center}
		\includegraphics[scale=.5]{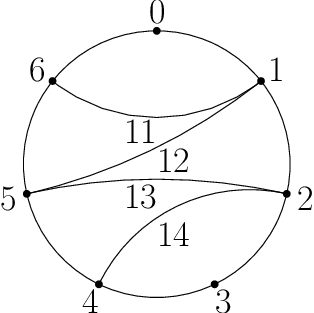}
	\end{center}
\caption{The initial triangulation $\Gamma_{\mathrm{RSG}}(6)$ of a 7-gon.}
\label{fig:7gon}
\end{figure}

Let us explain more explicitly how our $Y$-system
$\mathbb{Y}_{\mathrm{RSG}}(6)$ is realized by
a {\em sequence of triangulations\/} of a 7-gon.
To start, we take a specific triangulation of a 7-gon
$\Gamma_{\mathrm{RSG}}(6)$
as in Figure \ref{fig:7gon}.
Furthermore, we label diagonals with indices $(1,1)$, $(1,2)$, $(1,3)$, $(1,4)$
($11$, $12$, $13$, $14$ for short in Figure \ref{fig:7gon}), which
naturally 
correspond to the indices $(a,m)$ of $Y^{(a)}_m(u)$.
For the reason that will be apparent soon,
we also attach the sign $+$ to $(1,1)$ and $(1,3)$,
 and $-$ to $(1,2)$ and $(1,4)$,
 though they are not part of the labels.
The labeled triangulation $\Gamma_{\mathrm{RSG}}(6)$
serves as
 the {\em initial (labeled) seed\/} of a cluster algebra 
 of type $A_4$ with coefficients.
 Under the convention in Section \ref{subsec:exchange},
 the triangulation corresponds to the following alternating quiver,
 where we write an arrow form $i$ to $j$ if $b_{ij}=1$.
\medskip
$$
\begin{xy}
(0,0)*\cir<2pt>{},
(15,0)*\cir<2pt>{},
(30,0)*\cir<2pt>{},
(45,0)*\cir<2pt>{},
(0,-5)*{(1,1)^+},
(15,-5)*{(1,2)^-},
(30,-5)*{(1,3)^+},
(45,-5)*{(1,4)^-},
\ar@{->} (1,0);(14,0)
\ar@{<-} (16,0);(29,0)
\ar@{->} (31,0);(44,0)
\end{xy}
$$
\medskip
Starting from the {\em initial labeled
 triangulation\/} $\Gamma(0):=\Gamma_{\mathrm{RSG}}(6)$ at ``time'' $u=0$,
flip the diagonals with sign $-$ to obtain a new labeled triangulation
$\Gamma(1)$ at time $u=1$, then
flip the diagonals with sign $+$ to obtain a new labeled triangulation
$\Gamma(2)$ at time $u=2$.
By repeating this procedure, and also by doing it backward,
 we obtain a  sequence of labeled triangulations,
\begin{align}
\label{eq:7seq}
\cdots
\buildrel - \over{\longleftrightarrow}
\Gamma(-1)
\buildrel + \over{\longleftrightarrow}
\Gamma(0)
\buildrel - \over{\longleftrightarrow}
\Gamma(1)
\buildrel + \over{\longleftrightarrow}
\Gamma(2)
\buildrel - \over{\longleftrightarrow}
\cdots,
\end{align}
which is illustrated in Figure \ref{fig:7seq}.
The labels mutated at time $u$ 
in the forward direction $u+1$ 
in \eqref{eq:7seq}
are called the
{\em forward mutation points at $u$}, and
the corresponding diagonals are  marked by a circle
in  Figure \ref{fig:7seq}.

\begin{figure}
	\begin{center}
		\includegraphics[scale=.47]{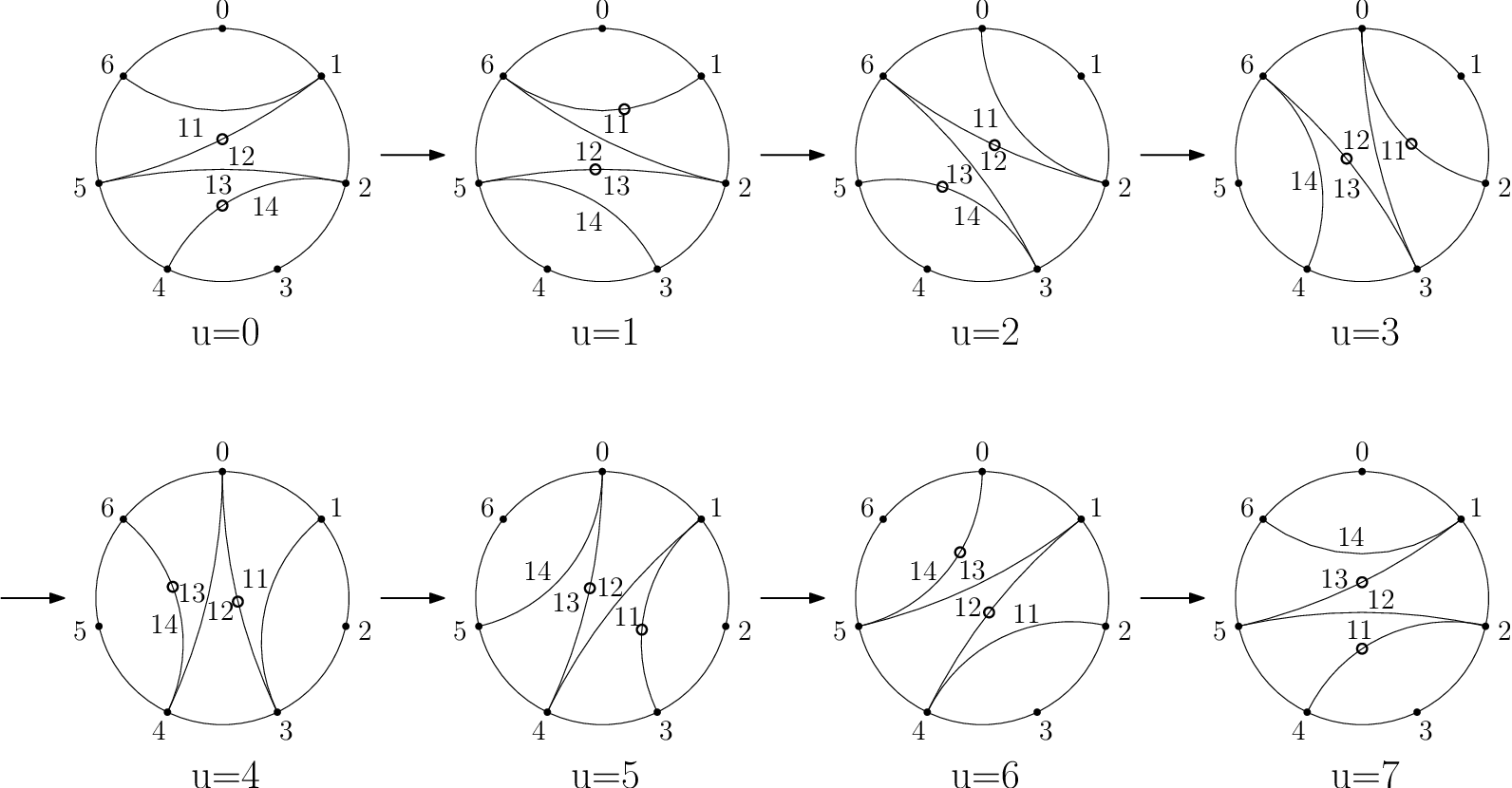}
	\end{center}
\caption{The mutation sequence 
\eqref{eq:7seq} at $u=0,\dots,7$.}
\label{fig:7seq}
\end{figure}

We can easily observe the following  facts in
Figure \ref{fig:7seq}.

{\bf Fact 1. Rotation of triangulations.}
As a labeled triangulation, we have
\begin{align}
\label{eq:period7}
\Gamma(u+2)
=
\Sigma(\Gamma(u)),
\end{align}
where $\Sigma$ denotes the   clockwise rotation of a labeled triangulation
of a polygon by one unit.

{\bf Fact 2. Realization of $Y$-system.}
The property \eqref{eq:period7},
in particular,  implies the periodicity
$B(u+2)=B(u)$ of the corresponding
exchange matrices.
One can generally associate a $Y$-system
to such a  periodicity of exchange matrices.
See \cite{Nakanishi10c} for a general procedure.
Here we explain it along the current example.
To each diagonal with index $(1,m)$ of a labeled triangulation $\Gamma(u)$,
we attach a coefficient ($y$-variable) of the cluster algebra in the universal
semifield,
which is naturally denoted by $y^{(1)}_m(u)$.
They mutate by the rule 
\eqref{eq:yrel}.
For example, we have
\begin{align}
\label{eq:y13}
y^{(1)}_1(2)&=y^{(1)}_1(1)^{-1},\\
\label{eq:y14}
y^{(1)}_1(3)&=y^{(1)}_1(2)(1+y^{(1)}_2(2)),\\
\label{eq:y11}
y^{(1)}_2(1)&=y^{(1)}_2(0)^{-1},\\
\label{eq:y12}
y^{(1)}_2(2)&=y^{(1)}_2(1)(1+y^{(1)}_1(1))(1+y^{(1)}_3(1)),
\end{align}
and so on.
Here is an important point:
at each time $u$,
we identify these $y$-variables $y^{(1)}_m(u)$ 
with the $Y$-variables $Y^{(1)}_m(u)$ of our $Y$-system
{\em only at the forward mutation points.}
For example,
at $u=0$ we have $Y^{(1)}_2(0)=y^{(1)}_2(0)$, $Y^{(1)}_4(0)=y^{(1)}_4(0)$,
and at $u=1$  we have $Y^{(1)}_1(1)=y^{(1)}_1(1)$,
$Y^{(1)}_3(1)=y^{(1)}_3(1)$, and so forth.
Then, the product of  \eqref{eq:y11} and \eqref{eq:y12},
and the product of  \eqref{eq:y13} and \eqref{eq:y14}, respectively,
yield the relations
\begin{align}
Y^{(1)}_1(1)Y^{(1)}_1(3)&=(1+Y^{(1)}_2(2)),\\
Y^{(1)}_2(0)Y^{(1)}_2(2)&=(1+Y^{(1)}_1(1))(1+Y^{(1)}_3(1)),
\end{align}
which agree with  our $Y$-system  \eqref{eq:Y7}.

To be more precise,
in the above procedure,
only the variables $Y^{(1)}_m(u)$
in $\mathcal{Y}_+$ appear.
Therefore, we realize the $Y$-system {\em for\/}
$\mathcal{Y}_+$ in the sense of Proposition
\ref{prop:bi1}.
Later we will see that this is a general phenomenon.

{\bf Fact 3. Periodicity of  $Y$-system.}
This is an immediate corollary of Facts 1 and 2.
Since $\Gamma(0)$ is a 7-gon, it follows from \eqref{eq:period7} that,
as a labeled triangulation,
\begin{align}
\label{eq:tp1}
\Gamma(14)=\Sigma^7(\Gamma(0))=\Gamma(0).
\end{align}

According to  \cite{Fomin08},
the labeled triangulations of an $n$-gon
bijectively parametrize
the (labeled) seeds of the cluster algebra of type $A_{n-3}$
 with any coefficients.
Thus,  
the periodicity  \eqref{eq:tp1} directly implies the 
periodicity of {\em seeds\/}
and, in particular,
the  periodicity of {\em $y$-variables}.
Thus, we have the  periodicity of {\em $Y$-variables}.

Alternatively, due to the results of \cite{Fock05,Fomin08b},
the labeled triangulation
completely determines
the attached {\em principal coefficients} 
(equivalently, the $c$-vectors, or the tropical $y$-variables)
of \cite{Fomin07}.
Thus,  
the periodicity  \eqref{eq:tp1}
implies the periodicity for {\em principal coefficients}.
Then, according to the tropicalization/categorification method developed
by \cite{Inoue10a,Plamondon10b},
the latter implies the  periodicity of {\em seeds}.
Therefore,
the periodicity of {\em $Y$-variables\/} follows again.

In either way, this construction proves the periodicity of the $Y$-system
with the desired period $14=2r$
in Theorem \ref{thm:periodRSG}.
(By the same reason, the half periodicity $7=r$
mentioned in Remark \ref{rem:period1} (b)
also follows from the half periodicity of the labeled triangulations
observed in Figure \ref{fig:7seq}.)

In summary,
to show periodicity for this $Y$-system, 
it is enough to realize it by a polygon,
and the rest is {\em automatic}.
We will apply this strategy
 to prove Theorem \ref{thm:periodRSG} in full generality.
 We remark that in the case $F=1$
 the connection between the $Y$-systems (of type $A$)
 and triangulations of polygons already appeared in
 \cite{Fomin03b}.

\subsection{Second generation: $F=2$, $(n_1,n_2)=(6,4)$}
\label{subsec:RSG64}
We turn to the RSG $Y$-system $\mathbb{Y}_{\mathrm{RSG}}(6,4)$
corresponding to the input data
 $F=2$, $(n_1,n_2)=(6,4)$.

The relations in $\mathbb{Y}_{\mathrm{RSG}}(6,4)$
are explicitly written as follows.
The relations  in \eqref{eq:Y7} hold except for the
last one, which is now
replaced with
\begin{align}
Y^{(1)}_4(u-1)Y^{(1)}_4(u+1)
&=(1+Y^{(1)}_3(u))(1+Y^{(2)}_1(u)^{-1})^{-1}.
\end{align}
Besides, the following four relations are added.
\begin{align}
\label{eq:Y74}
\begin{split}
Y^{(2)}_1(u-6)&Y^{(2)}_1(u+6)
=(1+Y^{(2)}_2(u)^{-1})^{-1}(1+Y^{(1)}_1(u))\\
&\quad
\times
(1+Y^{(1)}_4(u-5)^{-1})^{-1}
(1+Y^{(1)}_3(u-4)^{-1})^{-1}\\
&\quad
\times
(1+Y^{(1)}_2(u-3)^{-1})^{-1}
(1+Y^{(1)}_1(u-2)^{-1})^{-1}\\
&\quad
\times
(1+Y^{(1)}_1(u+2)^{-1})^{-1}
(1+Y^{(1)}_2(u+3)^{-1})^{-1}\\
&\quad
\times
(1+Y^{(1)}_3(u+4)^{-1})^{-1}
(1+Y^{(1)}_4(u+5)^{-1})^{-1},
\end{split}
\end{align}
and
\begin{align}
\label{eq:Y74b}
\begin{split}
Y^{(2)}_2(u-6)Y^{(2)}_2(u+6)
&=(1+Y^{(2)}_1(u)^{-1})^{-1}(1+Y^{(2)}_3(u)^{-1})^{-1},\\
Y^{(2)}_3(u-6)Y^{(2)}_3(u+6)
&=(1+Y^{(2)}_2(u)^{-1})^{-1}(1+Y^{(2)}_4(u)^{-1})^{-1},\\
Y^{(2)}_4(u-6)Y^{(2)}_4(u+6)
&=(1+Y^{(2)}_3(u)^{-1})^{-1}.
\end{split}
\end{align}
The shift of the parameter
$u$ by `6' in their left hand sides
and the complicated structure of the right hand side of
\eqref{eq:Y74} are the  mystery of this $Y$-system.

We formulate it
by using the
realization of a cluster algebra of type $A$ by a 31-gon,
where $31=r$ for $(n_1,n_2)=(6,4)$. See \eqref{eq:6-25}.
Actually, this  is the example whose cluster algebraic
formulation (without polygon realization) was presented in detail in
\cite[Section 7]{Nakanishi10b}
(the case $(n_1,n_2)=(4,7)$ therein).
So, we only need to translate it into the polygon language.
Like the previous example,
the formulation consists of two ingredients,
(i)  the initial labeled triangulation of 31-gon,
and
(ii) the mutation sequence of labeled triangulations.

\begin{figure}
		\includegraphics[scale=.39]{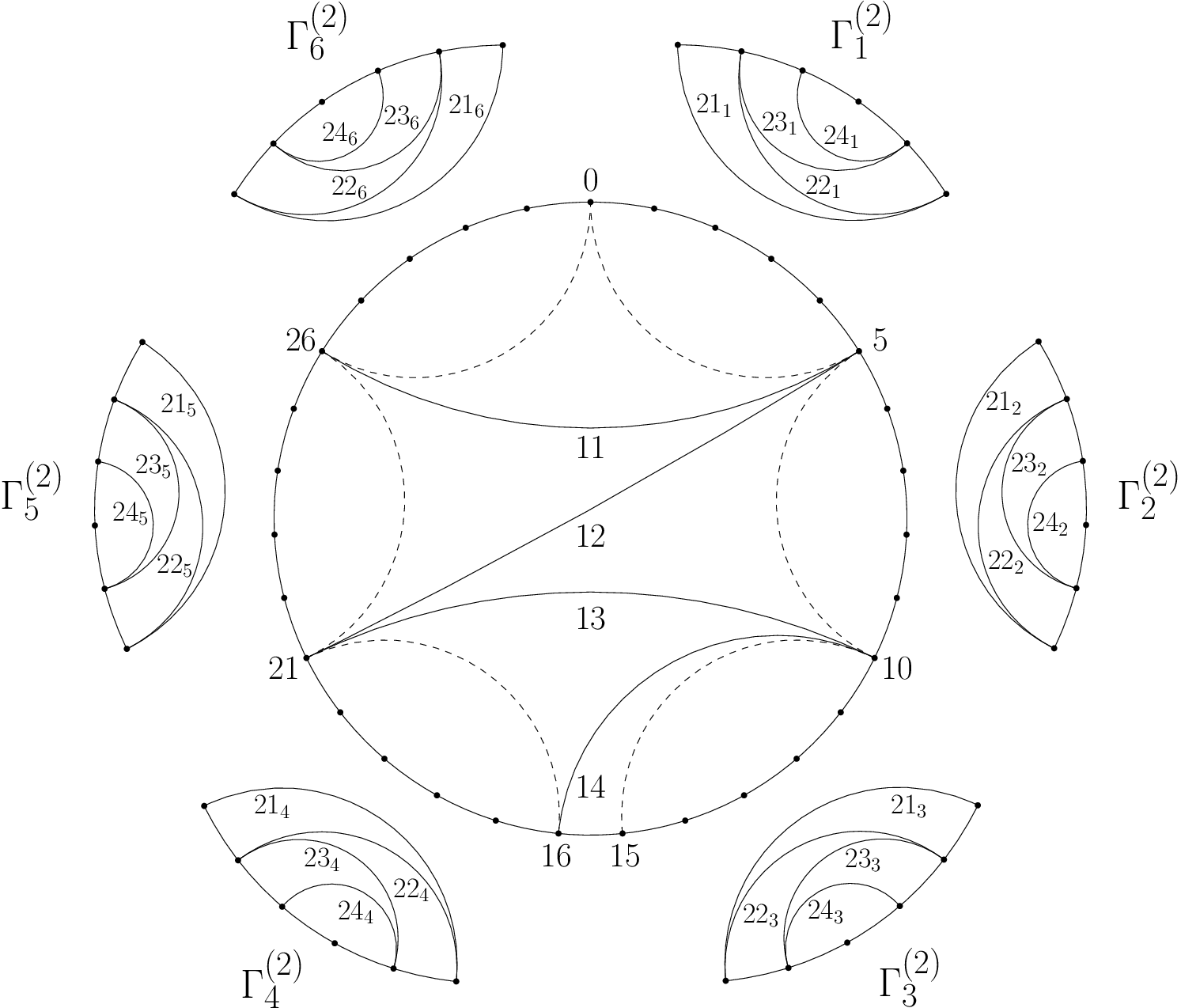}
\caption{Construction 
of the initial  triangulation $\Gamma_{\mathrm{RSG}}(6,4)$
of a 31-gon.}
\label{fig:31gon}
\end{figure}

{\bf (i).  Initial labeled triangulation
$\Gamma_{\mathrm{RSG}}(6,4)$ of 31-gon.}
We construct the labeled triangulation 
$\Gamma_{\mathrm{RSG}}(6,4)$ 
 from
 $\Gamma_{\mathrm{RSG}}(6)$.
To start, to each edge of  $\Gamma_{\mathrm{RSG}}(6)$
in Figure \ref{fig:7gon},
except for the edge 3-4,
we add 4 vertices,
and make $\Gamma_{\mathrm{RSG}}(6)$
into 31-gon as in
Figure \ref{fig:31gon}.
Then, we ``paste''
triangulated 6-gons
 $\Gamma^{(2)}_{s}$ ($s=1,\dots,6$)
as in
Figure \ref{fig:31gon}
and obtain a triangulated  31-gon.
Note that $\Gamma^{(2)}_{1}$,
$\Gamma^{(2)}_{2}$,  $\Gamma^{(2)}_{3}$
have  the same shape,
while $\Gamma^{(2)}_{4}$, $\Gamma^{(2)}_{5}$,  $\Gamma^{(2)}_{6}$
are their mirror images.

The labels $(1,m)$  of the diagonals
of the first generation are carried over
to $\Gamma_{\mathrm{RSG}}(6,4)$.
The  diagonals
of the second generation
coming from $\Gamma^{(2)}_{s}$
are labeled
 (with extra signs) as
 $(2,1)_s^+$, $(2,2)_s^-$, $(2,3)_s^+$,
$(2,4)_s^-$
($21_s$, $22_s$, $23_s$,
$24_s$ for short in Figure \ref{fig:31gon}),
starting from the inside of the 31-gon.
The result is the  
initial  labeled triangulation $\Gamma_{\mathrm{RSG}}(6,4)$
presented as the first diagram in Figure \ref{fig:31gon-seq},
where a part is shaded to help recognizing
the irregularity of the pattern of the triangulation.

 {\bf (ii). Mutation sequence of labeled triangulations.}
 \begin{figure}
	\begin{center}
		\includegraphics[scale=.42]{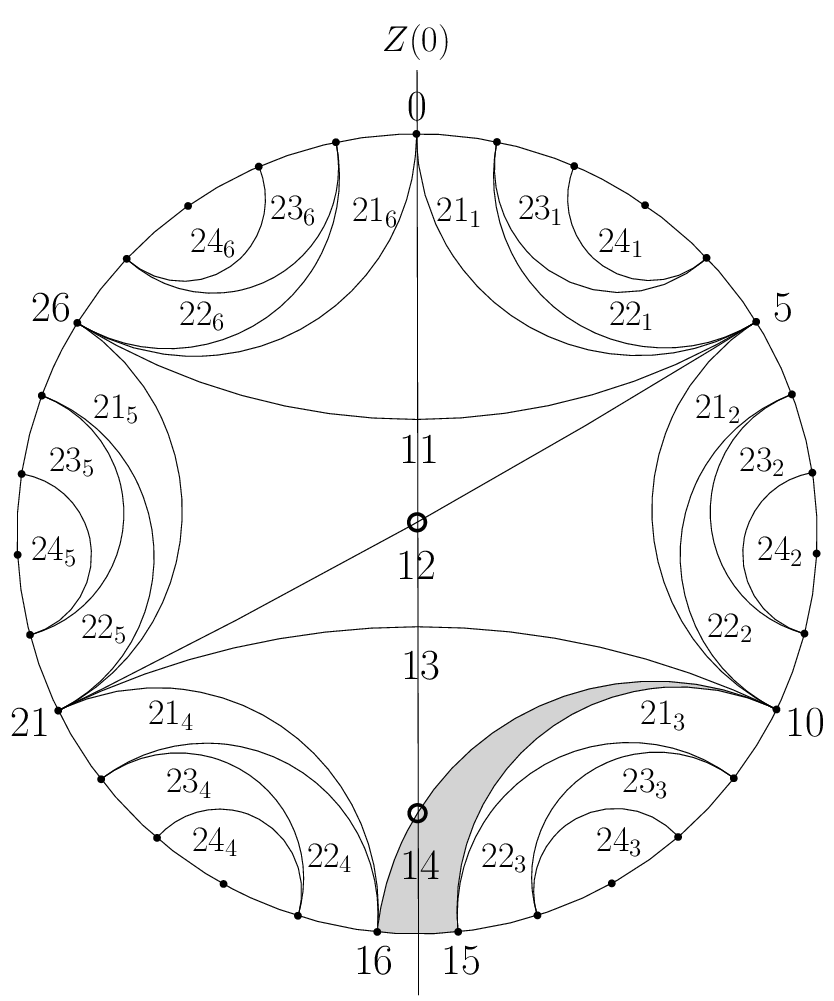}
		\hskip15pt
		\includegraphics[scale=.42]{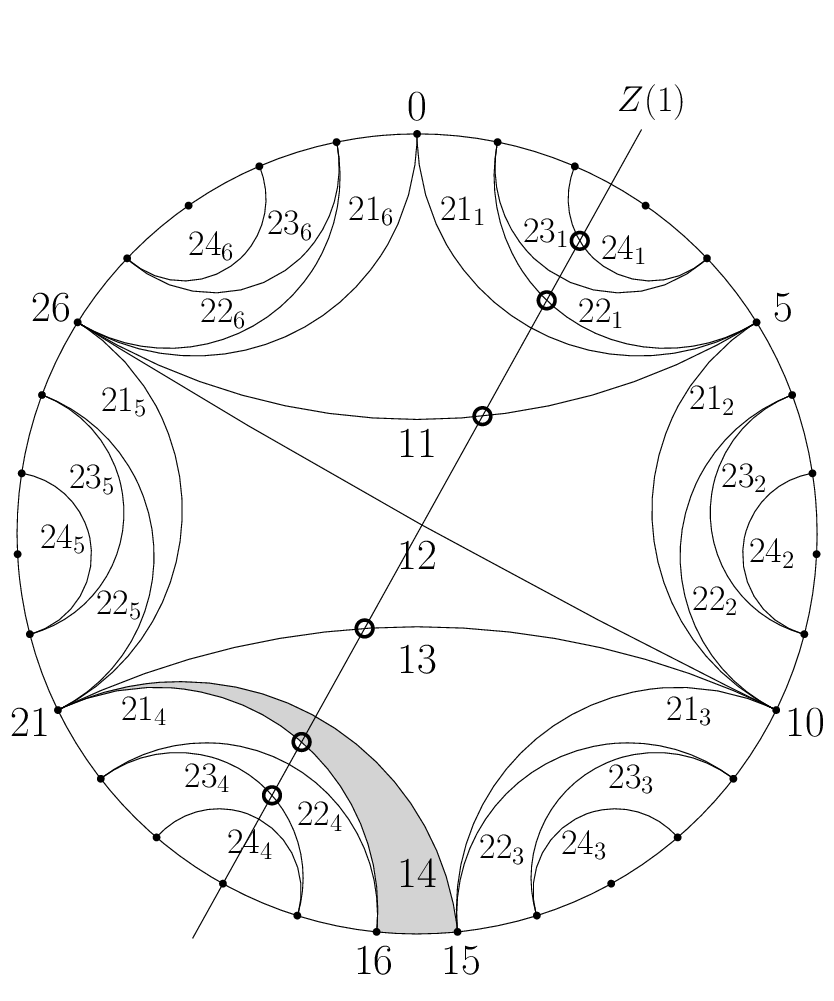}
	\end{center}
\centerline{$ u=0$ \hskip160pt $u=1$}
\caption{The mutation sequence  \eqref{eq:74mseq}
at $u=0,1$.}
\label{fig:31gon-seq}
\end{figure}
As  before,
$\Gamma(0):=\Gamma_{\mathrm{RSG}}(6,4)$
serves as
 the {\em initial (labeled) seed\/} of the cluster algebra 
 of type $A_{28}$ with coefficients 
 in the universal semifield.
Consider the sequence of mutations,
\begin{align}
\label{eq:74mseq}
\cdots
\buildrel {1^+ 2_6^{-} 2_3^{+}}  \over{\longleftrightarrow}
\Gamma(0)
\buildrel {1^{-}} \over{\longleftrightarrow}
\Gamma(1)
\buildrel {1^+ 2_1^{-} 2_4^{+}} \over{\longleftrightarrow}
\Gamma(2)
\buildrel {1^-} \over{\longleftrightarrow}
\Gamma(3)
\buildrel  {1^+ 2_2^{-} 2_5^{+}} \over{\longleftrightarrow}
\cdots,
\end{align}
where the mutation points are repeating  modulo $u=12=2p_2$,
and $ {1^+ 2_1^{-} 2_4^{+}}$, for example, stands for the composite
mutation at
\begin{align}
 (1,1)^{+}, (1,3)^{+}; (2,2)_1^{-}, (2,4)_1^{-};
(2,1)_4^{+},
(2,3)_4^{+}.
\end{align}
At any time $u$,
the composite mutation is well defined, i.e., it does not depend on
the order, since any pair of diagonals to be mutated
does not share a common triangle.
See Figure \ref{fig:31gon-seq} for the illustration
of the sequence \eqref{eq:74mseq} at $u=0$ and $1$,
where the diagonals corresponding to the forward mutation
points are marked by a circle.

{}From Figure \ref{fig:31gon-seq} we observe
 all   desired properties as follows.

{\bf Fact 1. Reflection/rotation of triangulations.}
As {\em unlabeled\/} triangulations,
$ \Gamma(2)$ and  $\Sigma^5(\Gamma(0))$ coincide,
which we write as
\begin{align}
\label{eq:period47un}
 \Gamma(2)\sim\Sigma^5(\Gamma(0)).
 \end{align}
Just observing this fact is  the  first
and probably the most important 
step
 of the whole analysis
  in the paper.
So, let us examine more closely in Figure \ref{fig:31gon-seq}
how this rotation  happens.
Let $Z(0)$ and $Z(1)$ be the axes 
in Figure \ref{fig:31gon-seq}.
Then, it is easy to recognize that
the forward mutations at $u=0$ and $1$ are
nothing but the {\em reflections\/} of diagonals with respect to the axes
$Z(0)$ and $Z(1)$, respectively.
Thus, 
the composition of two reflections results into the
rotation $\Sigma^5$ of \eqref{eq:period47un}.
Later we will see that the number $5$ 
here is $5=r^{(2)}$ for $(n_1,n_2)=(6,4)$.

On the other hand, 
as  {\em labeled\/} triangulations,
$\Gamma(2)$ and  $\Sigma^5(\Gamma(0))$ do not coincide,
since the mutations at $u=0$ and $1$ do
{\em not\/} act as reflections on labels.
However, they coincide
{\em up to the relabeling of diagonals of the second generation}.
 Let  $\nu$ be the permutation
 of the labels of
 the triangulations $\Gamma(u)$ defined by
 \begin{align}
 \label{eq:nu1}
 \begin{split}
 \nu:\ &(1,m) \mapsto (1,m),
 \quad
 (2,m)_s\mapsto (2,m)_{s+1},
 \end{split}
 \end{align}
 where the subscript $s$ is regarded modulo $6$.
 Let $\nu$ also denote  the {\em relabeling\/} of $\Gamma(u)$ by $\nu$,
 namely,
it  replaces the label $(a,m)$ attached to each diagonal with
 $\nu((a,m))$.
Then, as labeled triangulations, we have 
 \begin{align}
 \label{eq:period74}
 \Gamma(2)=\Sigma^5(\nu(\Gamma(0))).
 \end{align}
 
Similarly, the mutation at $u=2$ is the reflection of $\Gamma(2)$
with respect to the axis $Z(2)=\Sigma^5(Z(0))$,
where  $\Sigma$ also denotes the rotation of an axis by one unit
around the center of the polygon.
Then, by the same argument, we have
 $\Gamma(3)=\Sigma^5(\nu(\Gamma(1)))$,
and, more generally, for any $u\in \mathbb{Z}$.
 \begin{align}
 \label{eq:period74-3}
 \Gamma(u+2)=\Sigma^5(\nu(\Gamma(u))).
 \end{align}

{\bf Fact 2. Realization of $Y$-system.}
The property \eqref{eq:period74-3},
in particular,  implies the {\em partial\/} periodicity of the corresponding
exchange matrices $B(u+2)=\nu(B(u))$,
up to the  relabeling of $B(u)$ by \eqref{eq:nu1}
 ($\nu$-periodicity in \cite{Nakanishi10c}).
One can still associate a $Y$-system
to such a partial periodicity 
of exchange matrices.
As before,
the coefficient attached to the
diagonal with the label $(1,m)$ of the first generation
at time $u$
is denoted by $y^{(1)}_m(u)$.
Similarly, the coefficient attached to the
diagonal with the label $(2,m)_s$ of the second generation
at time $u$
is denoted by $y^{(2)}_{m,s}(u)$.
Note that the
variables $y^{(2)}_{m,s}(u)$ have some redundancy (by $s=1,\dots,6$)
compared with our target $Y$-variables $Y^{(2)}_{m}(u)$.

Like $\mathbb{Y}_{\mathrm{RSG}}(6)$,
at each time $u$
we identify  the $y$-variables of the {\em first\/} generation
 $y^{(1)}_m(u)$ 
with the $Y$-variables $Y^{(1)}_m(u)$
only at forward mutation points.
Here is another important point:
at each time $u$
we identify the $y$-variables of the {\em second\/} generation
 $y^{(2)}_{m,s}(u)$
with the $Y$-variables $Y^{(2)}_m(u)$
only at forward mutation points,
{\em regardless of $s$}.
Note that, in the mutation sequence  \eqref{eq:74mseq},
there is no simultaneous mutation
at $(2,m)_s$ and $(2,m)_{s'}$ with $s\neq s'$,
so that
the above identification does not create any conflict.

 \begin{figure}
	\begin{center}
		\includegraphics[scale=.42]{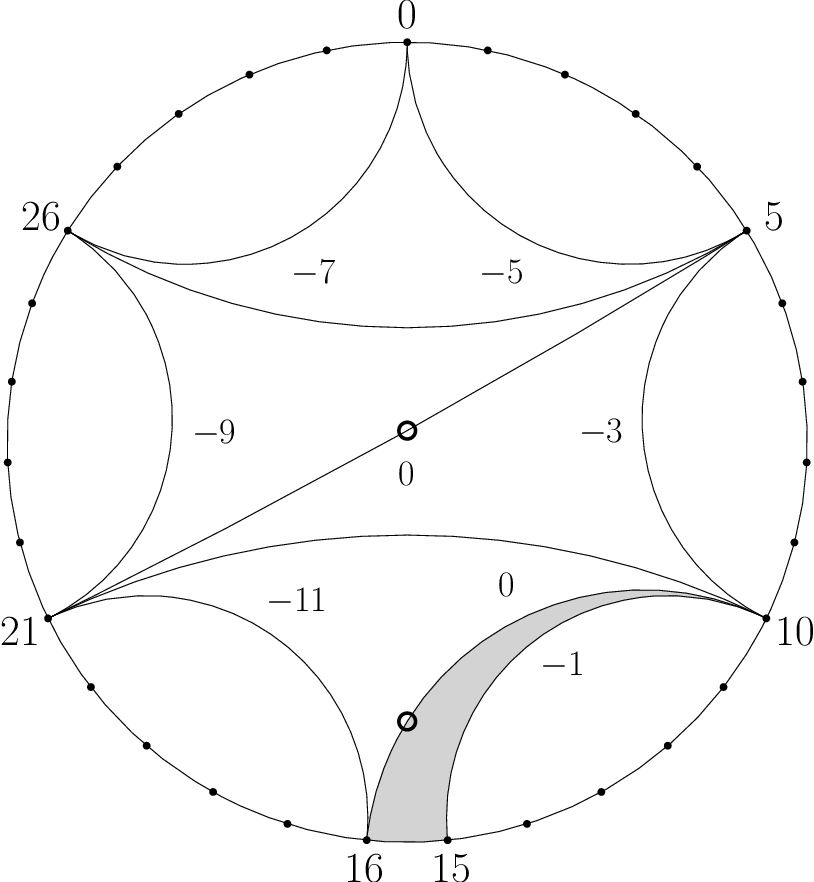}
		\hskip20pt
		\includegraphics[scale=.42]{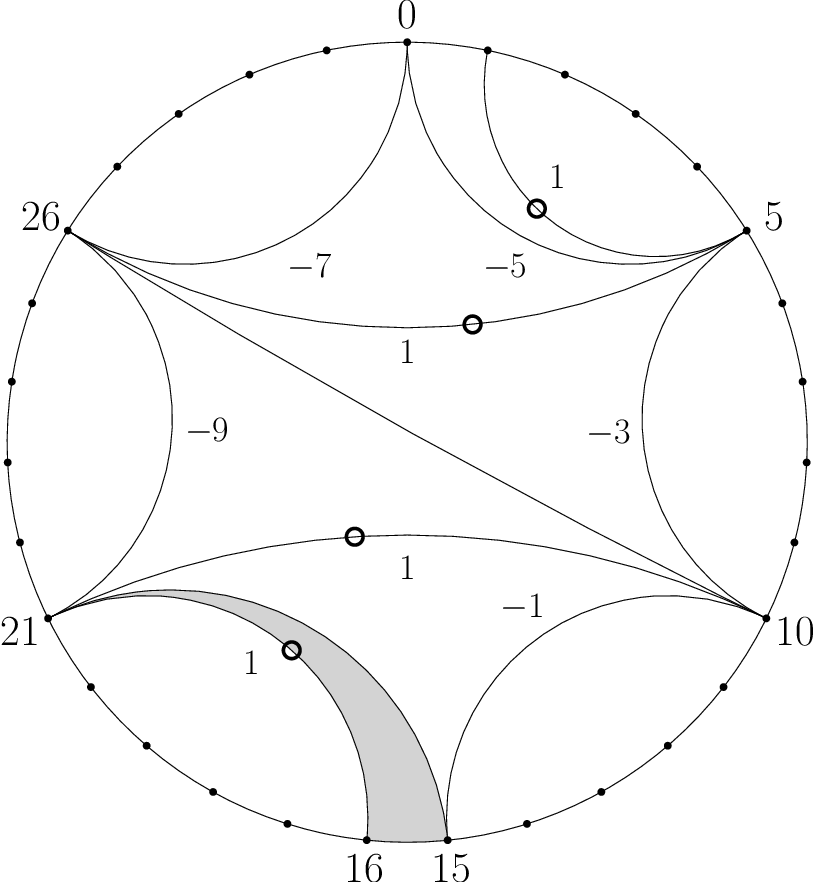}
	\end{center}
\centerline{$ u=0$ \hskip160pt $u=1$}
\caption{Snapshots at $u=0,1$ for the relation
\eqref{eq:Y74}.}
\label{fig:31gon-snap}
\end{figure}

With this identification, we claim
that the mutation sequence \eqref{eq:74mseq}
realizes the $Y$-system $\mathbb{Y}_{\mathrm{RSG}}(6,4)$.
Let us concentrate on the  relation \eqref{eq:Y74},
which is the most mysterious one.
To understand it,
it is useful to take
 a {\em snapshot\/} of the mutation sequence
 \eqref{eq:74mseq}
 at time
 $u$,
 that is,
the list of the {\em time of the most recent mutation\/}
 of  each diagonal as of $u$.
To be more precise, to each diagonal we attach the
integer  $u'$ which is the maximal one such that
the diagonal was mutated at time $u'\leq u$.
For example,
the snapshots at $u=0$ and $1$ are presented in Figure \ref{fig:31gon-snap}.
(For simplicity, we write only  the data  relevant 
to the  relation  \eqref{eq:Y74}.)
In the second diagram in Figure \ref{fig:31gon-snap},
for example, we have four 1's,
which are the forward mutation points at time $u=1$
and  correspond to
 $Y^{(1)}_1(1)$,
$Y^{(1)}_3(1)$,
$Y^{(2)}_1(1)$,
$Y^{(2)}_2(1)$.
Also, we have
$-1$, $-3$, $-5$, $-7$, 
$-9$,
which indicate that
$Y^{(2)}_1(-1)$,
$Y^{(2)}_1(-3)$,
$Y^{(2)}_1(-5)$,
$Y^{(2)}_1(-7)$,
$Y^{(2)}_1(-9)$
were attached there ``in the past''
in the sequence \eqref{eq:74mseq}.
We know this because of the rotation property
\eqref{eq:period47un}.
Similarly,
in the first diagram, we have two $0$'s
at 
 the forward mutation points at time $u=0$
 corresponding to $Y^{(1)}_2(0)$,
$Y^{(1)}_4(0)$.
The rest of data, except for $-11=1-12$, are transcribed from the one at $u=1$.
Note that the snapshot at a given time $u$ is also obtained from 
the one at $u=0$ or $u=1$ (depending on the parity of $u$)
by a total shift  of data and a rotation.

Now it is a pleasant exercise to confirm that these data precisely
produce the relation
 \eqref{eq:Y74}
 by using the exchange relation \eqref{eq:yex}.
For example,
in the snapshot of $u=1$
we see that
$Y^{(1)}_1(1)$ contributes
to the mutations of $Y^{(2)}_1(-7)$,
$Y^{(2)}_1(-5)$,
$Y^{(2)}_1(-3)$
as the multiplicative factors
$
(1+Y^{(1)}_1(1)^{-1})^{-1}
$,
$
1+Y^{(1)}_1(1)
$,
$
(1+Y^{(1)}_1(1)^{-1})^{-1}
$,
respectively.
This means that
during the mutation of $Y^{(2)}_1(u-6)$
to $Y^{(2)}_1(u+6)$, 
the factors
$
(1+Y^{(1)}_1(u-2)^{-1})^{-1}
$,
$
1+Y^{(1)}_1(u)
$,
$
(1+Y^{(1)}_1(u+2)^{-1})^{-1}
$
contribute, matching   \eqref{eq:Y74}.

In a similar and easier way,
using the full snapshots at $u=0$ and $1$,
one can verify the remaining relations of the $Y$-system.
Thus, the snapshots at $u=0$ and $1$ 
``visualize'' the whole $Y$-system.

\par
{\bf Fact 3. Periodicity of  $Y$-system}.
Again this is an immediate corollary of Facts 1 and 2.
Since $\Gamma(0)$ is a 31-gon,
using \eqref{eq:period74} ,
 $\Sigma \nu= \nu \Sigma$, and $\nu^6=\mathrm{id}$,
 we have
$\Gamma(62)=\Sigma^{155}(\nu^{31}(\Gamma(0)))=\nu(\Gamma(0))$.
Therefore, $62=2r$  is a (full) period
 of the {\em unlabeled\/} seeds of our cluster algebra.
It is  only a {\em partial\/} period of the {\em labeled\/} seed
  up to  the relabeling by $\nu$
  ({\em $\nu$-period\/} in the sense of \cite{Nakanishi10b}).
 However, since our identification of $y$-variables with
 $Y$-variables
 ignores  the relabeling by $\nu$,
 it gives a {\em full\/} period of the $Y$-system.
Thus, it proves the desired periodicity of $\mathbb{Y}_{\mathrm{RSG}}(6,4)$
in Theorem \ref{thm:periodRSG}.
Furthermore, this period is minimal, because
 $5=r^{(2)}$ and $31=r$ are coprime (see Proposition
 \ref{prop:cf1} (e)).

The moral of this example is that all  necessary information
for Facts 1--3
is encoded in the two diagrams $\Gamma(0)$ and $\Gamma(1)$
with the marking of forward mutation points.

\subsection{Third generation: $F=3$, $(n_1,n_2,n_3)=(6,4,3)$}
Let us proceed to the full 
 RSG $Y$-system $\mathbb{Y}_{\mathrm{RSG}}(6,4,3)$.

The relations in $\mathbb{Y}_{\mathrm{RSG}}(6,4,3)$
are explicitly written as follows.
The last relation in \eqref{eq:Y74b} is now replaced with
\begin{align}
Y^{(2)}_4(u-6)Y^{(2)}_4(u+6)
&=(1+Y^{(2)}_3(u)^{-1})^{-1}(1+Y^{(3)}_2(u)).
\end{align}
Besides, the following three relations are added.
\begin{align}
\label{eq:Y743}
\begin{split}
Y^{(3)}_1(u-25)Y^{(3)}_1(u+25)
&=(1+Y^{(3)}_2(u))(1+Y^{(1)}_4(u))\\
&\quad
\times
(1+Y^{(2)}_4(u-19))
(1+Y^{(2)}_3(u-13))\\
&\quad
\times
(1+Y^{(2)}_2(u-7))
(1+Y^{(2)}_1(u-1))\\
&\quad
\times
(1+Y^{(2)}_1(u+1))
(1+Y^{(2)}_2(u+7))\\
&\quad
\times
(1+Y^{(2)}_3(u+13))
(1+Y^{(2)}_4(u+19)),
\end{split}
\end{align}
and
\begin{align}
\begin{split}
Y^{(3)}_2(u-25)Y^{(3)}_2(u+25)
&=(1+Y^{(3)}_1(u))1+Y^{(3)}_3(u)),\\
Y^{(3)}_3(u-25)Y^{(3)}_3(u+25)
&=1+Y^{(3)}_2(u).\\
\end{split}
\end{align}

We formulate it
using
the realization of a cluster algebra of type $A$
by a 106-gon,
where $106=r$ in Example \ref{ex:743}.
Since we will give a full account for a general case
in Section \ref{sec:RSG},
here we limit ourselves to exhibit some of new feature.


 \begin{figure}
		\begin{center}
		\includegraphics[scale=.42]{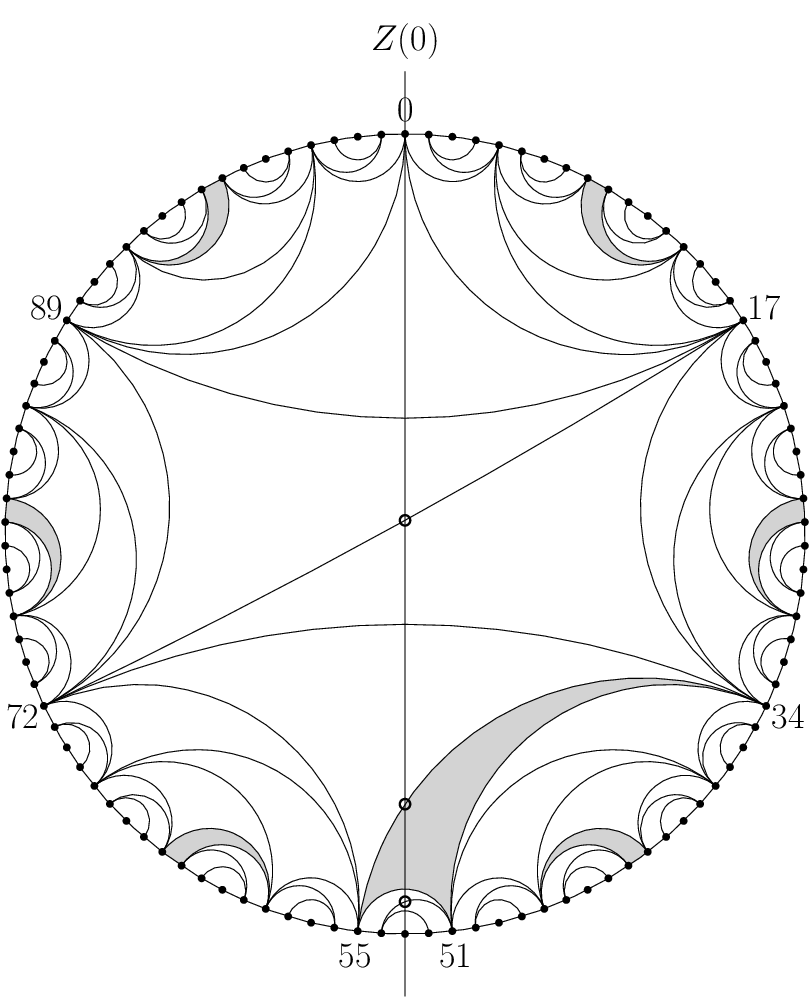}
		\hskip20pt
		\includegraphics[scale=.42]{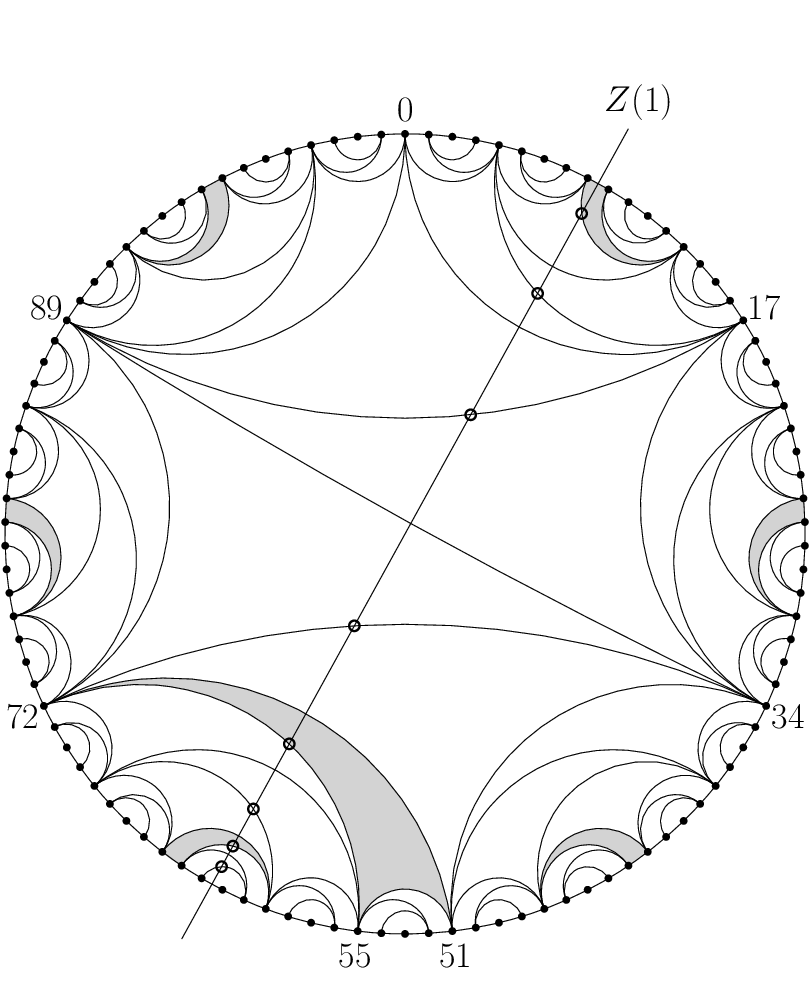}
	\end{center}
\centerline{$ u=0$ \hskip160pt $u=1$}
\caption{The mutation sequence
  \eqref{eq:743mseq}  at $u=0,1$.
 The first diagram is $\Gamma_{\mathrm{RSG}}(6,4,3)$.}
\label{fig:106gon-seq}
\end{figure}

{\bf (i).  Initial labeled triangulation
$\Gamma_{\mathrm{RSG}}(6,4,3)$ of 106-gon.}
The initial labeled triangulation
$\Gamma(0)=\Gamma_{\mathrm{RSG}}(6,4,3)$ of a 106-gon
is given in the top diagram
in Figure \ref{fig:106gon-seq},
where we omit the labels of the diagonals.
The triangulation $\Gamma_{\mathrm{RSG}}(6,4,3)$ is constructed
from $\Gamma_{\mathrm{RSG}}(6,4)$ by adding vertices
and pasting 25 triangulated 5-gons
 $\Gamma^{(3)}_s$ ($s=1,\dots,25$), according to some rule,
beginning from the top-right of the 106-gon and
proceeding clockwise.
The labels  of the earlier generations are carried over
to $\Gamma_{\mathrm{RSG}}(6,4,3)$.
The diagonals of the third generation 
coming from  $\Gamma^{(3)}_s$
are labeled (with extra signs)
as $(3,1)_s^+$, $(3,2)_s^-$, $(3,3)_s^+$ ($s=1,\dots,25$), starting from the inside
of the 106 gon.

 {\bf (ii). Mutation sequence of labeled triangulations.}
The first two steps of
the mutation sequence are given in Figure \ref{fig:106gon-seq},
where the forward mutation points are marked.
To be precise, we consider the sequence of mutations,
\begin{align}
\label{eq:743mseq}
\cdots
\buildrel {1^+ 2_6^{-} 2_3^{+} 3_{11}^{+}}  \over{\longleftrightarrow}
\Gamma(0)
\buildrel {1^{-}3_{13}^-} \over{\longleftrightarrow}
\Gamma(1)
\buildrel {1^+ 2_1^{-} 2_4^{+}3_{15}^{+}} \over{\longleftrightarrow}
\Gamma(2)
\buildrel {1^- 3_{17}^{-}} \over{\longleftrightarrow}
\Gamma(3)
\buildrel  {1^+ 2_2^{-} 2_5^{+}3_{19}^{+}} \over{\longleftrightarrow}
\cdots,
\end{align}
where
the mutation points are
repeating  modulo $u=2\times 6\times 25=2p_2p_3=300$,
Again, $ {3_{15}^{+}}$, for example,  stands for the composite
mutations at
$(3,1)_{15}^{+}$ and  $(3,3)_{15}^{+}$.
Once again, from Figure \ref{fig:106gon-seq} we can extract
 all  necessary information as follows.

{\bf Fact 1. Reflection/rotation of triangulations.}
We have an equality of unlabeled triangulations,
$\Gamma(2)\sim\Sigma^{17}(\Gamma(0))$.
Again, this happens because
the forward mutations at $u=0$ and $1$ are
 the {\em reflections\/} of diagonals with respect to the axes
$Z(0)$ and $Z(1)$ in Figure \ref{fig:106gon-seq}, respectively.
Later we will see that $17=r^{(2)}$ in Example \ref{ex:743}.
 Let us extend $\nu$ in \eqref{eq:nu1}
 to  a permutation of the labels of $\Gamma(u)$ here
 by
 $\nu:\ (3,m)_s \mapsto (3,m)_{s+4}$
 where the subscript $s$ for the third generation is regarded modulo $25$.
As labeled triangulations, we have
 $\Gamma(2)=\Sigma^{17}(\nu(\Gamma(0)))$,
and, more generally, for $u\in \mathbb{Z}$,
 $\Gamma(u+2)=\Sigma^{17}(\nu(\Gamma(u)))$.

 \begin{figure}
		
		\begin{center}
		\includegraphics[scale=.42]{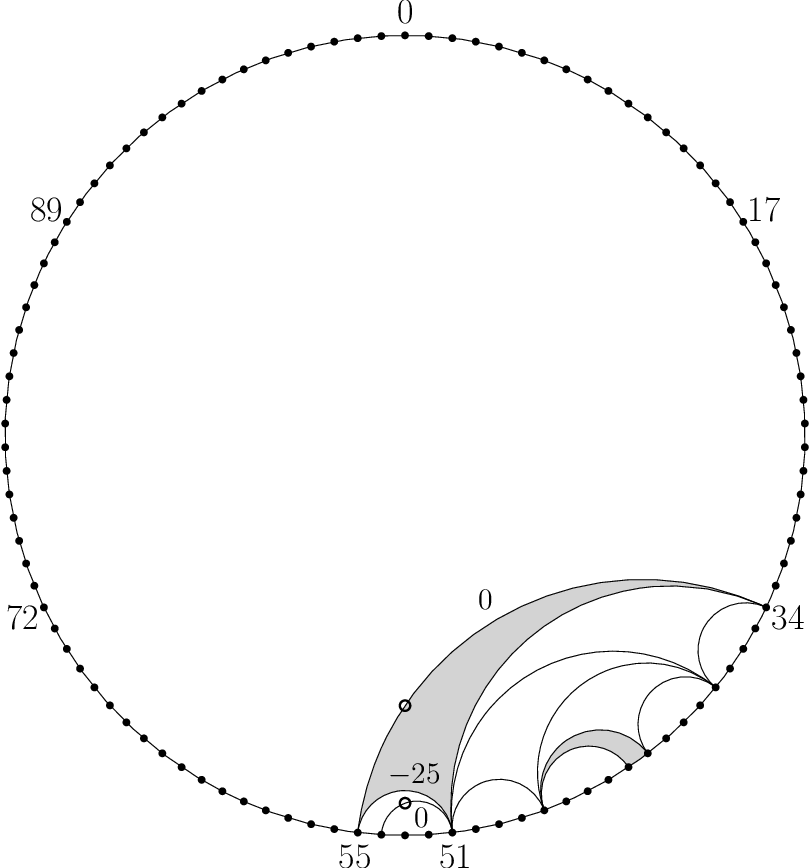}
		\hskip20pt
		\includegraphics[scale=.42]{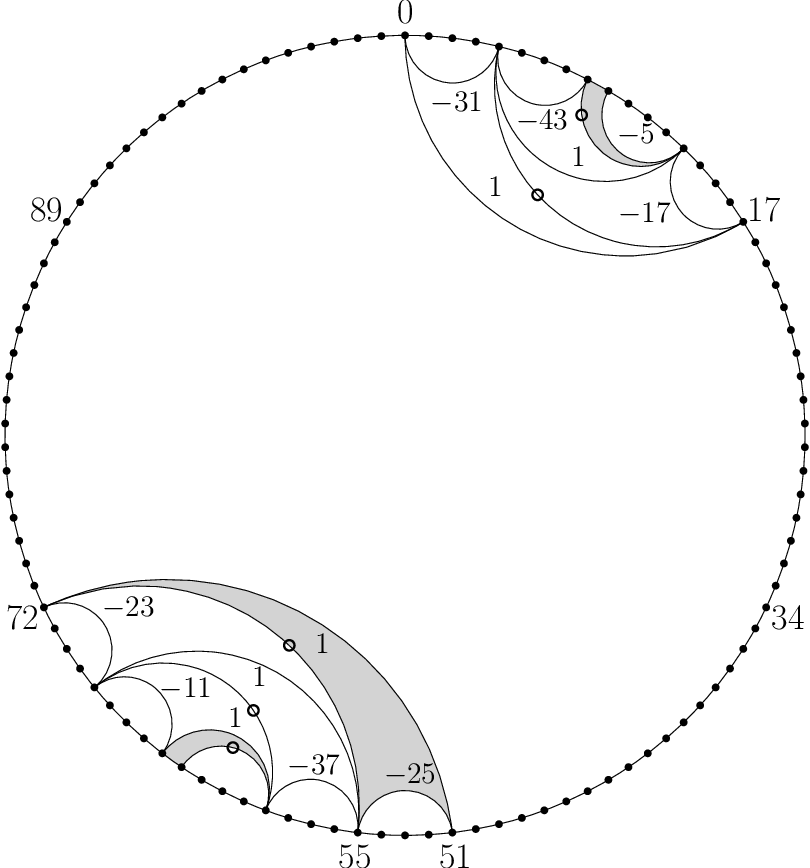}
	\end{center}
\centerline{$ u=0$ \hskip160pt $u=1$}
\caption{Snapshots  at $u=0,1$ for the relation
 \eqref{eq:Y743}.}
\label{fig:106gon-snap}
\end{figure}

{\bf Fact 2. Realization of $Y$-system.}
As before,
the $y$-variable attached to the
diagonal with the label $(3,m)_s$ 
at time $u$
is denoted by $y^{(3)}_{m,s}(u)$.
Again,
for each $u$
we identify 
variables $y^{(3)}_{m,s}(u)$
with the $Y$-variables $Y^{(3)}_m(u)$
 only at forward mutation points,
{\em regardless of $s$}.
With this identification, we claim that
the mutation sequence \eqref{eq:743mseq}
realizes  $\mathbb{Y}_{\mathrm{RSG}}(6,4,3)$.
Again, this can be checked using the snapshot method.
Let us concentrate on the  relation
 \eqref{eq:Y743},
which is the most mysterious one.
The snapshots 
of the mutation sequence \eqref{eq:743mseq}
at $u=0$ and $1$ are presented in Figure \ref{fig:106gon-snap},
where, for simplicity,  we write only the data relevant 
to the  relation  \eqref{eq:Y743}.
Again, it is a pleasant exercise
 to confirm that these data precisely produce the relation \eqref{eq:Y743}
 using the exchange relation \eqref{eq:yex}.

{\bf Fact 3. Periodicity of  $Y$-system.}
Using Facts 1 and 2 and repeating the  same argument as before,
we obtain  the desired period $212=2r$
of  $\mathbb{Y}_{\mathrm{RSG}}(6,4,3)$
in Theorem \ref{thm:periodRSG}.
Furthermore, 
it is minimal  because $17=r^{(2)}$ and $106=r$ are coprime.

\subsection{Quasi-reflection symmetry}
\label{subsec:QS1}
We observed that all information is encoded in
the two diagrams 
$\Gamma(0)$ and $\Gamma(1)$
with marking of the forward mutation points
as in Figure \ref{fig:106gon-seq}.
To work in full generality, however, it is not convenient to deal with
{\em two\/} diagrams.
Fortunately, one can unify them into {\em one\/} diagram
by introducing {\em backward mutation points}.

Let us concentrate on the case $\Gamma_{\mathrm{RSG}}(6,4,3)$.
By definition, the {\em backward mutation points at time $u$} in the sequence of mutations
 \eqref{eq:743mseq} are the forward mutation points
at time $u-1$.
In particular,
 the forward mutation points at $u=1$ are
 the backward mutation points at $u=2$.
Moreover,  
the latter are obtained from
the backward mutation points at $u=0$
by the relabeling $\nu$.
Therefore, the information of the two diagrams in
Figure \ref{fig:106gon-seq} can be packed
into a single	diagram as in Figure \ref{fig:643},
where the diagonals
for the forward (resp. backward) mutation points at $u=0$
are  marked by circles (resp. crosses).

 \begin{figure}
			\begin{center}
		\includegraphics[scale=.42]{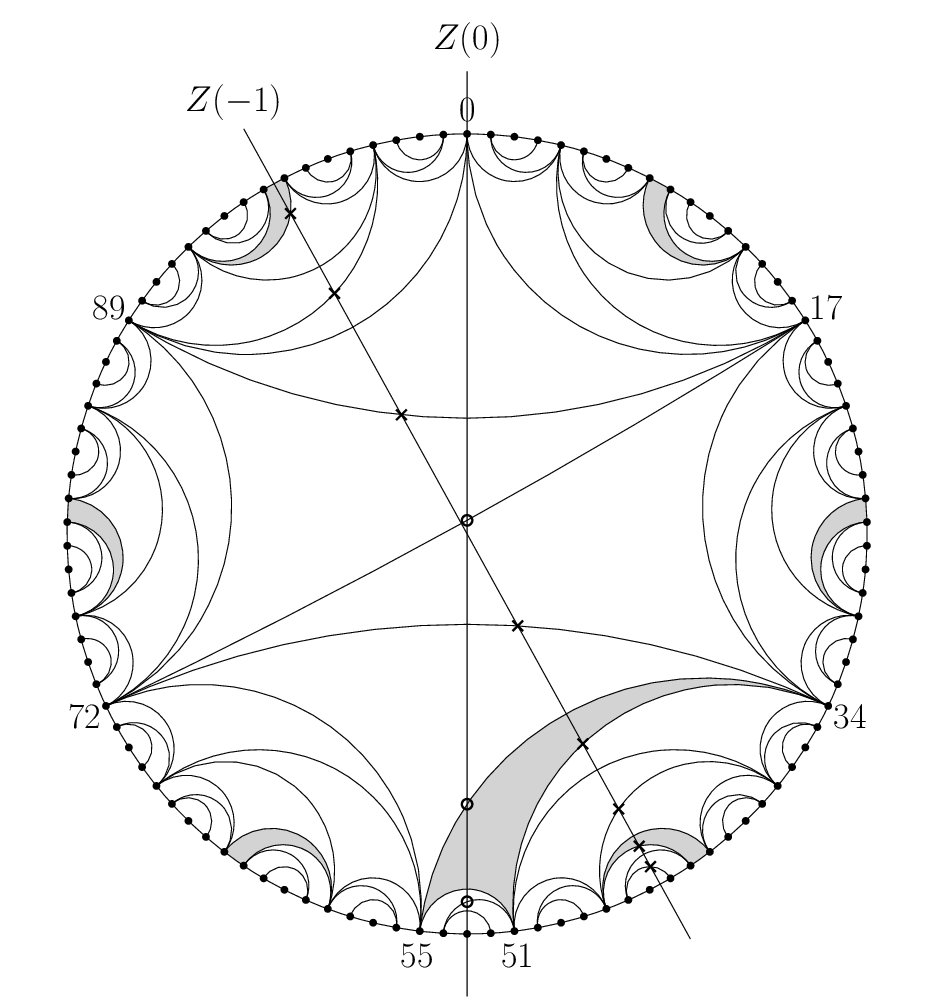}
	\end{center}
\caption{Forward and Backward mutation points
of $\Gamma(6,4,3)$.
The diagonals  
for the forward (resp. backward) mutation points at $u=0$
are marked with circles (resp. crosses).
}
\label{fig:643}
\end{figure}

We conclude this section
by introducing the notion of {\em quasi-reflection symmetry\/}
for our example $\Gamma(6,4,3)$.
Let $Z(-1)$ and $Z(0)$ be the axes in Figure \ref{fig:643},
so that the backward and forward mutations
at $u=0$ are the reflections with respect to them.
We say that  $\Gamma(6,4,3)$
is {\em quasi-symmetric with respect to the
axis $Z(u)$} ($u=-1,0$) in the following sense:
it is symmetric with respect to $Z(u)$
{\em except for\/} the diagonals which  intersect
$Z(u)$  in the interior of the 106-gon.
Moreover,
observe that
that a label is a forward (resp. backward) mutation point at $u=0$ if and only if
 the corresponding diagonal intersects  $Z(0)$ (resp. $Z(-1)$)
in the interior of the 106-gon and  it is {\em not} symmetric
with respect to it.
Thus, the quasi-symmetry of $\Gamma(6,4,3)$ is the source
of the mutations.


\section{Realization of RSG $Y$-systems by polygons}
\label{sec:RSG}

Now we will construct triangulations
of polygons realizing the RSG $Y$-system, in full generality,
and prove the periodicity of Theorem \ref{thm:periodRSG}.

First, we introduce  a triangulation of a polygon associated to an
arbitrary continued fraction.
Next, we show that such a triangulation has a nice
{\em quasi-reflection symmetry},
which naturally defines a sequence of mutations.
 Then, we show that this sequence of mutations realizes the 
corresponding RSG $Y$-system.
This gives the foundation of the entire method.
As a result, we obtain the periodicity of the RSG $Y$-systems.

\subsection{Construction of initial labeled triangulations}
\label{subsec:RSG}

Let $(n_1, \dots, n_F)$ be any sequence of positive integers
with $n_1\geq 2$, other than $(n_1)=(2)$ with $F=1$.
Recall that the numbers
 $r=r^{(1)}$, $r^{(2)}$,
\dots, $r^{(F)}$ are defined in
\eqref{eq:rad},
and we set $r^{(F+1)}=r^{(F+2)}=1$ (see Proposition \ref{prop:cf1} (b))
throughout the rest of the paper.

In this subsection we define the initial triangulation
$\Gamma(n_1,\dots,n_F)$
of an $r$-gon which realizes the
RSG $Y$-system
$\mathbb{Y}_{\mathrm{RSG}}(n_1,\dots,n_F)$.
We construct it by an iterative procedure
on generations
as suggested in the examples in Section \ref{sec:643}.
The vertices of an $r$-gon are counted as
$0$, \dots, $r-1$ modulo $r$ from the top and clockwise.

 \begin{figure}
 	\begin{center}
		\includegraphics[scale=.43]{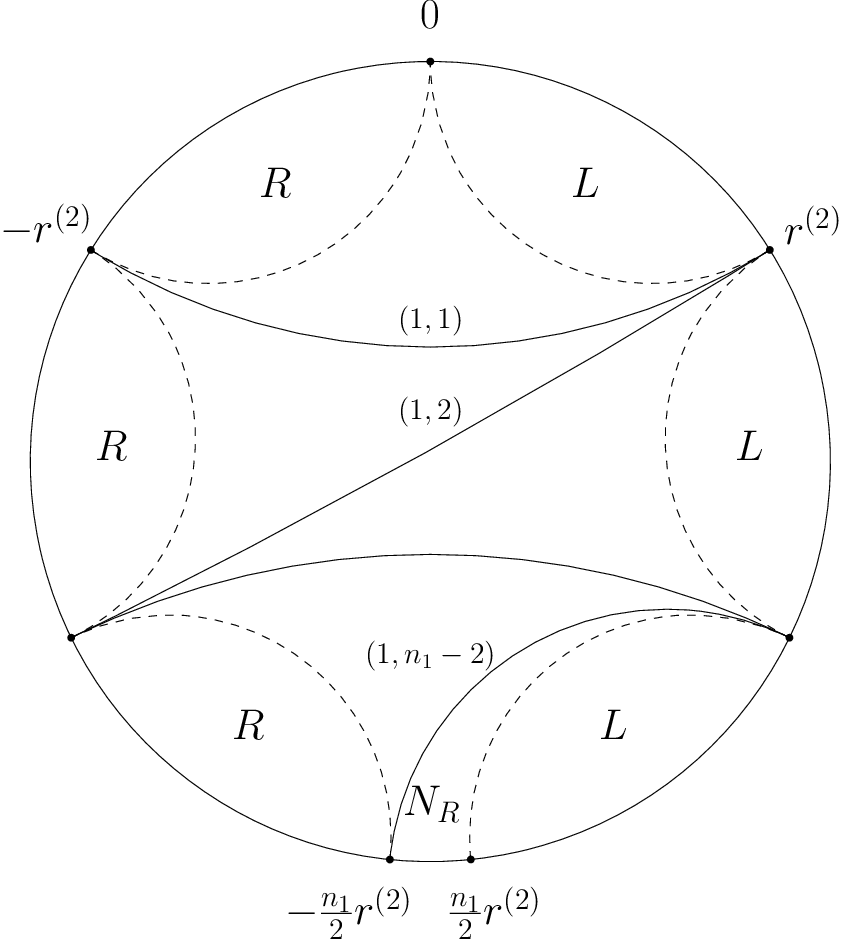}
\hskip7pt
		\includegraphics[scale=.43]{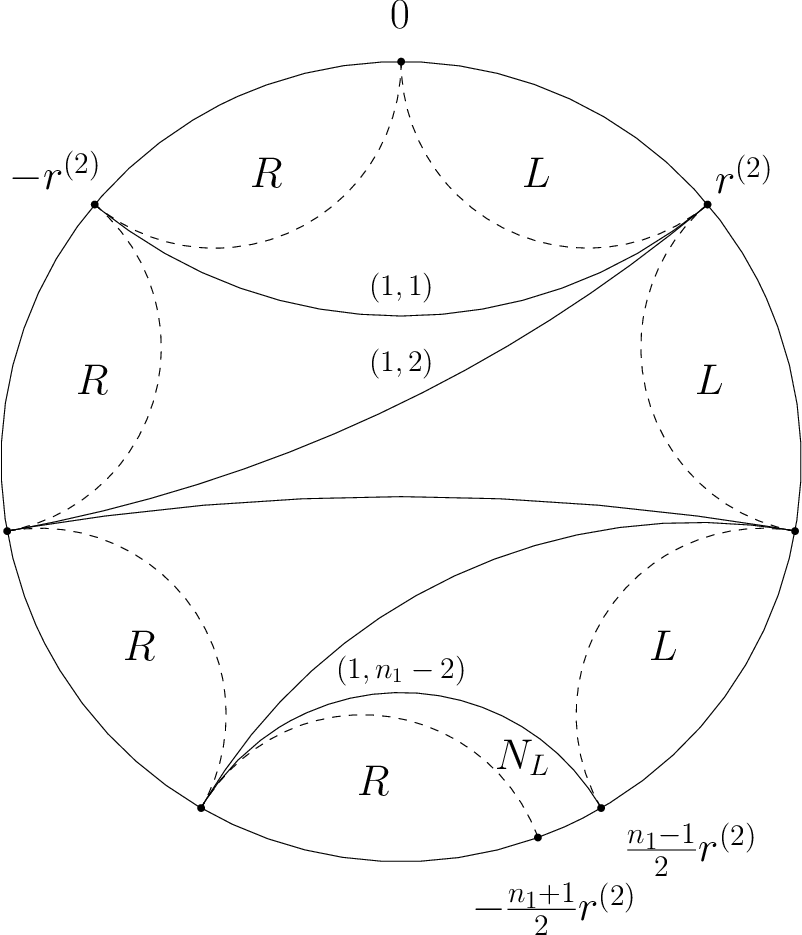}
	\end{center}
\centerline{
(a) even $n_1$
\hskip130pt
(b) odd $n_1$
}
\caption{Diagonals of the first generation
and  intervals of the second
generation. The doted lines indicate the outline of  diagonals of the second generation.
}
\label{fig:firstgen}
\end{figure}

The iterative procedure is divided into three steps as follows.

\medskip
 {\em Step 1. The diagonals of the first generation
 and the intervals of the second generation.} 
 
 This initial step  depends on the parity of $n_1$.

{\em (a). The case $n_1$  even.}
 We first mark $n_1+1$ vertices
\begin{align}
0,r^{(2)},
2r^{(2)},
\dots,
(n_1/2)r^{(2)};
-r^{(2)},
-2r^{(2)},
\dots,
-(n_1/2)r^{(2)},
\end{align}
then
 draw the $n_1-2$ diagonals of the first generation
as in Figure \ref{fig:firstgen} (a).
Namely, we draw a zigzag starting from $-r^{(2)}$,
$r^{(2)}$, 
$-2r^{(2)}$,
\dots, and ending at $-(n_1/2)r^{(2)}$.
No diagonal is drawn from vertices 0 and $(n_1/2)r^{(2)}$.
We label the diagonals $(1,1),\dots, (1,n_1-2)$ from the
top to the bottom.

The above vertices splits the boundary
of the $r$-gon
into $n_1+1$ intervals.
To each interval we assign its {\em type\/},
 $L$, $R$, or $N_R$ as in Figure \ref{fig:firstgen} (a).
Here, the symbols $L$, $R$, $N_R$, and forthcoming $N_L$
stand for  {\em left-twisted\/}, {\em right-twisted\/},
{\em right-twisted-neutral\/}, 
{\em left-twisted-neutral\/}, respectively.
We call them the {\em intervals 
of the second generation}.
We do so because later we  will fit  the {\em diagonals
of the second generation\/} inside these intervals.
Equivalently, we cut out the boundary of the $r$-gon 
at the top vertex $0$,
and present them in the following diagram.
\medskip
$$
\begin{xy}
(7.5,-10)*{\underbrace{\phantom{xxxxxx}}_{\displaystyle r^{(2)}}},
(50.5,-10)*{\underbrace{\phantom{xxxx}}_{\displaystyle r^{(3)}}},
(92.5,-10)*{\underbrace{\phantom{xxxxxx}}_{\displaystyle r^{(2)}}},
(22.5,-20)*{\underbrace{\phantom{xxxxxxxxxxxxxxxxxxx}}_{%
\displaystyle (n_1/2) r^{(2)}}},
(77.5,-20)*{\underbrace{\phantom{xxxxxxxxxxxxxxxxxxx}}_{%
\displaystyle (n_1/2) r^{(2)}}},
(0,4)*{0},
(3,4)*{1},
(8,4)*{\cdots},
(93,4)*{r-1},
(100,4)*{0},
(7.5,-3)*{L},
(22.5,-3)*{L},
(37.5,-3)*{L},
(50.5,-3.3)*{N_R},
(62.5,-3)*{R},
(77.5,-3)*{R},
(92.5,-3)*{R},
%
\ar@{-} (0,0);(100,0)
\ar@{-} (0,1);(0,-1)
\ar@{-} (3,0);(3,1)
\ar@{-} (15,0);(15,-1)
\ar@{-} (30,0);(30,-1)
\ar@{-} (45,0);(45,-1)
\ar@{-} (55,0);(55,-1)
\ar@{-} (70,0);(70,-1)
\ar@{-} (85,0);(85,-1)
\ar@{-} (97,0);(97,1)
\ar@{-} (100,1);(100,-1)
\end{xy}
$$
\medskip
By \eqref{eq:rr4},
we have
$r= n_1 r^{(2)} + r^{(3)}$
so that the diagram makes sense.

{\em (b). The case $n_1$  odd.}
 We first mark $n_1+1$ vertices
\begin{align}
0,r^{(2)},
2r^{(2)},
\dots,
((n_1-1)/2)r^{(2)};
-r^{(2)},
-2r^{(2)},
\dots,
-((n_1+1)/2)r^{(2)},
\end{align}
then
 draw the $n_1-2$ diagonals of the first generation
as in Figure \ref{fig:firstgen} (b).
Namely, we draw a zigzag starting from $-r^{(2)}$,
$r^{(2)}$, 
$-2r^{(2)}$,
\dots, and ending at $((n_1-1)/2)r^{(2)}$.
No diagonal is drawn from vertices 0 and $-((n_1+1)/2)r^{(2)}$.
We label the diagonals $(1,1),\dots, (1,n_1-2)$ from the
top to the bottom.

Again,
we also assign
 the types to intervals of the second generation
 as in Figure \ref{fig:firstgen} (b);
 or equivalently,
as presented in the following diagram.
\medskip
$$
\begin{xy}
(-7.5,-10)*{\underbrace{\phantom{xxxxxx}}_{\displaystyle r^{(2)}}},
(35,-10)*{\underbrace{\phantom{xxxx}}_{\displaystyle r^{(3)}}},
(92.5,-10)*{\underbrace{\phantom{xxxxxx}}_{\displaystyle r^{(2)}}},
(7.5,-20)*{\underbrace{\phantom{xxxxxxxxxxxxxxxxxxxxx}}_{%
\displaystyle ((n_1-1)/2) r^{(2)}}},
(70,-20)*{\underbrace{\phantom{xxxxxxxxxxxxxxxxxxxxxxxxxxx}}_{%
\displaystyle ((n_1+1)/2) r^{(2)}}},
(-15,4)*{0},
(-12,4)*{1},
(-7,4)*{\cdots},
(93,4)*{r-1},
(100,4)*{0},
(-7.5,-3)*{L},
(7.5,-3)*{L},
(22.5,-3)*{L},
(35,-3)*{N_L},
(47.5,-3.3)*{R},
(62.5,-3)*{R},
(77.5,-3)*{R},
(92.5,-3)*{R},
%
\ar@{-} (-15,0);(100,0)
\ar@{-} (-15,1);(-15,-1)
\ar@{-} (-12,0);(-12,1)
\ar@{-} (0,0);(0,-1)
\ar@{-} (15,0);(15,-1)
\ar@{-} (30,0);(30,-1)
\ar@{-} (40,0);(40,-1)
\ar@{-} (55,0);(55,-1)
\ar@{-} (70,0);(70,-1)
\ar@{-} (85,0);(85,-1)
\ar@{-} (97,0);(97,1)
\ar@{-} (100,1);(100,-1)
\end{xy}
$$

\medskip
{\em Step 2. The intervals of the generation $a=3,\dots,F$.} 
We will construct the
intervals of  newer generations by induction.
Suppose that we have the intervals of the $a$-th generation
such that the widths of the intervals of types $L$ and $R$
are all $r^{(a)}$, while the widths of the intervals of types
$N_L$ and $N_R$ are all $r^{(a+1)}$.
(This was done for $a=2$ in Step 1.)
Then, the intervals of the $(a+1)$-th generation are
defined by the subdivision of  the intervals of the $a$-th
generation as follows.

For even $n_{a}$,
\medskip
$$
\begin{xy}
(10,-10)*{\underbrace{\phantom{xxxx}}_{\displaystyle r^{(a+1)}}},
(27.5,-10)*{\underbrace{\phantom{x}}_{\displaystyle r^{(a+2)}}},
(45,-10)*{\underbrace{\phantom{xxxx}}_{\displaystyle r^{(a+1)}}},
(60,-10)*{\underbrace{\phantom{xxxx}}_{\displaystyle r^{(a+1)}}},
(75,-10)*{\underbrace{\phantom{xxxx}}_{\displaystyle r^{(a+1)}}},
(92.5,-10)*{\underbrace{\phantom{x}}_{\displaystyle r^{(a+2)}}},
(110,-10)*{\underbrace{\phantom{xxxx}}_{\displaystyle r^{(a+1)}}},
(15,-20)*{\underbrace{\phantom{xxxxxxxx}}_{%
\displaystyle (n_{a}/2) r^{(a+1)}}},
(40,-20)*{\underbrace{\phantom{xxxxxxxx}}_{%
\displaystyle (n_{a}/2) r^{(a+1)}}},
(80,-20)*{\underbrace{\phantom{xxxxxxxx}}_{%
\displaystyle (n_{a}/2) r^{(a+1)}}},
(105,-20)*{\underbrace{\phantom{xxxxxxxx}}_{%
\displaystyle (n_{a}/2) r^{(a+1)}}},
(27.5,10)*{L},
(60,9.7)*{N_T},
(92.5,10)*{R},
(10,-3)*{L},
(20,-3)*{L},
(27.5,-3.3)*{N_L},
(35,-3)*{R},
(45,-3)*{R},
(60,-3)*{T},
(75,-3)*{L},
(85,-3)*{L},
(92.5,-3.3)*{N_R},
(100,-3)*{R},
(110,-3)*{R},
%
\ar@{-} (5,13);(50,13)
\ar@{-} (55,13);(65,13)
\ar@{-} (70,13);(115,13)
\ar@{-} (5,12);(5,13)
\ar@{-} (50,12);(50,13)
\ar@{-} (55,12);(55,13)
\ar@{-} (65,12);(65,13)
\ar@{-} (70,12);(70,13)
\ar@{-} (115,12);(115,13)
\ar@{-} (5,0);(50,0)
\ar@{-} (55,0);(65,0)
\ar@{-} (70,0);(115,0)
\ar@{-} (5,0);(5,-1)
\ar@{-} (15,0);(15,-1)
\ar@{-} (25,0);(25,-1)
\ar@{-} (30,0);(30,-1)
\ar@{-} (40,0);(40,-1)
\ar@{-} (50,0);(50,-1)
\ar@{-} (55,0);(55,-1)
\ar@{-} (65,0);(65,-1)
\ar@{-} (70,0);(70,-1)
\ar@{-} (80,0);(80,-1)
\ar@{-} (90,0);(90,-1)
\ar@{-} (95,0);(95,-1)
\ar@{-} (105,0);(105,-1)
\ar@{-} (115,0);(115,-1)
\end{xy}
$$
\medskip
and, for odd $n_{a}$,
\medskip
$$
\begin{xy}
(0,-10)*{\underbrace{\phantom{xxxx}}_{\displaystyle r^{(a+1)}}},
(27.5,-10)*{\underbrace{\phantom{x}}_{\displaystyle r^{(a+2)}}},
(45,-10)*{\underbrace{\phantom{xxxx}}_{\displaystyle r^{(a+1)}}},
(57.5,-10)*{\underbrace{\phantom{xxxx}}_{\displaystyle r^{(a+1)}}},
(70,-10)*{\underbrace{\phantom{xxxx}}_{\displaystyle r^{(a+1)}}},
(87.5,-10)*{\underbrace{\phantom{x}}_{\displaystyle r^{(a+2)}}},
(115,-10)*{\underbrace{\phantom{xxxx}}_{\displaystyle r^{(a+1)}}},
(10,-20)*{\underbrace{\phantom{xxxxxxxxxxxxx}}_{%
\displaystyle ((n_{a}+1)/2) r^{(a+1)}}},
(40,-20)*{\underbrace{\phantom{xxxxxxxx}}_{%
\displaystyle ((n_{a}-1)/2) r^{(a+1)}}},
(75,-20)*{\underbrace{\phantom{xxxxxxxx}}_{%
\displaystyle ((n_{a}-1)/2) r^{(a+1)}}},
(105,-20)*{\underbrace{\phantom{xxxxxxxxxxxxx}}_{%
\displaystyle ((n_{a}+1)/2) r^{(a+1)}}},
(22.5,10)*{L},
(57.5,9.7)*{N_T},
(92.5,10)*{R},
(0,-3)*{L},
(10,-3)*{L},
(20,-3)*{L},
(27.5,-3.3)*{N_R},
(35,-3)*{R},
(45,-3)*{R},
(57.5,-3)*{T},
(70,-3)*{L},
(80,-3)*{L},
(87.5,-3.3)*{N_L},
(95,-3)*{R},
(105,-3)*{R},
(115,-3)*{R},
%
\ar@{-} (-5,13);(50,13)
\ar@{-} (52.5,13);(62.5,13)
\ar@{-} (65,13);(120,13)
\ar@{-} (-5,12);(-5,13)
\ar@{-} (50,12);(50,13)
\ar@{-} (52.5,12);(52.5,13)
\ar@{-} (62.5,12);(62.5,13)
\ar@{-} (65,12);(65,13)
\ar@{-} (120,12);(120,13)
\ar@{-} (-5,0);(50,0)
\ar@{-} (52.5,0);(62.5,0)
\ar@{-} (65,0);(120,0)
\ar@{-} (-5,0);(-5,-1)
\ar@{-} (5,0);(5,-1)
\ar@{-} (15,0);(15,-1)
\ar@{-} (25,0);(25,-1)
\ar@{-} (30,0);(30,-1)
\ar@{-} (40,0);(40,-1)
\ar@{-} (50,0);(50,-1)
\ar@{-} (52.5,0);(52.5,-1)
\ar@{-} (62.5,0);(62.5,-1)
\ar@{-} (65,0);(65,-1)
\ar@{-} (75,0);(75,-1)
\ar@{-} (85,0);(85,-1)
\ar@{-} (90,0);(90,-1)
\ar@{-} (100,0);(100,-1)
\ar@{-} (110,0);(110,-1)
\ar@{-} (120,0);(120,-1)
\end{xy}
$$
where $T=L,R$ and the upper and the lower intervals
are the ones of the $a$-th and $(a+1)$-th generations,
respectively. 
Note that the intervals of type $N_T$ only change their type to $T$
and do not change their widths.
Again, by
\eqref{eq:rr4},
we have
$r^{(a)}= n_{a} r^{(a+1)} + r^{(a+2)}$
so that the subdivision makes sense.

\begin{rem}
One can  formally 
regard the whole interval with width $r$
as the  interval of type $R$ of the {\em first generation}.
Then,
 the intervals of the second
generation are obtained from 
the above iteration rule.
\end{rem}

\begin{prop}
\label{prop:int1}
For $a=2,\dots, F$, the total number of intervals of types 
$L$ and $R$ (resp.  $N_L$ and $N_R$)
of the $a$-th generation
is $q_{a-1}=p_{a}$ (resp. $p_{a-1}$).
\end{prop}
\begin{proof}
This is true for $a=2$,
since $q_1=n_1$ and  $p_1=1$.
For $a\geq 3$, it is shown by
using the recursion relations
\eqref{eq:p-q1} and \eqref{eq:p-q2}.
\end{proof}

The proposition implies the equality
$r=q_{a-1} r^{(a)} + p_{a-1} r^{(a+1)}$.
This  certainly agrees with the formula
\eqref{eq:rF3} when $k=1$.

According to Proposition \ref{prop:int1},
we  name the intervals of types $L$ and $R$
of the $a$-th generation  as 
$I^{(a)}_s$  $(s=1,\dots, p_{a})$ from the top-right
of the $r$-gon and clockwise.

\medskip
{\em Step 3. The diagonals of the generation $a\geq 2$.} 
Next we draw the diagonals of the $a$-th generation
for $a=2,\dots, F$ as follows.
For each interval $I^{(a)}_s$ ($s=1,\dots, p_{a}$)
of type $L$ or $R$ of the $a$-th generation,
we draw the diagonals
and attach the labels $(a,1)_{s},\dots,
(a,n_a)_{s}$ 
 as specified in Figure
\ref{fig:diagonal1}.
This finishes the construction.

Note  that the vertices of the diagonals
in Figure
\ref{fig:diagonal1} match the boundaries of 
the intervals of the $(a+1)$-th generation,
so that the diagonals never intersect each other,
regardless of their generations.

 \begin{figure}
 	\begin{center}
		\includegraphics[scale=.5]{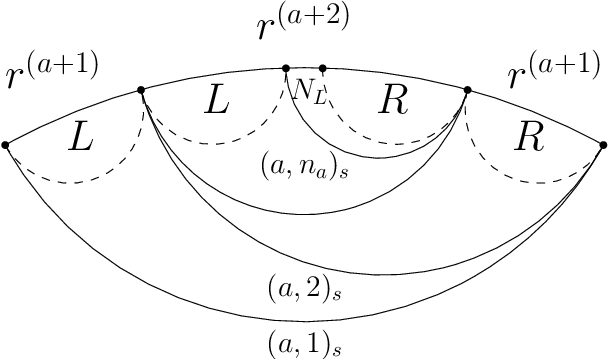}
		\hskip20pt
		\includegraphics[scale=.5]{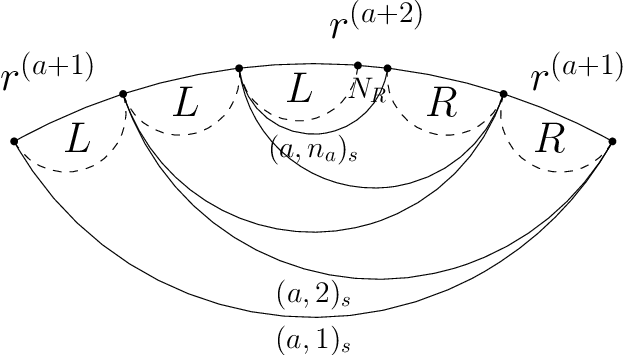}
	\end{center}
 type $L$, even $n_a$
\hskip100pt
 type $L$, odd $n_a$
\vskip10pt
	\begin{center}
		\includegraphics[scale=.5]{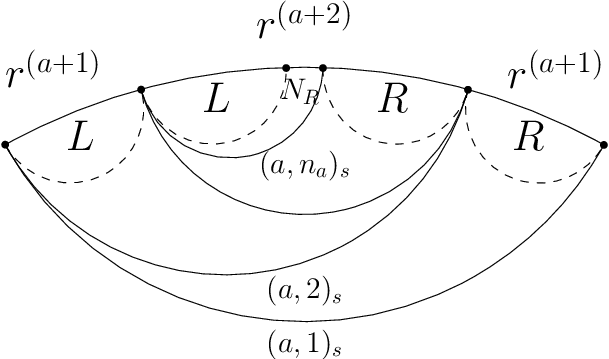}
		\hskip20pt
		\includegraphics[scale=.5]{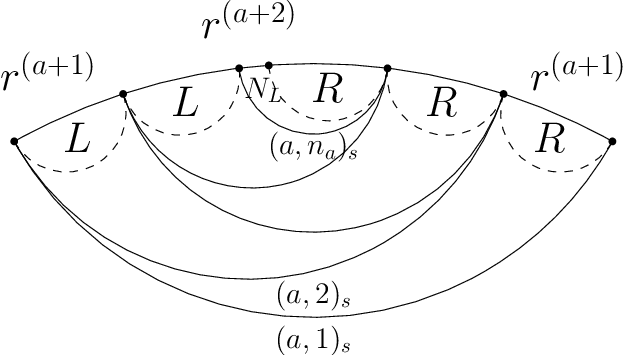}
	\end{center}
type $R$, even $n_a$
\hskip100pt
type $R$, odd $n_a$
\caption{Diagonals of the $a$-th generation
and  intervals of the $(a+1)$-th generation.
}
\label{fig:diagonal1}
\end{figure}

\begin{prop}
The diagonals drawn above give a triangulation
of an $r$-gon.
\end{prop}
\begin{proof}
From Figure \ref{fig:firstgen}
and the construction of the diagonals,
it is enough to show that
all  polygons in 
Figure
\ref{fig:diagonal1} will be triangulated
for any generation $a=2,\dots, F$
after drawing  all diagonals.
The claim is true for $a=F$
because $r^{(F+1)}=r^{(F+2)}=1$. 
Then, the claim  follows
by induction on $a$ in the decreasing order.
\end{proof}

The labeled triangulation obtained above
is denoted by
$\Gamma_{\mathrm{RSG}}(n_1,\dots,n_F)$.
Some examples are given
in Figure \ref{fig:examples1} (for even $n_1$)
and Figure \ref{fig:examples2} (for odd $n_1$).

\begin{rem}
\label{rem:mirror}
By construction, all  diagonals in an interval of type $L$
and the ones in an interval of type $R$ of a given generation $a\geq 2$
are mirror images of each other.
\end{rem}

 \begin{figure}
			\begin{center}
		\includegraphics[scale=.63]{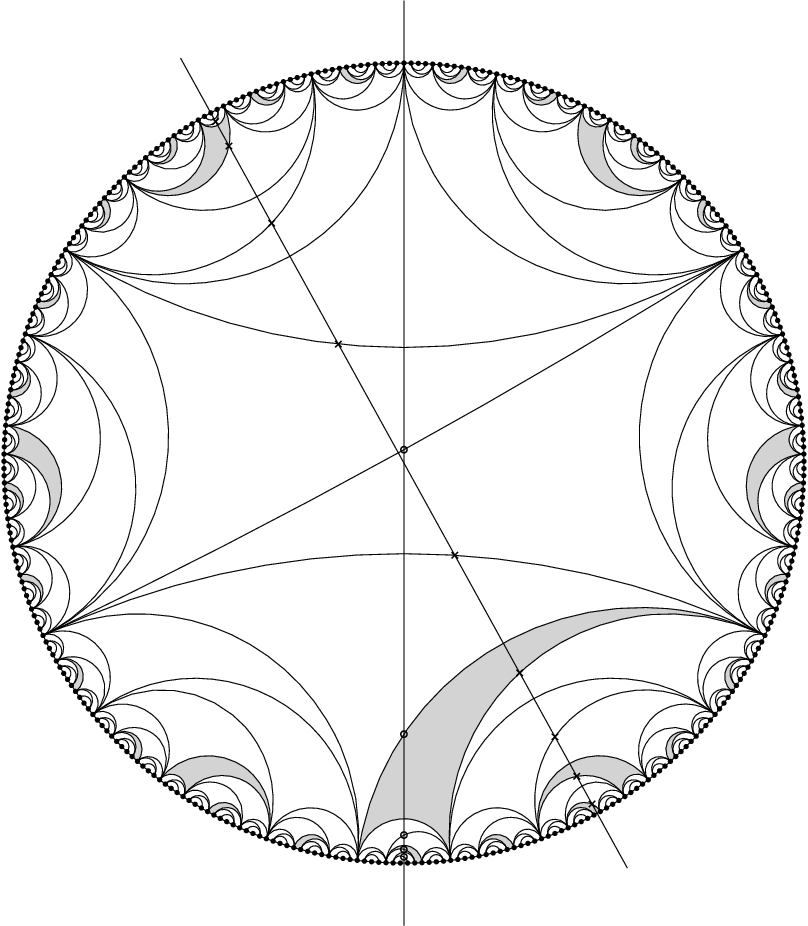}
		\includegraphics[scale=.63]{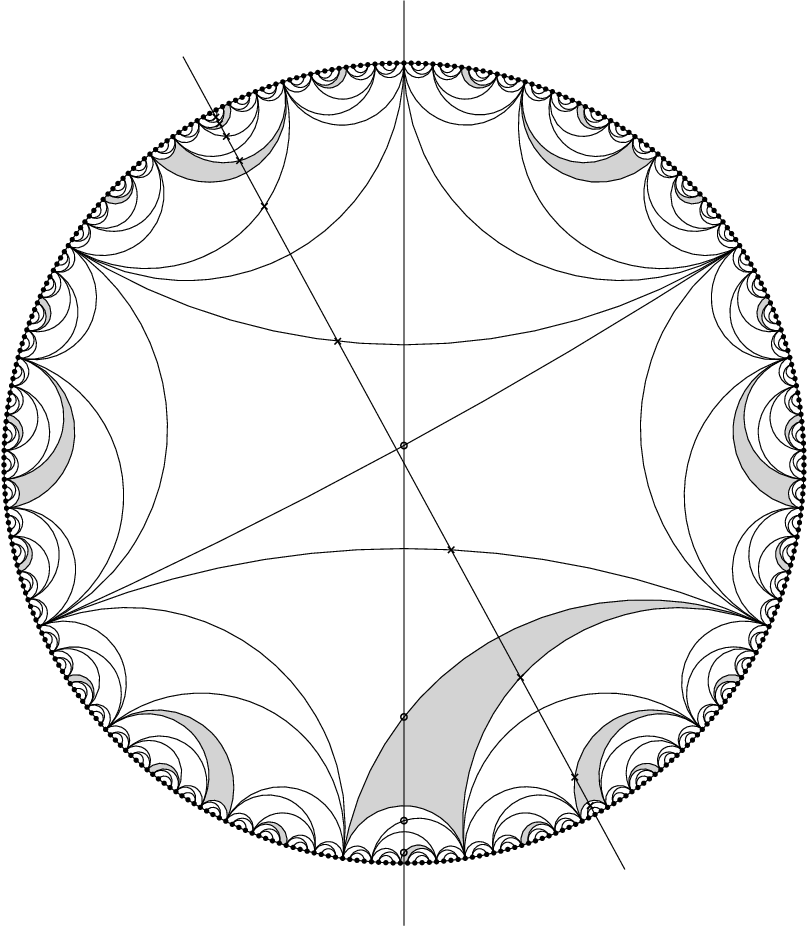}
	\end{center}
\caption{The triangulations
$\Gamma_{\mathrm{RSG}}(6,4,3,3)$ (left) and 
$\Gamma_{\mathrm{RSG}}(6,3,4,3)$ (right).
}
\label{fig:examples1}
\end{figure}
\begin{figure}
			\begin{center}
		\includegraphics[scale=.63]{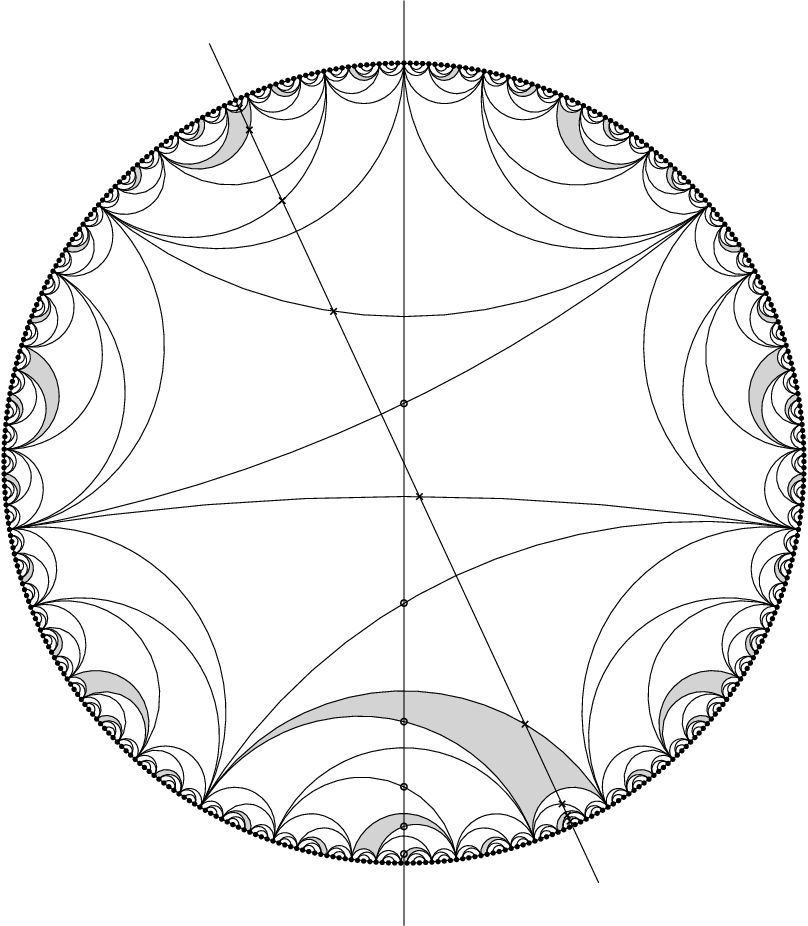}
		\includegraphics[scale=.63]{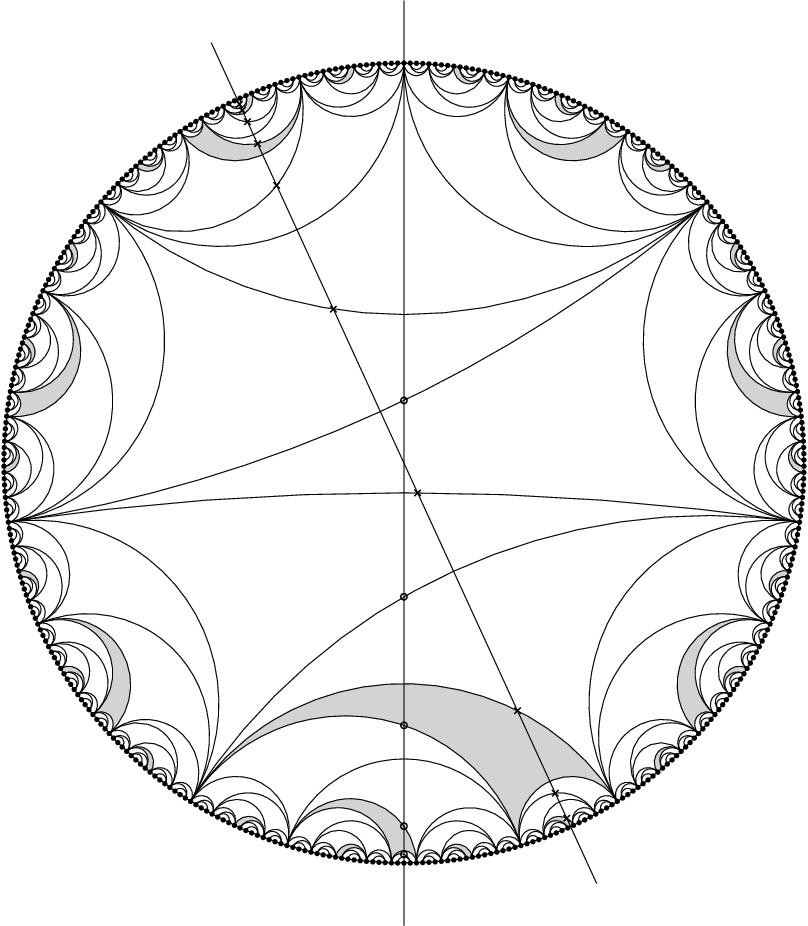}
	\end{center}
\caption{The triangulations
$\Gamma_{\mathrm{RSG}}(7,4,3,3)$ (left) and 
$\Gamma_{\mathrm{RSG}}(7,3,4,3)$ (right).
}
\label{fig:examples2}
\end{figure}

\subsection{Quasi-reflection symmetry}
\label{subsec:QS2}

To define the forward and backward mutation points of
$\Gamma_{\mathrm{RSG}}(n_1,\dots,n_F)$,
we use the quasi-reflection symmetry observed
in the example $\Gamma_{\mathrm{RSG}}(6,4,3)$ in Section \ref{subsec:QS1}.
Let us first establish the quasi-reflection symmetry
of $\Gamma_{\mathrm{RSG}}(n_1,\dots,n_F)$.
Let $Z(0)$ (resp. $Z(-1)$) be the axis of the $r$-gon 
passing through the points
$P(0)=0$ and $Q(0)=r/2$ (resp. $P(-1)=-r^{(2)}/2$ and $Q(-1)=-(r+r^{(2)})/2$)
on the boundary of the $r$-gon,
where we continue to use our clockwise coordinates mod $r$.
Since $r$ and $r^{(2)}$ are coprime by Proposition 
\ref{prop:cf1} (d), they are not even numbers simultaneously.
Thus, only one of $Q(0)$, $P(-1)$, $Q(-1)$ is a vertex of the $r$-gon,
and the other two are midpoints of two adjacent vertices.

\begin{defn} We say that an (unlabeled) triangulation $\Gamma$
of the $r$-gon
is {\em quasi-symmetric with respect to the axis $Z(u)$} ($u=-1,0$)
if it is symmetric {\em except for\/} the diagonals which intersect 
$Z(u)$ in the interior of the $r$-gon.
\end{defn}

\begin{defn}
For any  $a=2,\dots,F$,
the union $I\cup I'$ of 
an adjacent pair of intervals $I$ and $I'$
of the $a$-th generation
 of type $L$ and $N_R$
(resp. $N_L$ and $R$),
exactly in this order in our clockwise coordinate system,
is called a {\em joint-interval of type $(L,N_R)$} (resp. 
{\em of type $(N_L,R)$}).
\end{defn}

We first prove the following key lemma,
which can be observed  in the examples
in Figures \ref{fig:examples1} and \ref{fig:examples2}.

\begin{lem}[Trinity of intersections]
\label{lem:mid1}
For any $a=2,\dots, F$,
the axis $Z(u)$ $(u=-1,0)$
intersects  the intervals of the $a$-generation in the following way.

(a). The point $P(0)$ is the boundary of two adjacent intervals.

 (b). Let $P$ be any of $Q(0)$, $P(-1)$, $Q(-1)$.
Then, only one of the following three cases occurs.

\begin{itemize}
\item[(i).]  The point $P$ is exactly the midpoint of
an interval $J^{(a)}_1$ of type $N_L$ or $N_R$.

\item[(ii).] The point $P$ is exactly the midpoint of
an interval $J^{(a)}_2$ of type $L$ or $R$.

\item[(iii).] The point $P$ is
exactly the midpoint of a joint-interval $J^{(a)}_3$ of
type $(L,N_R)$ or $(N_L,R)$.
\end{itemize}
Furthermore, each case among (i), (ii), (iii) occurs for exactly  one of three points
$Q(0)$, $P(-1)$, $Q(-1)$.
\end{lem}

\begin{proof} (a). This is true for $a=2$,
and also true for $a\geq 3$
by the definition
of  subdivision of  intervals.
\par
(b). For $a=2$, by the definition of the intervals
in Figure \ref{fig:firstgen},
we know that the types of  $Q(0)$, $P(-1)$, $Q(-1)$
are (i) $N_R$, (ii) $R$, (iii) $(L,N_R)$ for even $n_1$,
and (iii) $(N_L,R)$, (ii) $R$, (i) $N_L$ for odd $n_1$, respectively.
Therefore, the claim is true for $a=2$.
Suppose that the claim is true for  $a$.
Then, the claim is true for  $a+1$
by the next lemma.
\end{proof}

The explicit rule by which  $P$ changes its type
is useful later, so we put it  here separately.  

\begin{lem}[Rule of type-change]
\label{lem:change1}
For any $P=Q(0),P(-1),Q(-1)$,
 the type of $P$ from Lemma \ref{lem:mid1} (b) 
changes  in the following way as $a$ increases:
for even $n_{a}$,
\begin{align}
\label{eq:type1}
\begin{tabular}{c||ccc|ccc|ccc}
$a$ & (i)&$N_L$ & $N_R$ &
(ii)  & $L$ & $R$ &
(iii)& $(L,N_R)$ & $(N_L,R)$\\
\hline
$a+1$ & (ii) &$L$&$R$& (i) &$N_L$&$N_R$& (iii)&
$(N_L,R)$&$(L,N_R)$ \\
\end{tabular}
\end{align}
and, for odd $n_{a}$,
\begin{align}
\label{eq:type2}
\begin{tabular}{c||ccc|ccc|ccc}
$a$ & (i)&$N_L$ & $N_R$ &
(ii)  & $L$ & $R$ &
(iii)& $(L,N_R)$ & $(N_L,R)$\\
\hline
$a+1$ & (ii) &$L$&$R$& (iii) &$(L,N_R)$&$(N_L,R)$& (i)&
$N_R$&$N_L$ \\
\end{tabular}
\end{align}
\end{lem}
\begin{proof}
This follows from  the definition of  subdivision of  intervals.
\end{proof}

Note that the rule preserves the {\em mirror images of types}
\begin{align}
\label{eq:mirror1}
N_L\leftrightarrow N_R,
\quad
 L\leftrightarrow R,
 \quad
(L,N_R) \leftrightarrow (N_L,R).
 \end{align}

\begin{defn}
For $a=2,\dots,F$,
let $J^{(a)}_1$, $J^{(a)}_2$, $J^{(a)}_3$ be the (joint-)intervals
 in Lemma \ref{lem:mid1} (b).
 We say that a diagonal (of any generation)  {\em belongs\/} to
 $J^{(a)}_i$ ($i=1,2,3$)  if  both its end points
 are in  $J^{(a)}_i$.
\end{defn}

\begin{figure}
%
		\includegraphics[scale=.45]{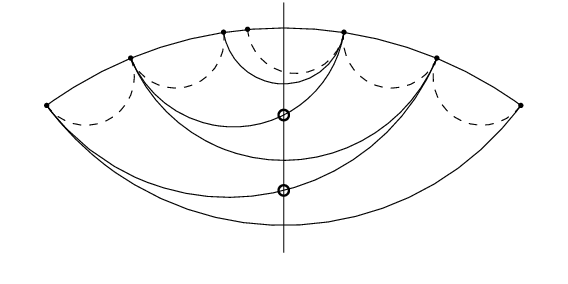}
		\hskip20pt
		\includegraphics[scale=.45]{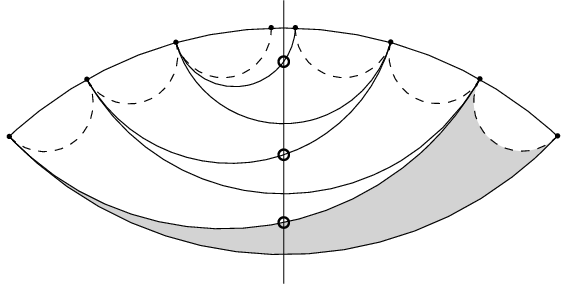}
		\centerline{\hskip16pt $R$ \hskip125pt $(L,N_R)$}
		\vskip10pt
		\includegraphics[scale=.45]{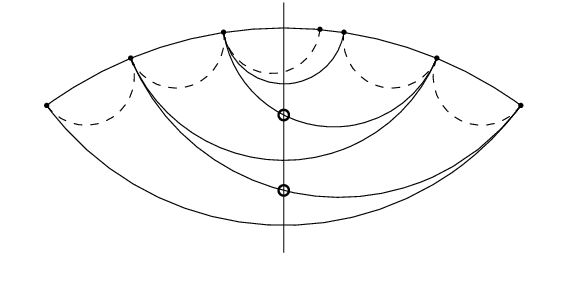}
		\hskip20pt
		\includegraphics[scale=.45]{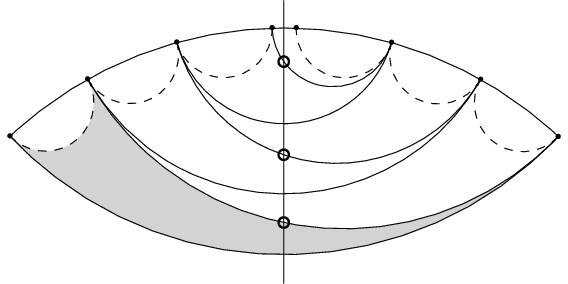}
		\centerline{\hskip16pt $L$ \hskip125pt $(N_L,R)$}
\caption{Intersection of diagonals and the axis $Z(u)$ ($u=-1,0$).
The marked diagonals are not symmetric with respect to
$Z(u)$.}
\label{fig:fmp1}
\end{figure}

Lemma  \ref{lem:mid1} has the following important consequence.
See Figure \ref{fig:fmp1} for the case $n_a=5$.
\begin{lem}
\label{lem:int}
Let $a=2,\dots,F$.
\par
(a). A diagonal of the $a$-th generation
intersects  one of the axes $Z(0)$ and $Z(-1)$
in the interior of the $r$-gon
if and only if it belongs to $J^{(a)}_2$ or $J^{(a)}_3$.
\par
(b).
Suppose that a diagonal with label $(a,m)_s$ belongs to
the interval $J^{(a)}_2$ (resp. $J^{(a)}_3$)
and intersects  $Z(u)$ ($u=-1$ or $0$).
Then, 
it is not symmetric with respect to $Z(u)$
if and only if $m$ is even (resp. $m$ is odd.).
\end{lem}
\begin{proof}
Both properties are immediate consequences of Lemma  \ref{lem:mid1}
and Figure \ref{fig:fmp1}.
\end{proof}

Now we prove the quasi-reflection symmetry
of $\Gamma_{\mathrm{RSG}}(n_1,\dots,n_F)$.

\begin{thm}
\label{thm:qs}
 The triangulation $\Gamma_{\mathrm{RSG}}(n_1,\dots,n_F)$
is quasi-symmetric with respect to both the axes $Z(-1)$ and $Z(0)$.
\end{thm}
\begin{proof}
Since all  diagonals of the first generation intersect  both
$Z(-1)$ and $Z(0)$, it is enough to prove the quasi-symmetry
for the diagonals of the generation $a\geq 2$.

Let us consider the diagonals of the second generation.
By 
 Figure \ref{fig:firstgen} and Remark \ref{rem:mirror},
the quasi-symmetry 
reduces to  the quasi-symmetry 
of the diagonals belonging to the intervals $J^{(2)}_1$, $J^{(2)}_2$,  $J^{(2)}_3$.
By Lemma \ref{lem:int} (a),
there is no diagonal of the second generation belonging to $J^{(2)}_1$,
and 
any diagonal of the second generation belonging to $J^{(2)}_2$ or  $J^{(2)}_3$
intersects  $Z(-1)$ or $Z(0)$.
Thus, the quasi-symmetry holds up to the second generation.
 
 Subdivide the intervals $J^{(2)}_1$, $J^{(2)}_2$, $J^{(2)}_3$
to get intervals of the third generation.
 Again,
by
 Figure \ref{fig:diagonal1} and Remark \ref{rem:mirror},
 quasi-symmetry 
reduces to  the quasi-symmetry 
of the diagonals belonging to $J^{(3)}_1$, $J^{(3)}_2$ or $J^{(3)}_3$.
By Lemma \ref{lem:int} (a),
there is no diagonal of the third generation belonging to $J^{(3)}_1$,
and
any diagonal of the third generation belonging to
 $J^{(3)}_2$ or  $J^{(3)}_3$
intersects  $Z(-1)$ or $Z(0)$.
Thus,  quasi-symmetry holds up to the third generation.

We repeat this argument by induction
on the generation, and  quasi-symmetry 
 reduces to  the quasi-symmetry 
of the diagonals of the $F$-th generation belonging to
$J^{(F)}_2$ or $J^{(F)}_3$.
Again, this follows from  Lemma \ref{lem:int} (a).
\end{proof}

\subsection{Mutation sequence}

We set $\Gamma(0):=\Gamma_{\mathrm{RSG}}(n_1,\dots,n_F)$ to be the
initial labeled triangulation.
Following the example in Section \ref{subsec:QS1},
we introduce the forward and backward mutation points
at time $u=0$ based on its quasi-reflection symmetry
of $\Gamma(0)$.
Let $S(u)$  ($u=-1,0$) be the
set of labels of $\Gamma(0)$ 
such that the corresponding diagonals of $\Gamma(0)$ 
intersects  $Z(u)$ in the interior of the $r$-gon
and are {\em not\/} symmetric with respect to the axis $Z(u)$.
We employ the labels in the set $S(0)$ (resp. $S(-1)$) as
the {\em forward (resp. backward) mutation points\/}
 at time $u=0$.
 The labels in $S(-1)$ are also the
 {\em forward\/} mutation points at $u=-1$.

\begin{lem}
\label{lem:mset}
(a). $S(-1)\cap S(0) = \emptyset$.
\par
(b).
The label $(1,m)$ for the first generation  belongs to $S(-1)$ if $m$ is odd
and belongs to $S(0)$ if $m$ is even.
\par
(c). For any  $a=2,\dots, F$ and $m=1,\dots, n_a$,
a label $(a,m)_s$ 
belongs to $S(-1)\sqcup S(0)$
if and only if the corresponding diagonal belongs to $J^{(a)}_2$ and
$m$ is even,
or the corresponding diagonal belongs to $J^{(a)}_3$ and
$m$ is odd.
In particular, for each $(a,m)$ there is the unique $s$
such that $(a,m)_s\in S(-1)\sqcup S(0)$.

(d). $|S(-1)\sqcup S(0)|=\sum_{a=1}^F n_a - 2.$
\end{lem}
\begin{proof}
(a) and (b) are clear from Figure \ref{fig:firstgen}.
(c) is an immediate consequence of 
Lemma \ref{lem:int}.
(d)  follows from (b) and (c).
\end{proof}

\begin{defn}
For a given  labeled triangulation $\Gamma$,
we say that a set $S$ of  labels for $\Gamma$
 is {\em mutation-compatible in $\Gamma$}
if, for any pair of labels in $S$, the corresponding diagonals
do not belong to a common triangle.
\end{defn}
A mutation-compatible set of labels can be mutated ``simultaneously'',
without caring about the order of mutations.

\begin{prop}
\label{prop:compatible}
For each $u=-1,0$,
the set $S(u)$ is mutation-compatible in $\Gamma(0)$.
\end{prop}
\begin{proof}
First we give a general remark.
All  diagonals in $\Gamma(0)$
which intersect $Z(u)$ are linearly ordered
 along $Z(u)$.
  By the construction of $\Gamma(0)$,
 a pair of labels in $S(u)$   belong to a common triangle
 in $\Gamma(0)$
 if and only if the corresponding diagonals 
  are {\em adjacent\/} in this order.

Let us take any pair of labels in $S(u)$.
We claim that  the corresponding diagonals
are not adjacent in the above order.
If one of the labels belongs to the first generation,
then we can check the claim
using Figure \ref{fig:firstgen}.

Let
 $(a,m)_s, (a',m')_{s'}\in S(u)$ ($a,a'\geq 2$),
 and suppose that the corresponding
 diagonals are $d$ and $d'$.
 The following two cases should be examined.
\par
(i). {\em The case $a=a'$.}
If $d$ and $d'$ are adjacent,
then we should have $s=s'$, $m= m'\pm 1$.
But, this never occurs
thanks to Lemma \ref{lem:mset} (c).

(ii). {\em The case $a'=a+1,a+2$.}
If $d$ and $d'$ are adjacent,
then we should have $m=n_a$ and $m'=1$.
There are two subcases to consider. 
When $n_a$ is even,
$d$ is in $J^{(a)}_2$ by Lemma \ref{lem:mset} (c).
Then by Lemma \ref{lem:change1}, $a'=a+2$ and
$d'$ is in $J_2^{(a+2)}$.
When $n_a$ is odd,
$d$ is in $J^{(a)}_3$ by Lemma \ref{lem:mset} (c).
Again, by Lemma \ref{lem:change1}, $a'=a+2$ and 
$d'$ is in $J_2^{(a+2)}$.
In either case, $m'$ is even by Lemma \ref{lem:mset} (c),
so that $d$ and $d'$ are not adjacent.
\end{proof}

The following is an analogue of
Proposition \ref{prop:int1},
and it clarifies the meaning of the numbers $q^{(k)}_a$ and $p^{(k)}_a$
in our triangulation.

\begin{prop}
\label{prop:int2}
Let $k=2,\dots, F-1$.
For $a=k+1, \dots, F$,
the total number of intervals of types 
$L$ and $R$ (resp.  $N_L$ and $N_R$)
of the $a$-th generation
inside an interval
of type $L$ or $R$ of the $k$-th generation
is $q^{(k)}_{a-1}=p^{(k)}_{a}$ (resp. $p^{(k)}_{a-1}$).
\end{prop}
\begin{proof}
This is proved in the same way as Proposition \ref{prop:int1}.
\end{proof}

The proposition implies the equality
 $r^{(k)}=q^{(k)}_{a-1} r^{(a)} + p^{(k)}_{a-1} r^{(a+1)}$.
This certainly agrees with the formula
\eqref{eq:rF3}.

In particular, 
for $a=2,\dots,F$,
there are $p^{(2)}_a$ intervals of types $L$ and $R$
of the $a$-th generation between
0 and $r^{(2)}$.
In view of this,
we define the permutation $\nu$  of the labels of $\Gamma(0)$ by
\begin{align}
\nu: (1,m) \mapsto (1,m),\quad
	 (a,m)_s \mapsto (a,m)_{s+p^{(2)}_a},
	\quad a=2, \dots, F,	
\end{align}
where the subscript $s$ for the $a$-th generation is defined modulo $p_a$.
We define the subsets $S(u)$ ($u\in \mathbb{Z}$) of the labels of $\Gamma(0)$
by
\begin{align}
S(u)=
\begin{cases}
\nu^{u/2}(S(0))& \mbox{$u$ is even}\\
\nu^{(u+1)/2}(S(-1))& \mbox{$u$ is odd}.
\end{cases}
\end{align}
Also, we
 define the axes $Z(u)$  ($u\in \mathbb{Z}$)
by
\begin{align}
\label{eq:Zu1}
Z(u)=
\begin{cases}
\Sigma^{(u/2)r^{(2)}}(Z(0))& \mbox{$u$ is even}\\
\Sigma^{((u+1)/2)r^{(2)}}(Z(-1))& \mbox{$u$ is odd}.
\end{cases}
\end{align}
Thus, we have
\begin{align}
\label{eq:SZ1}
S(u+2)=\nu(S(u)),
\quad
Z(u+2)=\Sigma^{r^{(2)}}(Z(u)),
\end{align}

Now we  define a mutation sequence
for  $\mathbb{Y}_{\mathrm{RSG}}(n_1,\dots,n_F)$,
\begin{align}
\label{eq:RSGmseq}
\cdots
\buildrel {S(-2)}  \over{\longleftrightarrow}
\Gamma(-1)
\buildrel {S(-1)} \over{\longleftrightarrow}
\Gamma(0)
\buildrel {S(0)} \over{\longleftrightarrow}
\Gamma(1)
\buildrel {S(1)} \over{\longleftrightarrow}
\Gamma(2)
\buildrel  {S(2)} \over{\longleftrightarrow}
\cdots.
\end{align}

We have the following desired properties of
the mutation sequence \eqref{eq:RSGmseq}.

\begin{prop}[Reflection/rotation of triangulations]
\label{prop:ref1}
 For any $u\in \mathbb{Z}$,
 the following holds:
 \par
 (a). The set $S(u)$ is mutation-compatible in $\Gamma(u)$.
 \par
(b). The forward mutation at $u$ 
in \eqref{eq:RSGmseq} is the
reflection of $\Gamma(u)$ with respect to the
axis $Z(u)$, without moving the labels.
\par
(c).
As labeled triangulations, we have
\begin{align}
\label{eq:periodRSG1}
 \Gamma(u+2)=\Sigma^{r^{(2)}}(\nu(\Gamma(u))).
 \end{align}
\end{prop}
\begin{proof}
This is true for $u=-1,0$
by Theorem \ref{thm:qs} and Proposition \ref{prop:compatible}.
Then, one can prove it by  induction on $u$
in  both directions.
\end{proof}

\subsection{Patterns of forward mutation points}
\label{subsec:patterns}
Before working on the $Y$-system,
we establish some {\em patterns\/}  of
the forward mutations.

Let $T^{(a)}_1$ be the type of the (joint-)interval
in Lemma \ref{lem:mid1}
having
 $Q(0)$ as   midpoint at $u=0$.
Similarly,
let $T^{(a)}_2$ and  $T^{(a)}_3$  be the  mirror images
of the types of the (joint-)intervals
in Lemma \ref{lem:mid1}
having
 $P(-1)$  and $Q(-1)$ as  midpoints, respectively,
 where the mirror images of types
 are defined by
\eqref{eq:mirror1}.
We take the mirror images
for $P(-1)$  and $Q(-1)$,
 because we are interested 
in the {\em forward\/} mutation at $u=-1$,
instead of the {\em backward\/} mutation at $u=0$.

\begin{defn}
For the triangulation $\Gamma_{\mathrm{RSG}}(n_1,\dots,n_F)$
and $a=2,\dots,F$,
we call the triplet $\mathcal{X}_a=(T^{(a)}_1,T^{(a)}_2,T^{(a)}_3)$
the {\em pattern of  forward mutation of the $a$-th
generation\/}.
\end{defn}

For example, we have $\mathcal{X}_2=(N_R, L, (N_L,R))$
for even $n_1$ and $\mathcal{X}_2=((N_L,R), L, N_R)$
for odd $n_1$ as in the proof of 
Lemma \ref{lem:mid1}.
The pattern $\mathcal{X}_{a+1}$ can be
computed  by applying the  rule in
Lemma \ref{lem:change1} to  $\mathcal{X}_{a}$ termwise;
we show the possible patterns explicitly
 up to the fourth generation
in Figure \ref{fig:type1}.

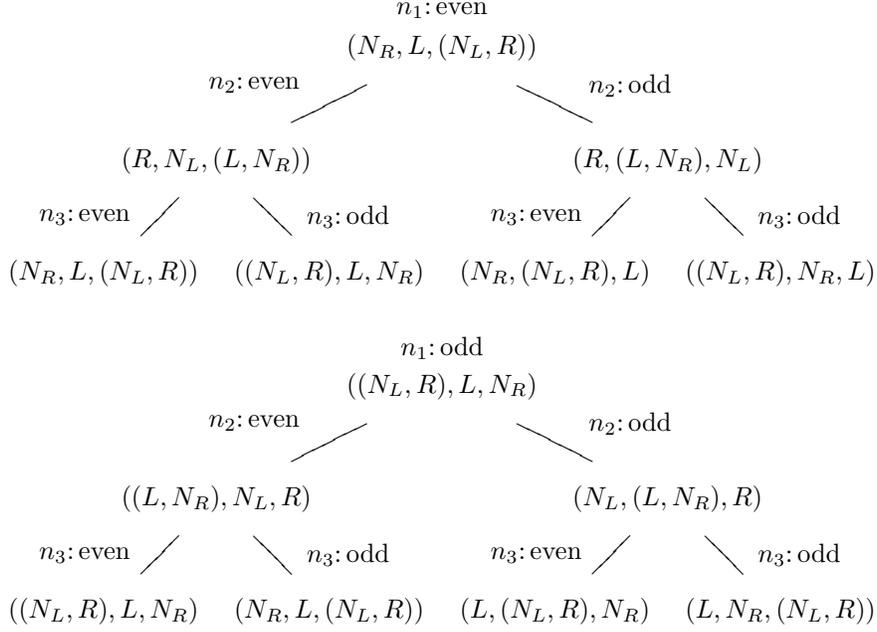
\begin{figure}
$$
\begin{xy}
(0,0)*{(N_R,L,(N_L,R))},
(-30,-15)*{(R,N_L,(L,N_R))},
(30,-15)*{(R,(L,N_R),N_L)},
(-45,-30)*{(N_R,L,(N_L,R))},
(-15,-30)*{((N_L,R),L,N_R)},
(15,-30)*{(N_R,(N_L,R),L)},
(45,-30)*{((N_L,R),N_R,L)},
(0,-45)*{((N_L,R),L,N_R)},
(-30,-60)*{((L,N_R),N_L,R)},
(30,-60)*{(N_L,(L,N_R),R)},
(-45,-75)*{((N_L,R),L,N_R)},
(-15,-75)*{(N_R,L,(N_L,R))},
(15,-75)*{(L,(N_L,R),N_R)},
(45,-75)*{(L,N_R,(N_L,R))},
(0,5)*{\mbox{$n_1$:\,even}},
(-25,-5)*{\mbox{$n_2$:\,even}},
(25,-5)*{\mbox{$n_2$:\,odd}},
(-47.5,-22.5)*{\mbox{$n_3$:\,even}},
(-12.5,-22.5)*{\mbox{$n_3$:\,odd}},
(12.5,-22.5)*{\mbox{$n_3$:\,even}},
(47.5,-22.5)*{\mbox{$n_3$:\,odd}},
(0,-40)*{\mbox{$n_1$:\,odd}},
(-25,-50)*{\mbox{$n_2$:\,even}},
(25,-50)*{\mbox{$n_2$:\,odd}},
(-47.5,-67.5)*{\mbox{$n_3$:\,even}},
(-12.5,-67.5)*{\mbox{$n_3$:\,odd}},
(12.5,-67.5)*{\mbox{$n_3$:\,even}},
(47.5,-67.5)*{\mbox{$n_3$:\,odd}},
\ar@{-} (-10,-5);(-20,-10)
\ar@{-} (10,-5);(20,-10)
\ar@{-} (-35,-20);(-40,-25)
\ar@{-} (-25,-20);(-20,-25)
\ar@{-} (25,-20);(20,-25)
\ar@{-} (35,-20);(40,-25)
\ar@{-} (-10,-50);(-20,-55)
\ar@{-} (10,-50);(20,-55)
\ar@{-} (-35,-65);(-40,-70)
\ar@{-} (-25,-65);(-20,-70)
\ar@{-} (25,-65);(20,-70)
\ar@{-} (35,-65);(40,-70)
%
%
\end{xy}
$$
\caption{Patterns of forward mutation up to the fourth generation.}
\label{fig:type1}
\end{figure}

Looking at Figure \ref{fig:type1},
we have the following observation.

\begin{lem}
\label{lem:parity}
Let $a=2,\dots, F$. For even $a$, the pattern $\mathcal{X}_a$
is a permutation of
$N_R$, $L$, $(N_L,R)$,
while
for odd $a$, the pattern $\mathcal{X}_a$
is a permutation of
$N_L$, $R$, $(L,N_R)$.
\end{lem}
\begin{proof}
This is an immediate consequence  
of the pattern of the second generation and
the rule in Lemma \ref{lem:change1}.
\end{proof}

We say that the pattern $\mathcal{X}_a$
is {\em of type I, II, III\/} if $N_L$ or $N_R$ is in the
first, the second, the third position in $\mathcal{X}_a$, respectively.
For example, $(N_R,L,(N_L,R))$ is of type I
and $((N_L,R),L,N_R)$ is of type III.

The following condition is useful in analyzing the $Y$-system.

\begin{lem}
\label{lem:pp1}
(a). The pattern $\mathcal{X}_a$ is of type I
if and only if $p_a$ is even.
(In this case, $p^{(2)}_a$ is odd,
because $p_a$ and $p^{(2)}_a$ are coprime due to
Proposition \ref{prop:cf1}.)
\par
(b). The pattern $\mathcal{X}_a$ is of type II
if and only if $p_a$ is odd and $p^{(2)}_a$ is even.
\par
(c). The pattern $\mathcal{X}_a$ is of type III
if and only if $p_a$ is odd and $p^{(2)}_a$ is odd.
\end{lem}
\begin{proof}
One can directly check the claim for $a=2$ and 3.
Assume that the claim is true up to $a$.
Suppose that $\mathcal{X}_{a+1}$ is of type I.
We have four possibilities for the parity of $n_{a}$ and $n_{a-1}$.
For example, suppose that $n_a$ and $n_{a-1}$ are both even.
Due to Lemma \ref{lem:change1}.
this implies that $\mathcal{X}_{a}$ is not of type I,
and $\mathcal{X}_{a-1}$ is of type I.
Then, $p_{a+1}=n_a p_{a} + p_{a-1}$ is even,
thanks to the induction hypothesis.
The other three cases can be checked in a similar way.
This proves  the only-if-part of (a).
Similarly, one can prove the only-if-parts of (b) and (c).
Since the conditions of (a), (b), (c) are mutually exclusive,
the if-parts of (a), (b), (c) also follow.
\end{proof}

\subsection{Anatomy of RSG $Y$-systems}
\label{subsec:anatomy}

As we did in Section \ref{sec:643},
for each time $u$,
we identify the $y$-variables $y^{(1)}_m(u)$ with
$Y^{(1)}_m(u)$,
and $y^{(a)}_{m,s}(u)$ with
$Y^{(a)}_m(u)$ for $a\geq 2$ regardless of $s$,
only at  forward mutation points.
We say that a $Y$-variable 
$Y^{(a)}_m(u)$  {\em occurs\/} 
in the mutation sequence \eqref{eq:RSGmseq}
if it is one of such identified variables.

\begin{prop}
\label{prop:cont1}
 Let $a=2,\dots, F-1$.
In the mutation sequence
\eqref{eq:RSGmseq}, the following holds.
\par
(a). 
$Y^{(a+1)}_1(u)$ occurs  if and only if $Y^{(a)}_{n_a-1}(u)$ occurs.
\par
(b). $Y^{(a+1)}_2(u)$ occurs  if and only if $Y^{(a-1)}_{n_{a-1}-
2\delta_{a-1,1}}(u)$ occurs.
\end{prop}
\begin{proof}
It is enough to prove it for $u=-1,0$ by the rotation property.
\par
(a). 
Suppose that $Y^{(a+1)}_1(u)$ occurs, for example,
for $u=-1$.
By Lemma \ref{lem:mset} (c),
there is some $s$ such that
the diagonal with  label $(a+1,1)_s$ belongs to $J^{(a+1)}_3$.
By Lemma \ref{lem:change1},
$J^{(a+1)}_3$ is in $J^{(a)}_3$ (resp. $J^{(a)}_2$)
if $n_a$ is even (resp. odd).
Therefore,
there is some $s'$ such that
the diagonal with  label $(a,n_a-1)_{s'}$ belongs to  $J^{(a)}_3$
(resp. $J^{(a)}_2$)
if $n_a$ is even (reap. odd).
Then, again by Lemma \ref{lem:mset} (c),
we conclude that $Y^{(a)}_{n_a-1}(u)$ occurs for $u=-1$.
The converse is also shown by reversing the argument.
\par
(b). This is shown similarly   using
 Lemma \ref{lem:change1}
and Lemma \ref{lem:mset} (b) and (c).
\end{proof}

\begin{prop}
\label{prop:bisect3}
 In the mutation sequence \eqref{eq:RSGmseq}, 
$Y^{(a)}_m(u)$ occurs if and only if
$Y^{(a)}_m(u)\in \mathcal{Y}_+$.
\end{prop}

\begin{proof} It is enough to show it for $u=-1$ and $0$
by the rotation property.
 
 {\em The only-if part.} Suppose that $Y^{(a)}_m(u)$ occurs at $u=-1$ or $0$.
 Then, we need to show that $\theta^{(a)}_m(u)$ in
 \eqref{eq:bi1} is even.
  Let us prove it by induction on $a$.
 This is true for $a=1$ by Lemma \ref{lem:mset} (b).
 It is also true for $a=2$ by the case check of 
 Figure \ref{fig:firstgen}.
 Suppose that it is true for $a$.
 Then, by Proposition \ref{prop:cont1},
 $\theta^{(a+1)}_1(u)=u+p_a+p_{a+1}=\theta^{(a)}_{n_a-1}(u)+2p_a$ is even.
 Similarly,
 $\theta^{(a+1)}_2(u)=u+p_a+2p_{a+1}=\theta^{(a-1)}_{n_{a-1}-2\delta_{a-1,1}}(u)
 + 2p_{a+1}$ is even.
 So, $Y^{(a+1)}_1(u), Y^{(a+1)}_2(u)\in \mathcal{Y}_+$.
 Then, by Lemma \ref{lem:mset} (c), 
 we have $Y^{(a+1)}_m(u)\in \mathcal{Y}_+$ for any $m$.
 \par
  {\em The if part.} By Lemma  \ref{lem:mset} (d),
  the number of $Y^{(a)}_m(u)$'s occurring at $u=-1,0$
  is equal to the number of the elements in $\mathcal{Y}_+$
   with $u=-1,0$.
   Thus,  $Y^{(a)}_m(u)$'s occurring at $u=-1$ or $0$
    exhaust the elements of $\mathcal{Y}_+$ with $u=-1$ or $0$. 
 \end{proof}

Now we are ready to state the fundamental theorem of the paper.

\begin{thm}
\label{thm:RSGY}
The mutation sequence \eqref{eq:RSGmseq} realizes the
RSG $Y$-system $\mathbb{Y}_{\mathrm{RSG}}(n_1,\dots,n_F)$
for $\mathcal{Y}_+$.
\end{thm}

The rest of this subsection will be devoted to prove
Theorem \ref{thm:RSGY}.
We start by establishing the three basic properties
on the mutation sequence.

The first property explains the left hand sides of the equations
of the $Y$-systems.

\begin{prop}
\label{prop:pp1}
Any label for any generation $a$ is
a forward mutation point at some $u\in \mathbb{Z}$.
Furthermore, such $u$'s occur exactly with period $2p_a$.
\end{prop}
\begin{proof}
This is true for $a=1$ by Figure \ref{fig:firstgen}.
Let $a\geq 2$.
It follows from Lemma \ref{lem:mset} (c) that,
for a given $m=1,\dots, n_a$,
there exists the unique $s\in \{1,\dots, p_a\}$
such that $(a,m)_s$ is a forward mutation point at $u=-1$ or $0$.
Then, by Proposition \ref{prop:int2},
$s'=s+p_a^{(2)}$ is the unique $s'\in \{1,\dots, p_a\}$
such that
$(a,m)_{s'}$ is a forward mutation point at $u=1$ or $2$.
Repeating this,
we see that
 every label
$(a,m)_{s''}$ ($s''=1,\dots,p_a$) 
appears as a forward mutation point exactly
with period $2p_a$,
 since $p_a$ and $p^{(2)}_a$ are coprime by Proposition
\ref{prop:cf1} (c),
\end{proof}

Thanks to Lemma \ref{lem:pp1},
we have the second basic property.

\begin{prop}
\label{prop:pp2}
Let $a=2,\dots,F$ and $s=1,\dots,p_a$.
Suppose that the labels $(a,m)_s$ for even $m$ (resp. odd $m$)
are
the forward mutation points at $u$.
Then, the labels $(a,m)_s$ for odd $m$ (resp. even $m$)
are the forward mutation points at $u+p_a$.
\end{prop}

\begin{figure}
		\includegraphics[scale=.3]{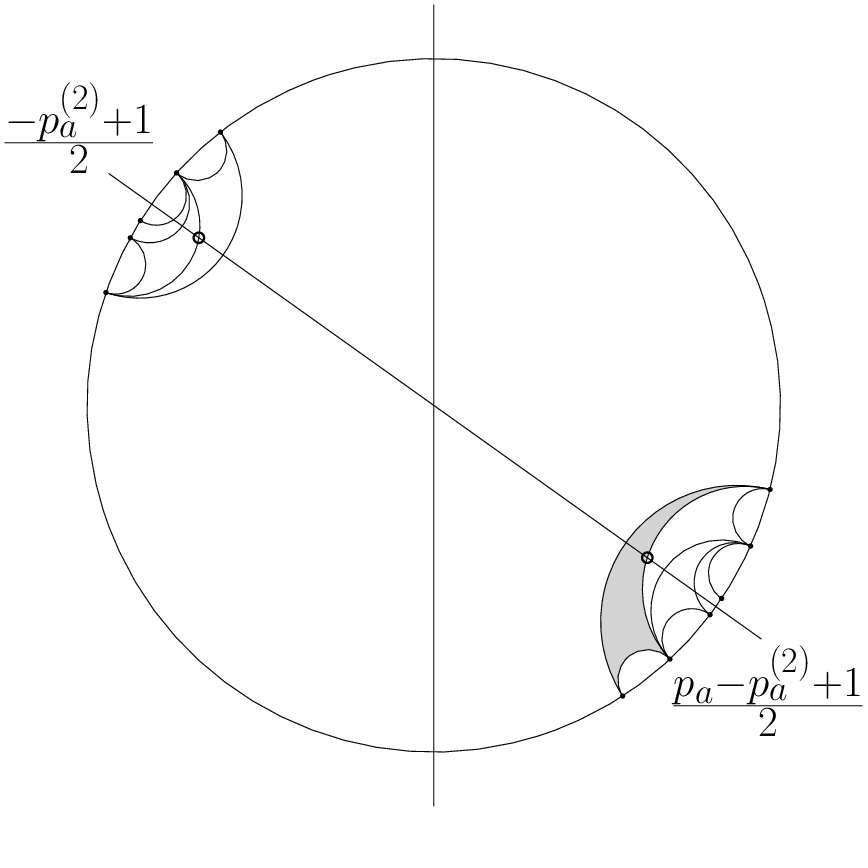}
		\hskip-12pt
		\includegraphics[scale=.3]{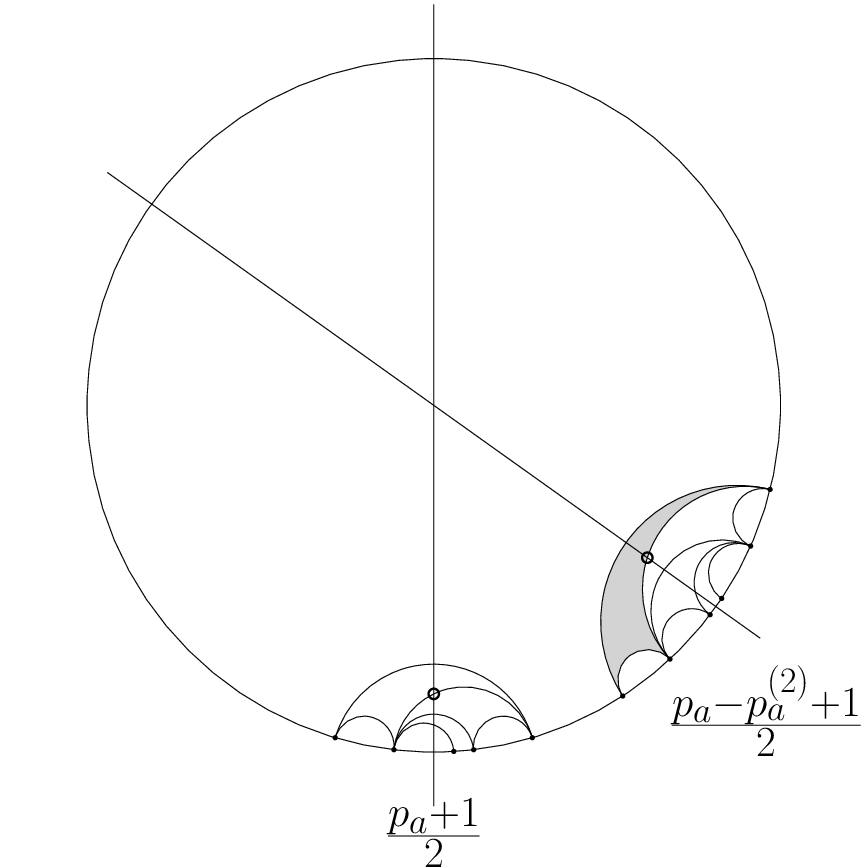}
		\hskip-12pt
		\includegraphics[scale=.3]{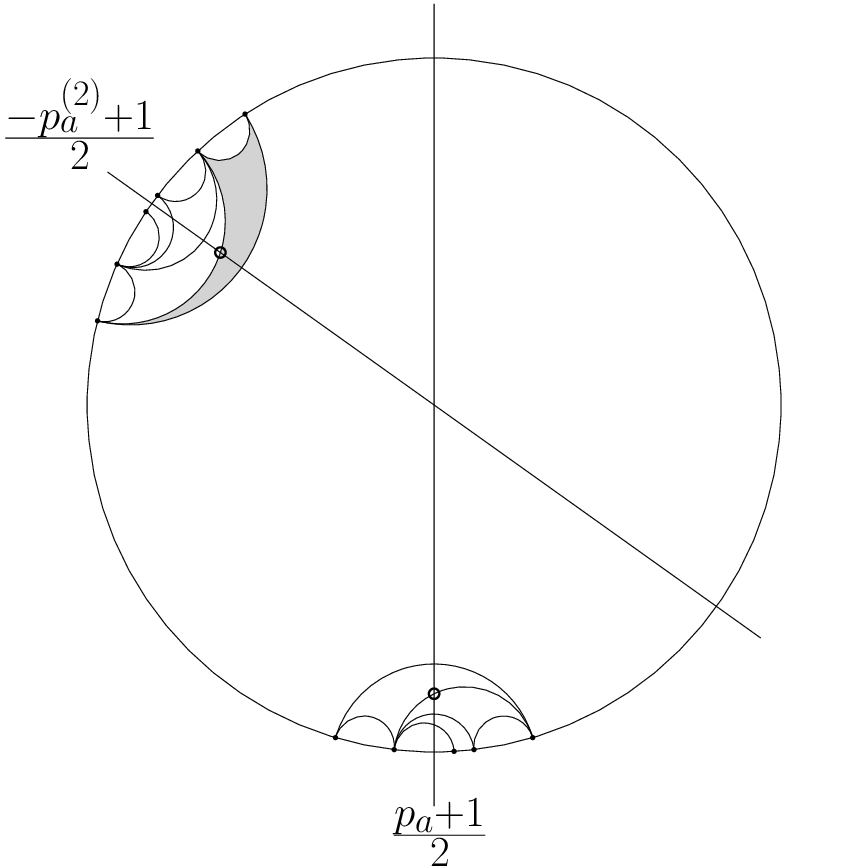}
\centerline{
(a) type I
\hskip70pt
(b) type II
\hskip70pt
(c) type III
}
\caption{Patterns of forward mutation.
}
\label{fig:fm1}
\end{figure}

\begin{proof}
By the rotation property, it is enough to prove it
for $u=-1$ or $0$.
Assume that $a$ is odd. The other case is similar.
We prove the claim depending on the type of the pattern $\mathcal{X}_a$
in Section \ref{subsec:patterns}.
\par

(a). {\em Type I}.
By Lemma \ref{lem:pp1},
$p_a$ is even and $p^{(2)}_a$ is odd.
Assume that $\mathcal{X}_a=(N_L, R, (L,N_R))$.
(The other case is similar.)
For odd $m$,
the label $(a,m)_{(p_a-p^{(2)}_a+1)/2}$ is a forward mutation point at $u=-1$.
For even $m$,
the label $(a,m)_{(-p^{(2)}_a+1)/2}$ is a forward mutation point at $u=-1$.
See Figure \ref{fig:fm1} (a).
Since $p_a(p^{(2)}_a-1)/2\equiv 0$ mod $p_a$,
we have
\begin{align}
 \frac{p_a}{2}  p^{(2)}_a \equiv \frac{p_a}{2}
 \mod p_a.
\end{align}
It follows that,
for odd $m$,
the label $(a,m)_{(-p^{(2)}_a+1)/2}$ is a forward mutation point at $u=p_a-1$,
and that for even $m$,
the label $(a,m)_{(p_a-p^{(2)}_a+1)/2}$ is a forward mutation point at $u=p_a-1$.
\par

(b). {\em Type II}.
By Lemma \ref{lem:pp1},
 $p_a$ is odd and $p^{(2)}_a$ is even.
Assume that $\mathcal{X}_a=(R,N_L,  (L,N_R))$.
(The other case is similar.)
For odd $m$,
the label $(a,m)_{(p_a-p^{(2)}_a+1)/2}$ is a forward mutation point at $u=-1$.
For even $m$,
the label $(a,m)_{(p_a+1)/2}$ is a forward mutation point at $u=0$.
See Figure \ref{fig:fm1} (b).
Since $(p_a\pm1)p^{(2)}_a/2\equiv \pm p^{(2)}_a/2$
mod $p_a$,
we have
\begin{align}
 \frac{p_a\pm1}{2}  p^{(2)}_a \equiv \pm\frac{p^{(2)}_a}{2}\quad
  \mod p_a.
\end{align}
It follows that,
for odd $m$,
the label $(a,m)_{(p_a+1)/2}$ is a forward mutation point at $u=p_a$,
and that
for even $m$,
the label $(a,m)_{(p_a-p^{(2)}_a+1)/2}$ is a forward mutation point at $u=p_a-1$.

\par
(c). {\em Type III}.
By Lemma \ref{lem:pp1},
 $p_a$ is odd and $p^{(2)}_a$ is odd.
Assume that $\mathcal{X}_a=(R, (L,N_R),N_L)$.
(The other case is similar.)
For odd $m$,
the label $(a,m)_{(-p^{(2)}_a+1)/2}$ is a forward mutation point at $u=-1$.
For even $m$,
the label $(a,m)_{(p_a+1)/2}$ is a forward mutation point at $u=0$.
See Figure \ref{fig:fm1} (c).
Since $(p_a\pm1)(p^{(2)}_a-1)/2\equiv \pm (p^{(2)}_a-1)/2$ mod $p_a$,
we have
\begin{align}
 \frac{p_a\pm 1}{2}  p^{(2)}_a \equiv  \frac{p_a\pm p^{(2)}_a}{2}
   \mod p_a.
\end{align}
It follows that,
for odd $m$,
the label $(a,m)_{(p_a+1)/2}$ is a forward mutation point at $u=p_a$,
and that
for even $m$,
the label $(a,m)_{(-p^{(2)}_a+1)/2}$ is a forward mutation point at $u=p_a-1$.
\end{proof}

The third property is as follows.

\begin{prop}
\label{prop:pp3}
 Let $a=2,\dots,F$.
 Suppose that a label $(a,m)_s$ is a forward mutation
 point at $u$. Then, the label $(a,m)_{s+1}$
 is a forward mutation point at $u+(-1)^a 2p_{a-1}$.
 \end{prop}
\begin{proof}
The claim is equivalent
to the equality
\begin{align}
\label{eq:pp2}
p_{a-1}p^{(2)}_{a}
\equiv (-1)^a
\mod p_a,
\quad
a=2,\dots,F.
\end{align}
For $a=2$, this is trivially true.
For $a=3,\dots, F$, this is an immediate corollary of \eqref{eq:pp0}.
 \end{proof}

\begin{ex}
In the example $(6,4,3)$ in
Section \ref{sec:643},
where $p_1=1$ and $p_2=6$,
the cases $a=2$ and $3$ in Proposition \ref{prop:pp3}
have been observed
in the snapshots in Figure \ref{fig:31gon-snap} and
\ref{fig:106gon-snap}.
\end{ex}

Now we will prove 
Theorem \ref{thm:RSGY}.
For the $Y$-variables occurring
in the  mutation sequence \eqref{eq:RSGmseq},
we will check the relations 
\eqref{eq:RSG1}, \eqref{eq:RSG5}, \eqref{eq:RSG4},
according to the generation $a$ in their left hand sides. 
To do that, we use 
the snapshot method introduced in Section \ref{sec:643}.

For \eqref{eq:RSG1} with $a$=1 and for \eqref{eq:RSG5},
this can be 
easily done  
with Figure \ref{fig:firstgen}.
The case  $n_1$  even
 was already treated
in Figure \ref{fig:31gon-snap};
the other cases can be done in a similar manner.
So, we omit repeating them,
and we concentrate on 
\eqref{eq:RSG1} for $a\geq 2$ and
\eqref{eq:RSG4}.

Having Propositions
\ref{prop:pp1}, \ref{prop:pp2},
\ref{prop:pp3}
and Lemmas \ref{lem:change1},
\ref{lem:int}, we are able to obtain all   necessary snapshots.
In the proof, we can assume that {\em $a$ is odd},
without loss of generality.
Indeed,
by
Lemma \ref{lem:parity} and the factor $(-1)^a$
in Proposition \ref{prop:pp3},
the snapshot
in the  even case is given by the
mirror image of the snapshot in the odd case.
This difference is absorbed by
 the signs $\varepsilon_b$
 in \eqref{eq:RSG1} and \eqref{eq:RSG4}.

\begin{prop}
\label{prop:Y3a}
The $Y$-variables occurring
in the  mutation sequence \eqref{eq:RSGmseq}
satisfy the relation \eqref{eq:RSG1}
for $a\geq 2$.
\end{prop}

\begin{figure}
		\includegraphics[scale=.45]{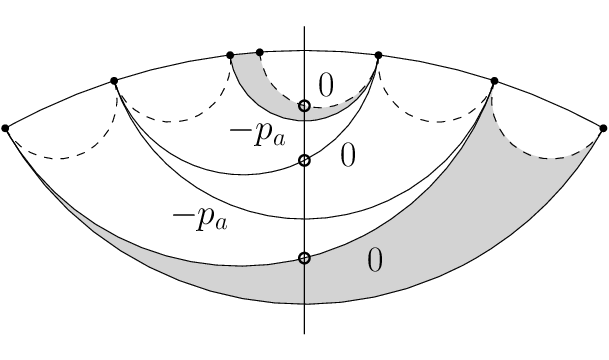}
		\hskip10pt
		\includegraphics[scale=.45]{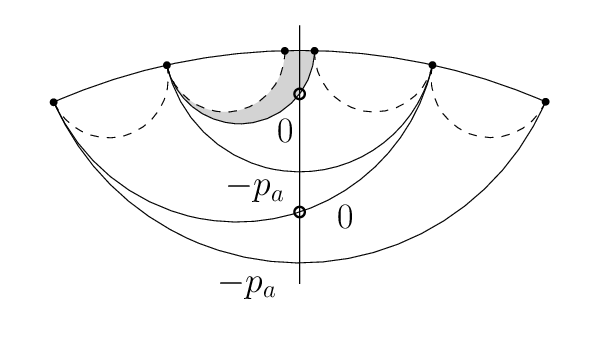}
		\centerline{\hskip3pt (a) \hskip130pt (b)}
\caption{Snapshots  reduced at $u=0$ for the relation \eqref{eq:RSG1}.
(a).  Joint-interval
 of type $(L,N_R)$ of the $a$-th generation.
 (b). Interval
 of type $R$ of the $a$-th generation.
}
\label{fig:snap1}
\end{figure}

\begin{proof}
Assume that $a$ is odd, as mentioned above.
We further assume that $n_{a}$ is even.
The other case is similar.
In the snapshot method, only the {\em differences\/}
of data are relevant.
So, we may think that the mutation occurs
at $u=0$,
even if it actually occurs only at odd $u$.
We call it the {\em snapshot reduced at $u=0$}.
Then, using Propositions \ref{prop:pp1}, \ref{prop:pp2},
and 
Lemmas \ref{lem:change1}, \ref{lem:int}, 
we can write the snapshots reduced at $u=0$ relevant to \eqref{eq:RSG1}
as Figure  \ref{fig:snap1}.
The relation \eqref{eq:RSG1}  is then easily confirmed by
inspecting Figure  \ref{fig:snap1}.
\end{proof}

Finally, we clarify the most mysterious part of the
RSG $Y$-system.

\begin{prop}
\label{prop:Y3b}
The $Y$-variables occurring
in the  mutation sequence \eqref{eq:RSGmseq}
satisfy the relation \eqref{eq:RSG4}.
\end{prop}

\begin{figure}
		\includegraphics[scale=.58]{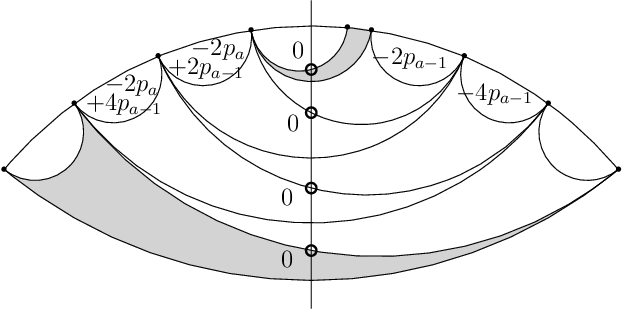}
		\hskip20pt
		\includegraphics[scale=.58]{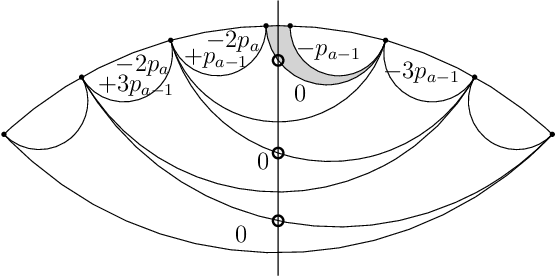}
		\centerline{\hskip10pt (a)\hskip175pt(b)}
		\vskip0pt
		\includegraphics[scale=.58]{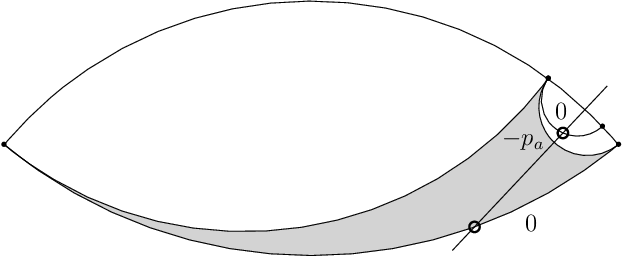}
		\centerline{(c)}
\caption{Snapshots reduced at $u=0$ for the relation \eqref{eq:RSG4}.
(a).  Joint-interval
 of type $(N_L,R)$ of the $(a-1)$-th generation.
 (b). Interval
 of type $L$ of the $(a-1)$-th generation.
 (c). Interval
 of type $N_R$ of the $(a-1)$-th generation.
}
\label{fig:snap2}
\end{figure}

\begin{proof}
Assume that $a$ is odd, as mentioned above.
We further assume that $n_{a-1}$ is even.
The opposite case is similar.
By Propositions \ref{prop:pp1}, \ref{prop:pp2}, \ref{prop:pp3}
and 
Lemmas \ref{lem:change1}, \ref{lem:int}, 
 we have 
the snapshots reduced at $u=0$ relevant to \eqref{eq:RSG4}
in Figure  \ref{fig:snap2}.
Then, the relation \eqref{eq:RSG4}  is easily confirmed by 
inspecting Figure  \ref{fig:snap2}.
\end{proof}

This completes the proof of 
Theorem \ref{thm:RSGY}.

\subsection{Proof of Theorem \ref{thm:periodRSG}}
\label{subsec:periodRSG}
As a corollary of  Proposition \ref{prop:ref1} and 
Theorem \ref{thm:RSGY},
we obtain the periodicity of Theorem \ref{thm:periodRSG}
by the same argument as in Section \ref{sec:643}.
Furthermore,
thanks to the coprime property of 
$r$ and $r^{(2)}$ in Proposition \ref{prop:cf1} (d),
the period $2r$ is the minimal one.

\subsection{Time-ordered index and solution by Gliozzi and Tateo}
\label{subsec:GT}

Gliozzi and Tateo obtained
a general solution of any RSG $Y$-system
in terms of  cross-ratio of $r$ points  \cite{Gliozzi96}.
They were guided by  considerations on the decompositions of a
certain {\em threefold\/},
but no  systematic derivation of the solution was provided.
Here we derive their solution from our formulation of the Y-system.

The key to interpret their solution
is  the introduction of another indexing of
the vertices of $\Gamma(u)$ for 
$\mathbb{Y}_{\mathrm{RSG}}(n_1,\dots,n_F)$.
We put  the integer $t=0$, \dots, $r-1$ at the $tr^{(2)}$-th vertex of 
$\Gamma(u)$
in our standard ordering.
(Here and below ``$m$-th''  means in our standard ordering.)
Since $r$ and $r^{(2)}$ are coprime,
that gives a new index of the vertices of the $r$-gon.
We call it the {\em time-ordered index\/} of $\Gamma(u)$.
See Figure \ref{fig:time} for an example.

\begin{figure}
	\begin{center}
		\includegraphics[scale=.42]{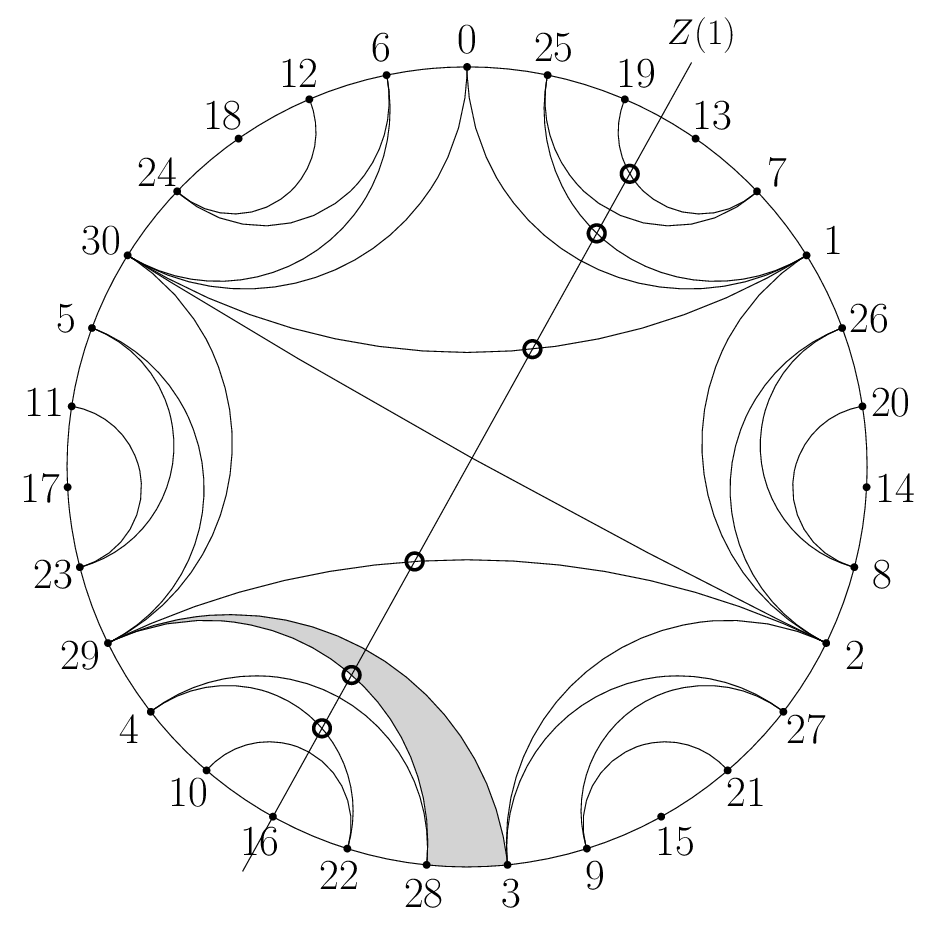}
	\end{center}
\caption{Time-ordered index of $\Gamma(1)$ for
$\mathbb{Y}_{\mathrm{RSG}}(6,4)$.}
\label{fig:time}
\end{figure}

By definition, if the $k$-th vertex has the time-ordered index $t$,
then the $(k+r^{(2)})$-th vertex has the time-ordered index $t+1$.
This can be generalized in the following way. 

\begin{prop}
\label{prop:time1}
For $a=1,\dots, F$, if the $k$-th vertex has the time-ordered index $t$,
then the $(k+r^{(a+1)})$-th vertex has the time-ordered index $t+
(-1)^{a-1} p_{a}$. 
\end{prop}
\begin{proof}
This is equivalent to the equality
\begin{align}
p_ar^{(2)} \equiv (-1)^{a-1}r^{(a+1)}
\mod r.
\end{align}
The equality is true for $a=1$, since $p_1=1$.
It is also true for $a=2$, since
$p_2r^{(2)} \equiv n_1 r^{(2)}
\equiv  -r^{(3)}$,
where we used \eqref{eq:rr4} with $k=1$.
Now suppose that the claim is true up to $a$.
Then,
\begin{align*}
\begin{split}
p_{a+1}r^{(2)}
& \equiv (n_{a}p_{a} +p_{a-1})r^{(2)}
\equiv  (-1)^{a-1}n_ar^{(a+1)} +  (-1)^{a}r^{(a)}
\equiv  (-1)^{a}r^{(a+2)},
\end{split}
\end{align*}
where we used \eqref{eq:rr4} with $k=a$ in the last equality.
\end{proof}

In particular, setting $a=F$, we see that
the difference of the time-ordered indices of  two adjacent
vertices is $p_F$, as we observe in Figure \ref{fig:time}.

The next proposition justifies the name of the time-ordered index.
Roughly speaking,
the  time-ordered index $t$ of each of the two 
ends of the ``hour hand'' $Z(u)$
is always {\em a half of the standard time $u$} mod $r$.

\begin{prop} 
\label{prop:center1}
Let $Z(u)$ be the axis in \eqref{eq:Zu1}.
 Let $P$ be any of the two points where $Z(u)$ intersects  the boundary
 of the $r$-gon. ($P$ is either a vertex, or a midpoint of two
 adjacent vertices of the $r$-gon.)
  Then, we have the following.
 \par
 (a). In  case $P$ is a vertex of the $r$-gon, 
let $t$  the time-ordered index
 of $P$. Then, $u\equiv 2t$ mod $r$.
 \par
 (b). In  case  $P$ is the midpoint of the two adjacent vertices
 $Q_1$ and $Q_2$,
 let $t$ be the average of
 the time-ordered indices of $Q$ and $Q'$.
 Then, $u\equiv 2t$ mod $r$.
 \end{prop}

\begin{proof}
By the rotation property of the time-ordered index,
it is enough to show the claim for $u=0$ and $1$.
Thanks to the remark after
Proposition \ref{prop:time1},
the sum of the time-ordered indices of any pair of vertices
in symmetric position with respect to $Z(u)$ is the same.
The sum is 0 for $u=0$, and $1$ for $u=1$,
since the  time-ordered index of the $0$-th and the $r^{(2)}$-th
vertices are $0$ and $1$, respectively.
This proves the claim.
\end{proof}

We recall a general formula by Fock and Goncharov
 \cite[Section 4.1]{Fock05} (see also \cite{Fock98}) 
expressing the $y$-variables in terms of  cross-ratios,
which is applicable here.
We use the following definition of  cross-ratio,
which is suitable for our purpose:
\begin{align}
(\alpha,\beta,\gamma,\delta):=\frac{(\alpha-\delta)
(\beta-\gamma)}{(\alpha-\beta)(\gamma-\delta)}.
\end{align}
Note that
\begin{align}
\label{eq:cr1}
(\alpha,\delta,\gamma,\beta)
=(\alpha,\beta,\gamma,\delta)^{-1}.
\end{align}
In general, suppose
that, in a given triangulation of a polygon,
a diagonal  with label $i$ is surrounded by 
a quadrilateral in the following way.
$$
\begin{xy}
(10,13)*{\alpha},
(-4,0)*{\beta},
(10,-14)*{\gamma},
(23,0)*{\delta},
(10,-2)*{i},
\ar@{-} (0,0);(20,0)
\ar@{-} (20,0);(10,10)
\ar@{-} (10,10);(0,0)
\ar@{-} (10,-10);(20,0)
\ar@{-} (10,-10);(0,0)
\end{xy}
$$
Then, the corresponding $y$-variable $y_i$ can be 
represented by a cross-ratio as
\begin{align}
\label{eq:cr2}
y_i=(z(\alpha),z(\beta),z(\gamma),z(\delta)),
\end{align}
where $z(\alpha)$, \dots, $z(\delta)$ are formal variables
associated to the vertices $\alpha$, \dots, $\delta$.

The following formula gives a
complete and explicit description
of the diagonals for the forward mutation points at any $u$,
together with their surrounding quadrilaterals,
in terms of the time-ordered index.

\begin{figure}
	\begin{center}
		\includegraphics[scale=.31]{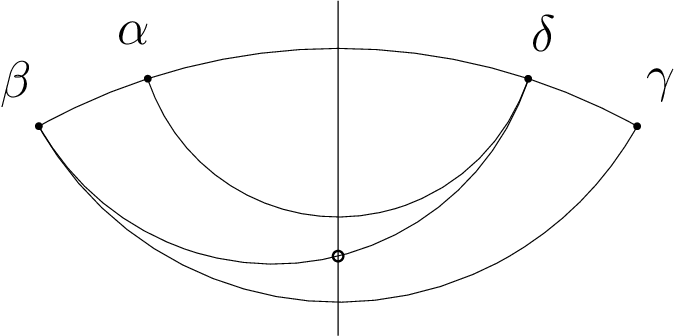}
		\hskip30pt
		\includegraphics[scale=.31]{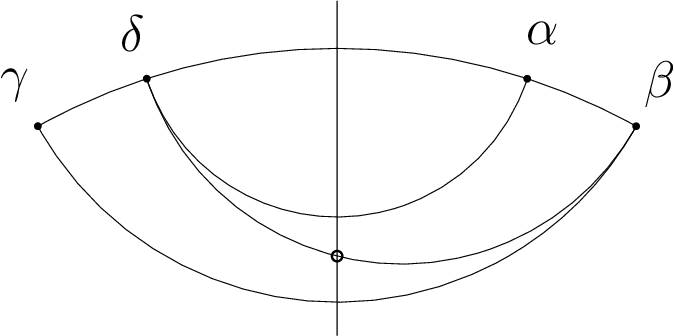}
	\end{center}
\centerline{(a) odd $a$ \hskip88pt (b) even $a$}
\caption{Positions of the vertices
$\alpha$, $\beta$, $\gamma$, $\delta$
in Proposition \ref{prop:quad1}.
}
\label{fig:quad1}
\end{figure}

\begin{prop}
\label{prop:quad1}
Let $Y^{(a)}_m(u)\in \mathcal{Y}_+$.
Let $Q$ be the unique quadrilateral in $\Gamma(u)$ surrounding
the diagonal
whose label is identified with $Y^{(a)}_m(u)$ at time $u$.
Let $\alpha$, $\beta$, $\gamma$, $\delta$
be the vertices of $Q$ whose positions
are specified by Figure
\ref{fig:quad1},
depending on the parity of $a$.
Then the time-ordered indices
$\alpha^{(a)}_m(u)$,
$\beta^{(a)}_m(u)$,
$\gamma^{(a)}_m(u)$,
$\delta^{(a)}_m(u)$ of
the vertices
 $\alpha$, $\beta$, $\gamma$, $\delta$
are given by the following formula
(they are integers since $Y^{(a)}_m(u)\in \mathcal{Y}_+$):
\begin{align}
\label{eq:alpha1}
\begin{split}
\alpha^{(1)}_m(u)&=\frac{1}{2}
(u+m+2),\quad
\beta^{(1)}_m(u)=\frac{1}{2}
(u+m),\\
\gamma^{(1)}_m(u)&=\frac{1}{2}
(u-m),\quad
\delta^{(1)}_m(u)=\frac{1}{2}
(u-m-2).
\end{split}
\end{align}
and, for $a=2,\dots,F$,
\begin{align}
\label{eq:alpha2}
\begin{split}
\alpha^{(a)}_m(u)&=\frac{1}{2}
(u+p_{a+1}-(n_a-m)p_a),\quad
\beta^{(a)}_m(u)=\frac{1}{2}
(u+p_{a+1}-(n_a+2-m)p_a),\\
\gamma^{(a)}_m(u)&=\frac{1}{2}
(u-p_{a+1}+(n_a+2-m)p_a),\quad
\delta^{(a)}_m(u)=\frac{1}{2}
(u-p_{a+1}+(n_a-m)p_a).
\end{split}
\end{align}
\end{prop}

\begin{proof}
By Proposition \ref{prop:time1}
and
the  familiar structure of
the forward mutation points in Figure \ref{fig:snap1},
 we have
the equalities
\begin{align}
\label{eq:shift1}
\begin{split}
\alpha^{(a)}_{m}(u)&=\beta^{(a)}_{m}(u)+p_a,
\quad
\delta^{(a)}_{m}(u)=\gamma^{(a)}_{m}(u)-p_a,\\
\alpha^{(a)}_{m+2}(u)&=\alpha^{(a)}_{m}(u)+p_a,
\quad
\beta^{(a)}_{m+2}(u)=\beta^{(a)}_{m}(u)+p_a,\\
\gamma^{(a)}_{m+2}(u)&=\gamma^{(a)}_{m}(u)-p_a,
\quad
\delta^{(a)}_{m+2}(u)=\delta^{(a)}_{m}(u)-p_a.
\end{split}
\end{align}
Meanwhile, the right hand sides of 
\eqref{eq:alpha1} and \eqref{eq:alpha2}
also satisfy the same equalities.
Therefore, for each $a$ it is enough to prove the claim
only for $m=1,2$.

We prove it by induction on $a$.
We first prove it for $a=1,2$.
By the rotation property, it is enough to prove it for $u=-1$ or $0$.
Then, the claim  is directly verified using
Figure \ref{fig:firstgen}.
Next suppose that the claim is true up to $a$,
and let us prove it for $a+1$.
Consider the case $a+1$ is odd. (The other case is done by 
considering the mirror image.)
Repeating the same analysis of Proposition \ref{prop:cont1}
we can verify
\begin{align}
\beta^{(a+1)}_1(u)&=\delta^{(a)}_{n_a-1}(u),
\quad
\gamma^{(a+1)}_1(u)=\alpha^{(a)}_{n_a-1}(u),\\
\beta^{(a+1)}_2(u)&=\alpha^{(a-1)}_{n_{a-1}-2\delta_{a-1,1}}(u),
\quad
\gamma^{(a+1)}_2(u)=\delta^{(a-1)}_{n_{a-1}-2\delta_{a-1,1}}(u).
\end{align}
The right hand sides of 
\eqref{eq:alpha1} and \eqref{eq:alpha2}
also satisfy the same equalities;
therefore, the claim  is true for $a+1$.
\end{proof}

Now we are able to recover the solution of \cite{Gliozzi96}
in our convention.
Let $z(0)$, \dots, $z(r-1)$ be 
formal variables such that $z(i+r)=z(i)$ ($i\in \mathbb{Z}$).
\begin{thm}[\cite{Gliozzi96}]
The $Y$-system $\mathbb{Y}_{\mathrm{RSG}}(n_1,\dots,n_F)$
for $\mathcal{Y}_+$
has a following general solution.
\begin{align}
\label{eq:ycross1}
Y^{(a)}_m(u)&=(z(\alpha^{(a)}_m(u)),
z(\beta^{(a)}_m(u)),
z(\gamma^{(a)}_m(u)),
z(\delta^{(a)}_m(u)))^{\varepsilon_a}.
\end{align}
\end{thm}

\begin{proof}
This is an immediate corollary of
Proposition \ref{prop:quad1},
\eqref{eq:cr1}, and \eqref{eq:cr2}.
\end{proof}

\subsection{Summary of  meaningful numbers in triangulations}
Here, we record the roles of numbers from continued fractions
in the triangulation $\Gamma_{\mathrm{RSG}}(n_1,\dots,n_F)$.

\begin{itemize}
\item $r$
\par\noindent
the number of  vertices
\item $r^{(2)}/2$
\par\noindent
the width between the axes $Z(-1)$ and $Z(0)$
\item $r^{(a)}$ ($a=2,\dots, F$)
\par\noindent
the width of an interval of type $L$ or $R$ of the $a$-th generation
\par\noindent
the width of an interval of type $N_L$ or $N_R$ of the $(a-1)$-th generation
\item $q_a=p_{a+1}$ ($a=1,\dots, F-1$)
\par\noindent
the  number of the intervals of types $L$ and $R$ of the $(a+1)$-th generation
\item $p_a$ ($a=1,\dots, F-1$)
\par\noindent
the  number of the intervals of types $N_L$ and $N_R$ of the $(a+1)$-th generation
\item $q^{(k)}_a=p^{(k)}_{a+1}$ ($k=2,\dots, F-1$; $a=k,\dots, F-1$)
\par\noindent
the  number of the intervals of types $L$ and $R$ of the $(a+1)$-th generation
inside an interval of type $L$ or $R$ of the $k$-th generation
\item  $p^{(k)}_a$ ($k=2,\dots, F-1$; $a=k,\dots, F-1$)
\par\noindent
the  number of the intervals of types $N_L$ and $N_R$ of the $(a+1)$-th generation
inside an interval of type $L$ or $R$ of the $k$-th generation
\end{itemize}

\section{Realization of SG $Y$-systems by  polygons
with a puncture}
\label{sec:SG}

Here we will construct triangulations
of polygons {\em with one puncture\/}
realizing  the SG $Y$-systems, in full generality,
and prove the periodicity of Theorem \ref{thm:periodRSG}.

\subsection{Examples}
\label{subsec:SGex}
It is known \cite{Fomin08} that a cluster algebra
of type $D_n$ can be realized by  {\em tagged\/} triangulations
of an $n$-gon with one puncture.
It turns out that the underlying cluster algebra for the
SG $Y$-system $\mathbb{Y}_{\mathrm{SG}}(n_1,\dots,n_F)$
is of type $D_{r}$, where we continue to use the notation
$r=r_F$. Therefore, it share the same $r$-gon with the
RSG $Y$-system parametrized by  the same data
and the only difference is the puncture
in the center.
We ask the reader to refer to \cite{Fomin08,Fomin08b}
for generalities on  tagged triangulations.
Here we only need a nominal use of them;
namely, we have one {\em notched\/} arc;
otherwise, all the other arcs are {\em plain\/} (i.e., ordinary) arcs.

The construction of the initial (tagged) triangulation
$\Gamma_{\mathrm{SG}}(n_1,\dots,n_F)$
 for
$\mathbb{Y}_{\mathrm{SG}}(n_1,\dots,n_F)$ is easy.
We only need to modify the diagonals of the {\em first\/} generation
of $\Gamma_{\mathrm{RSG}}(n_1,\dots,n_F)$.

We explain the idea by three examples,
 which are
the counterparts of those in Section \ref{sec:643}.

\begin{figure}
	\begin{center}
		\includegraphics[scale=0.37]{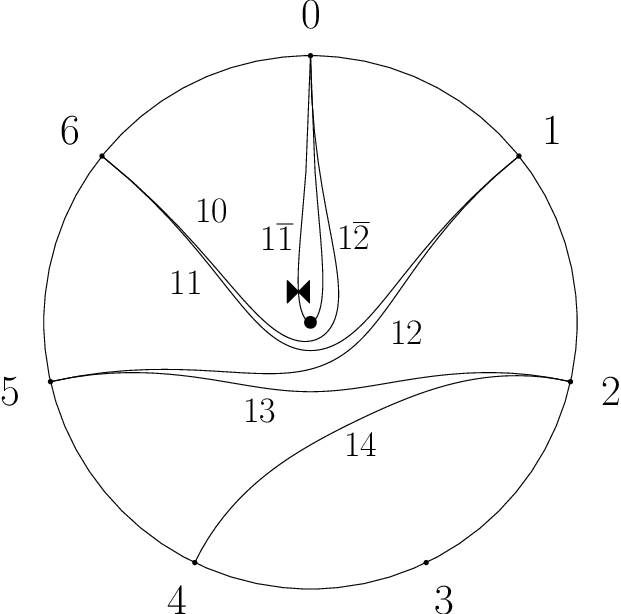}
	\end{center}
\caption{The initial triangulation $\Gamma_{\mathrm{SG}}(6)$ of a 7-gon
with a puncture in the center.
 The label $1\overline{2}$ is for the plain arc ending
at the puncture.}
\label{fig:d7gon}
\end{figure}

\medskip
{\bf Example 1.} $F=1$, $n_1=6$.
In this case, the Y-system 
$\mathbb{Y}_{\mathrm{SG}}(6)$
is nothing but the $Y$-system of type $D_7$,
and it is realized by a $7$-gon with one puncture.
We define the labeled triangulation $\Gamma_{\mathrm{SG}}(6)$ shown
in Figure \ref{fig:d7gon}.
Note that it is a natural extension of
$\Gamma_{\mathrm{RSG}}(6)$ as
in Figure \ref{fig:7gon}.
Besides the diagonals for $\Gamma_{\mathrm{RSG}}(6)$,
we have three new diagonals labeled with
 $(1,0)$, $(1,\overline{1})$,
$(1,\overline{2})$,
 and also the signs $-$, $+$, $+$ are attached to them, respectively.
In particular, $(1,\overline{1})$ and $(1,\overline{2})$
are
the labels for
notched and plain arcs ending at the puncture,
respectively.
The triangulation corresponds to the following alternating quiver
of type $D_7$.
$$
\begin{xy}
(-30,7.5)*\cir<2pt>{},
(-30,-7.5)*\cir<2pt>{},
(-15,0)*\cir<2pt>{},
(0,0)*\cir<2pt>{},
(15,0)*\cir<2pt>{},
(30,0)*\cir<2pt>{},
(45,0)*\cir<2pt>{},
(-37,7.5)*{(1,\overline{1})^+},
(-37,-7.5)*{(1,\overline{2})^+},
(-15,-5)*{(1,0)^-},
(0,-5)*{(1,1)^+},
(15,-5)*{(1,2)^-},
(30,-5)*{(1,3)^+},
(45,-5)*{(1,4)^-},
\ar@{->} (-29,7);(-16,1)
\ar@{->} (-29,-7);(-16,-1)
\ar@{->} (-1,0);(-14,0)
\ar@{->} (1,0);(14,0)
\ar@{<-} (16,0);(29,0)
\ar@{->} (31,0);(44,0)
\end{xy}
$$
Setting
$\Gamma(0)=\Gamma_{\mathrm{SG}}(6)$,
 we apply the sequence of mutations \eqref{eq:7seq}.
The result is given in Figure \ref{fig:d7seq}.
As in the RSG case, one can easily check that
the sequence realizes the $Y$-system
$\mathbb{Y}_{\mathrm{SG}}(6)$ for $\mathcal{Y}_+$.
Continuing the sequence in Figure \ref{fig:d7seq} up to $u=14$, 
one observes that
it comes back to the original triangulation
{\em except that\/} the labels $(1,\overline{1})$ and  $(1,\overline{2})$
are interchanged.
This happens because, they swap every {\em two\/} time units of $u$,
and $r=7$ is odd.
Continuing it up to $u=28$,
we get back the initial triangulation.
Then, applying the same argument as for the RSG case,
we obtain the  periodicity of Theorem \ref{thm:periodSG}
in this  case.
On the contrary, for odd $n_1$, 
$\mathbb{Y}_{\mathrm{SG}}(n_1)$ does not have a half periodicity;
for example, for $n_1=7$,
 $2r=16$ is the full periodicity,
since $r=8$ is even.

\begin{figure}
	\begin{center}
		\includegraphics[scale=.47]{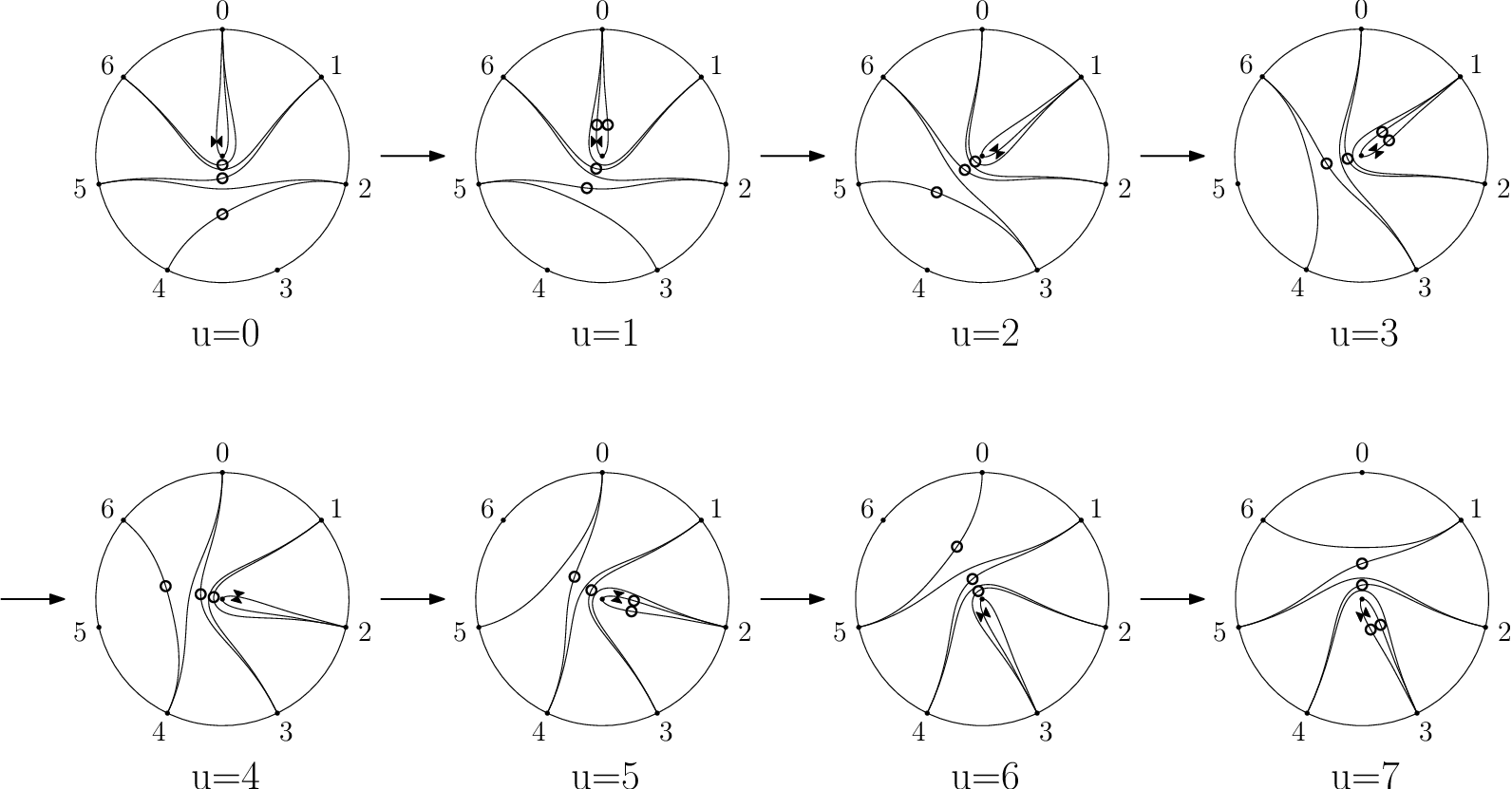}
	\end{center}
\caption{The mutation sequence 
\eqref{eq:7seq} at $u=0,\dots,7$.}
\label{fig:d7seq}
\end{figure}

\medskip
{\bf Example 2.} $F=2$, $(n_1,n_2)=(6,4)$.
 \begin{figure}
	\begin{center}
		\includegraphics[scale=.42]{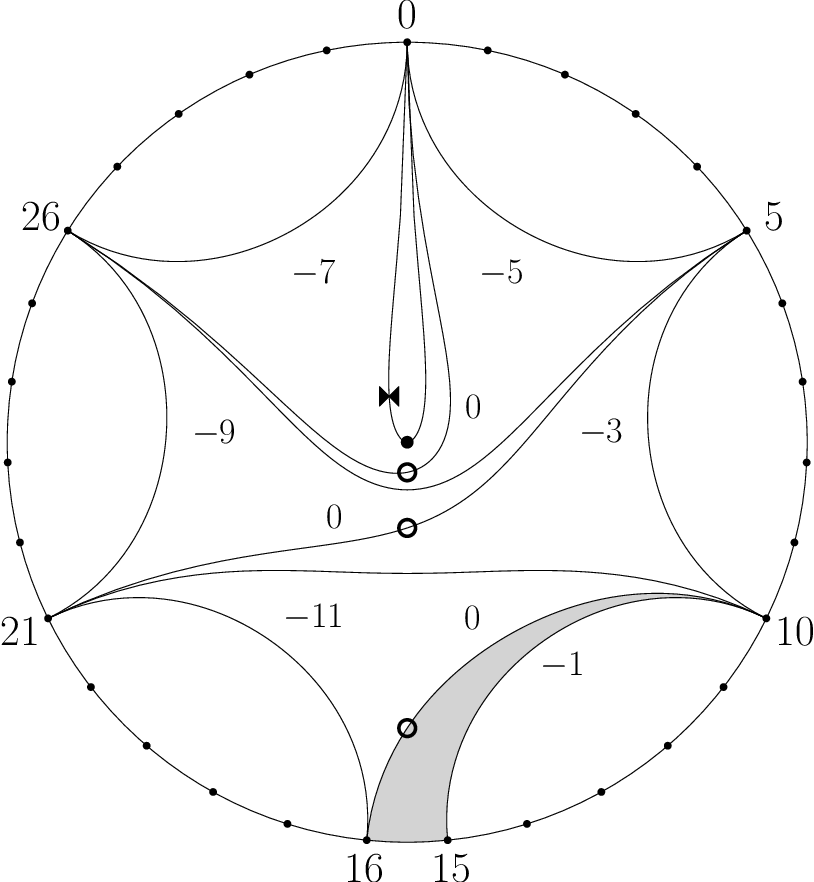}
		\hskip20pt
		\includegraphics[scale=.42]{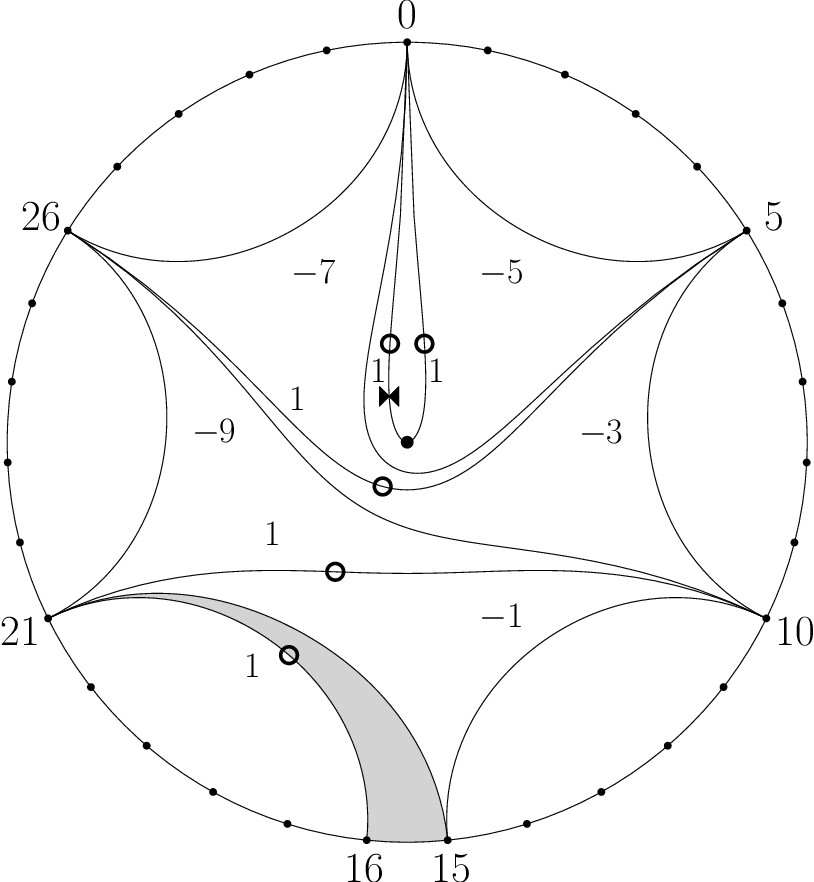}
	\end{center}
$$ u=0 \hskip163pt u=1$$
\caption{Snapshots at $u=0,1$ for the relation
\eqref{eq:Y74SG}.}
\label{fig:31gon-snap2}
\end{figure}
The triangulation $\Gamma_{\mathrm{SG}}(6,4)$
is obtained from $\Gamma_{\mathrm{RSG}}(6,4)$
by replacing the diagonals of the first generation therein
with the diagonals in 
$\Gamma_{\mathrm{SG}}(6)$.
Then, we apply the same mutation sequence  \eqref{eq:7seq}.
We have the same rotation property of triangulations
 \eqref{eq:period7}.
To see the realization of the $Y$-system,
the only new thing to be checked is the following relation
\begin{align}
\label{eq:Y74SG}
\begin{split}
Y^{(2)}_1(u-6)&Y^{(2)}_1(u+6)\\
&=
(1+Y^{(2)}_{2}(u)^{-1})^{-1}
(1+Y^{(1)}_{\overline{1}}(u)^{-1})^{-1}
(1+Y^{(1)}_{\overline{2}}(u)^{-1})^{-1}
\\
&\quad
\times
(1+Y^{(1)}_4(u-5)^{-1})^{-1}
(1+Y^{(1)}_3(u-4)^{-1})^{-1}\\
&\quad
\times
(1+Y^{(1)}_2(u-3)^{-1})^{-1}
(1+Y^{(1)}_1(u-2)^{-1})^{-1}\\
&\quad
\times
(1+Y^{(1)}_0(u-1)^{-1})^{-1}
(1+Y^{(1)}_0(u+1)^{-1})^{-1}\\
&\quad
\times
(1+Y^{(1)}_1(u+2)^{-1})^{-1}
(1+Y^{(1)}_2(u+3)^{-1})^{-1}\\
&\quad
\times
(1+Y^{(1)}_3(u+4)^{-1})^{-1}
(1+Y^{(1)}_4(u+5)^{-1})^{-1},
\end{split}
\end{align}
which replaces \eqref{eq:Y74}.
This can be done by using the snapshots
at $u=0$ and 1
in Figure \ref{fig:31gon-snap2},
which replaces Figure \ref{fig:31gon-snap}.
As for the periodicity property,
since $r=31$ is odd,
after $u=62$ step the diagonals of the first generation
show a half periodicity by the same reason of
Example 1.
Thus, we have a full periodicity of $u=124$,
which proves Theorem \ref{thm:periodSG} in this case.

\medskip
{\bf Example 3.} $F=3$, $(n_1,n_2,n_3)=(6,4,3)$.
We repeat the same procedure. 
The initial triangulation 
$\Gamma_{\mathrm{SG}}(6,4,3)$ of the punctured 106-gon is given in
 Figure  \ref{fig:106gon}.
The diagonals added in the first generation
do not participate to the relation \eqref{eq:Y743},
which, therefore,
remains unchanged.
Since $r=106$ is even, $2r=212$  is the full periodicity,
thus proving
Theorem \ref{thm:periodSG} in this case.

\subsection{Realization of SG $Y$-systems and
proof of Theorem \ref{thm:periodSG}}

\label{subsec:periodSG}

 \begin{figure}
 	\begin{center}
		\includegraphics[scale=.43]{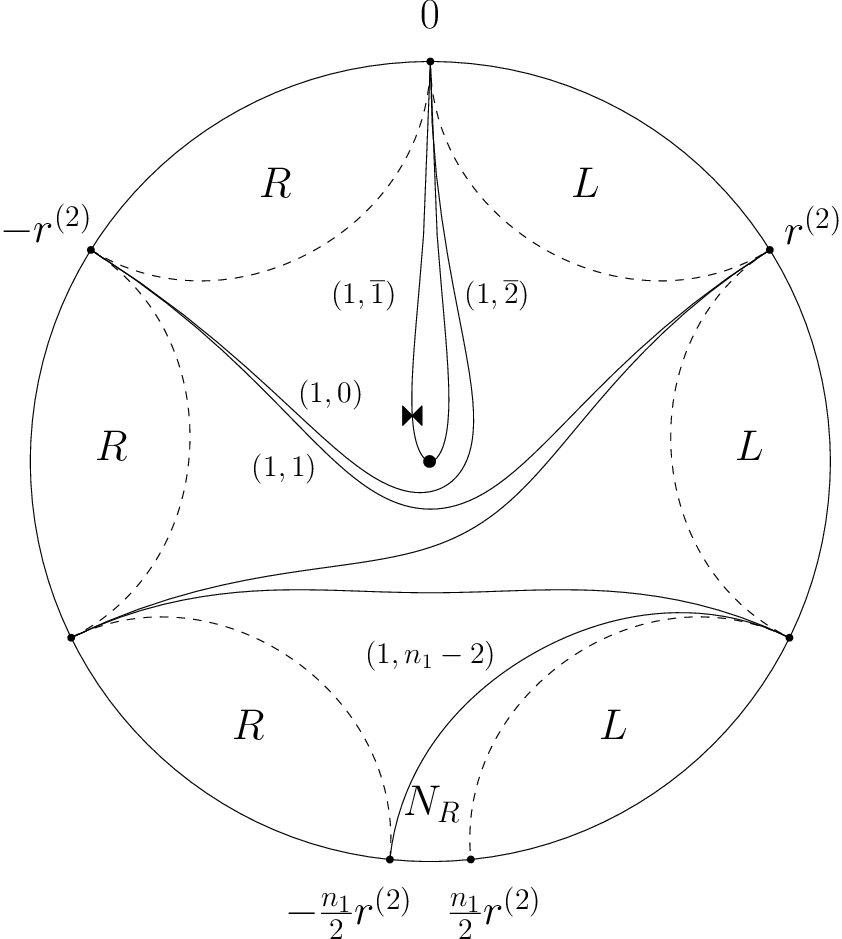}
\hskip15pt
		\includegraphics[scale=.43]{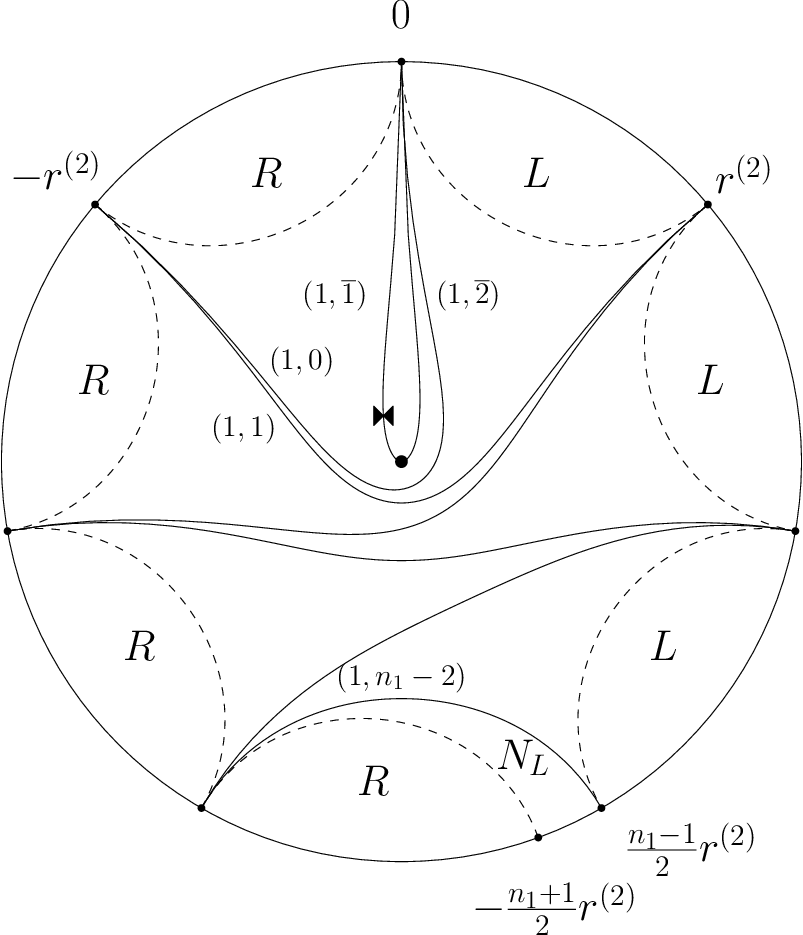}
	\end{center}
\centerline{
(a) even $n_1$
\hskip130pt
(b) odd $n_1$
}
\caption{Diagonals of the first generation
and intervals of the second
generation. The doted lines indicate the outline of   diagonals of the second generation. The label $(1,\overline{2})$ is for the plain arc ending
at the puncture.}
\label{fig:firstgen2}
\end{figure}

The construction of the triangulation
$\Gamma_{\mathrm{SG}}(n_1,\dots,n_F)$
 is done in the same way as in the previous subsection.
We only need to modify two points in
 the construction of
$\Gamma_{\mathrm{RSG}}(n_1,\dots,n_F)$.
The first one is the replacement
of the diagonals of the first generation
using Figure \ref{fig:firstgen2},
instead of Figure
\ref{fig:firstgen}.
The second one is that we
add the labels $(1,\overline{1})$, $(1,\overline{2})$
to the set $S(-1)$ 
and $(1,0)$ to the set $S(0)$.

Then, we need to show that the mutation sequence
\eqref{eq:RSGmseq} realizes  
the SG $Y$-system
$\mathbb{Y}_{\mathrm{SG}}(n_1,\dots,n_F)$.
The only thing we have to check is
the relation \eqref{eq:RSG5}.
But, this can be done 
by the snapshot method
based on Figure \ref{fig:firstgen2},
as in Example 2 of
Section \ref{subsec:SGex}.
As a corollary,
we obtain the proof of
Theorem \ref{thm:periodSG}.

\section{Dilogarithm identities}

As the second main result of our paper,
we prove the dilogarithm identities of the RSG and SG $Y$-systems
conjectured by \cite{Tateo95}.

\subsection{Conjectures on dilogarithm identities by Tateo}
Let $L(x)$ be the Rogers dilogarithm
\begin{align}
L(x)=-\frac{1}{2}
\int_0^x
\left(
\frac{\log (1-y)}{y}+
\frac{\log y}{1-y}
\right)
dy,\quad
(0<x<1).
\end{align}
The following formula holds (Euler's identity).
\begin{align}
\label{eq:Euler}
L(x) + L(1-x) = \frac{\pi^2}{6}
\quad (0\leq x \leq 1).
\end{align}

The following  identities were conjectured by
Tateo \cite{Tateo95}.
Let $\mathcal{I}_+:=\{ (a,m,u) \mid Y^{(a)}_m(u)\in \mathcal{Y}_+\}$
for the RSG/SG $Y$-systems.

\begin{conj}[\cite{Tateo95} Dilogarithm identities for RSG/SG $Y$-systems]
\label{conj:dilog1}
For any real positive  solution of the RSG/SG $Y$-system
$\mathbb{Y}_{\mathrm{RSG/SG}}(n_1,\dots,n_F)$ for $\mathcal{Y}_+$,
the following identities hold:
\begin{align}
\label{eq:dilog1}
\frac{6}{\pi^2}
\sum_{ (a,m,u)\in \mathcal{I}_+
\atop 0\leq u < 2r}
L\left(\frac{1}{1+Y^{(a)}_m(u)}
\right)
&= M_+,
\\
\label{eq:dilog2}
\frac{6}{\pi^2}
\sum_{(a,m,u)\in \mathcal{I}_+
\atop 0\leq u < 2r}
L\left(\frac{Y^{(a)}_m(u)}{1+Y^{(a)}_m(u)}
\right)
&= M_-,
\end{align}
where,
for the RSG case,
\begin{align}
M_+&=
r\left(-6 A_F +\sum_{a: \mathrm{even}} n_a+2\right),
\quad
M_-=
r\left(6 A_F + \sum_{a: \mathrm{odd}} n_a-4\right),\\
\label{eq:AF1}
A_F&=\sum_{a=1}^{F-1}
(-1)^{a+1}\frac{1}{p_aq_a}
+
(-1)^{F+1}
\frac{1}{p_Fr},
\end{align}
and, for the SG case,
\begin{align}
M_+&=
r\left( \sum_{a: \mathrm{even}} n_a +1 \right),
\quad
M_-=r\left( \sum_{a: \mathrm{odd}} n_a \right),
\end{align}
and the summation runs in the range $1\leq a \leq F$.
\end{conj}

Note that we have
\begin{align}
\label{eq:M1}
M:=M_+ +M_-
=\begin{cases}
r(\sum_{a=1}^F n_a -2) & \mbox{for RSG}\\
r(\sum_{a=1}^F n_a +1 )& \mbox{for SG}.\\
\end{cases}
\end{align}
In either case, $M$ is equal to the cardinality of
the set
$\{(a,m,u)\in \mathcal{I}_+
\mid
 0\leq u < 2r
\}$
which is also equal to the total number of the forward mutation points
in the period $0\leq u < 2r$.
This means that, by \eqref{eq:Euler},
the identities \eqref{eq:dilog1} and \eqref{eq:dilog2}
are equivalent to each other.

For $F=1$, the identities reduce to the dilogarithm identities
of type $A$ and $D$,
which were proved by \cite{Frenkel95} for type $A$,
and by \cite{Chapoton05} for type $D$.
The case $F=2$ was proved by \cite{Nakanishi10b}.
We will give a proof of Conjecture \ref{conj:dilog1} in full generality
based on our formulation of the RSG/SG $Y$-systems.

\subsection{Dilogarithm identities in general form}
Formulas such as \eqref{eq:dilog1} and \eqref{eq:dilog2}
were once very formidable to prove,
 but nowadays they are rather well understood from the
point of view of cluster algebras.
According to a general theorem \cite[Theorem 6.1]{Nakanishi10c},
{\em 
a dilogarithm identity is associated
to any period of labeled seeds of a cluster algebra.}
Furthermore, the proof of  \cite[Theorem 6.1]{Nakanishi10c}
works also for any partial period of labels seeds
(= period of unlabeled seed).
As we have shown,
the underlying cluster algebra of our $Y$-system
has exactly such periodicity.
Thus, we automatically obtains the associated dilogarithm identity.

To present these dilogarithm identities explicitly,
we introduce some terminology from cluster algebras.
Let $y_i=y^{(a)}_{m,s}:=y^{(a)}_{m,s}(0)$ be the initial $y$-variables
of the cluster algebra associated to a RSG/SG $Y$-system.
Then, any $y$-variable  $y^{(a)}_{m,s}(u)$
are in the {\em universal semifield
$\mathbb{Q}_+(y)$ of $y$}, that is, the semifield consisting
of the rational functions in the variables $y=(y_i)$.
Let ${\mathrm{Trop}}(y)$ be the {\em tropical semifield generated by
$y=(y_i)$},
consisting of the Laurent monomials in $y$ with coefficient $1$
endowed with the ordinary multiplication and the following tropical addition
\begin{align}
\prod_{i} y_i^{m_i}
\oplus
\prod_{i} y_i^{n_i}
:=
\prod_{i} y_i^{\min(m_i,n_i)}.
\end{align}
There is the canonical semifield homomorphism
(the {\em tropicalization map})
$\pi: \mathbb{Q}_+(y)  \rightarrow 
{\mathrm{Trop}}(y)$
defined by
$
\pi(y_i)= y_i
$
and
$\pi(c)=1$  ($c\in \mathbb{Q}_+$).
Since our $Y$-variables $Y^{(a)}_m(u)=y^{(a)}_{m,s}(u)$ are in $\mathbb{Q}_+(y)$,
we can apply the tropicalization map $\pi$ to them.

\begin{defn}[\cite{Fomin07}] For any $Y$-variable $Y^{(a)}_m(u)$,
the integer vector $c=c(Y^{(a)}_m(u))=
(c_i)$ defined by
\begin{align}
\pi( Y^{(a)}_m(u))=\prod_{i}  y_i{}^{c_i}
\end{align}
is called the {\em $c$-vector\/} of $Y^{(a)}_m(u)$.
\end{defn}

The following fact is well known.
\begin{thm}[\cite{Derksen10}]
Any $c$-vector of a cluster algebras of type $A$ or $D$
 is a nonzero vector, and
its components are either all nonnegative or all nonpositive.
\end{thm}

Based on the above theorem, we introduce the following notion.

\begin{defn}
To each $Y^{(a)}_m(u)\in \mathcal{Y}_+$
we attach a sign $\varepsilon=\varepsilon(Y^{(a)}_m(u))$
such that $\varepsilon$ is $+$ (resp. $-$)
if the $c$-vector of $Y^{(a)}_m(u)$ is a positive vector
(resp. negative vector).
We call $\varepsilon$ the {\em tropical sign of $Y^{(a)}_m(u)$}.
\end{defn}

Now we can state the dilogarithm identity 
 associated to the period $2r$
of the RSG/SG $Y$-systems
for Theorems \ref{thm:periodRSG} and 
 \ref{thm:periodSG} in a general form.

\begin{thm}[Dilogarithm identities in general form
{\cite[Theorem 6.1]{Nakanishi10c}}]
\label{thm:dilog2}
For any real positive  solution of the RSG/SG $Y$-system
$\mathbb{Y}_{\mathrm{RSG/SG}}(n_1,\dots,n_F)$ for $\mathcal{Y}_+$,
the following identities hold.
\begin{align}
\label{eq:dilog3}
\frac{6}{\pi^2}
\sum_{(a,m,u)\in \mathcal{I}_+
\atop 0\leq u < 2r}
L\left(\frac{1}{1+Y^{(a)}_m(u)}
\right)
&= N_+,\\
\label{eq:dilog4}
\frac{6}{\pi^2}
\sum_{(a,m,u)\in \mathcal{I}_+
\atop 0\leq u < 2r}
L\left(\frac{Y^{(a)}_m(u)}{1+Y^{(a)}_m(u)}
\right)
&= N_-,
\end{align}
where
\begin{align}
\label{eq:Npm1}
\begin{split}
N_+&=|\{(a,m,u)\in \mathcal{I}_+ \mid 0\leq u < 2r, \varepsilon(
Y^{(a)}_m(u))=+\}|,\\
N_-&=|\{(a,m,u)\in \mathcal{I}_+ \mid 0\leq u < 2r, \varepsilon(
Y^{(a)}_m(u))=-\}|.
\end{split}
\end{align}
\end{thm}

Therefore, Conjecture \ref{conj:dilog1} reduces to a
counting problem on $N_{\pm}$;
i.e., it reduces to show 
the equalities
\begin{align}
\label{eq:NM1}
N_{\pm}= M_{\pm}.
\end{align}
Note that we have
\begin{align}
\label{eq:N1}
N:=N_+ +N_- = M,
\end{align}
where $M$ is the number in \eqref{eq:M1}.
\subsection{Counting formula for $N_\pm$}

To perform the counting of $N_{\pm}$,
we employ the description of $c$-vectors in terms
of  {\em laminations} by \cite{Fomin08b}.
The idea is as follows.
Let $Q$ be a quadrilateral surrounding the diagonal
corresponding to $Y^{(a)}_m(u)$ at $u$.
Then, the tropical sign $\varepsilon=\varepsilon(Y^{(a)}_m(u))$ 
can be determined by the way
in which the elementary laminations
associated with the initial triangulation
cross the quadrilateral $Q$ as follows
(the dotted line is a lamination).
\medskip
\begin{align}
\label{eq:lami1}
\begin{xy}
%
%
(10,-15)*{\varepsilon=+},
(10,0)*{\cir<2pt>{}},
\ar@{-} (0,0);(20,0)
\ar@{-} (20,0);(10,10)
\ar@{-} (10,10);(0,0)
\ar@{-} (10,-10);(20,0)
\ar@{-} (10,-10);(0,0)
\ar@{--} (0,10);(20,-10)
\end{xy}
\hskip50pt
\begin{xy}
%
%
(10,-15)*{\varepsilon=-},
(10,0)*{\cir<2pt>{}},
\ar@{-} (0,0);(20,0)
\ar@{-} (20,0);(10,10)
\ar@{-} (10,10);(0,0)
\ar@{-} (10,-10);(20,0)
\ar@{-} (10,-10);(0,0)
\ar@{--} (20,10);(0,-10)
\end{xy}
\end{align}
There is also an extra rule for arcs ending at the puncture
for type $D$.
See \cite[Fig.13]{Nakanishi12} for a summary of this rule.

By carefully studying  the possible configurations of quadrilaterals,
 we can obtain the following formula for $N_{\pm}$,
which is the main result of this section.
\begin{thm}[Counting formula]
\label{thm:Npm1}
For the RSG $Y$-system, we have
\begin{align}
N_+&=
r\left(\sum_{a: \mathrm{even}} n_a+2\right) -6r^{(2)},
\quad
N_-=
r\left( \sum_{a: \mathrm{odd}} n_a-4\right) + 6 r^{(2)}.
\end{align}
For the SG $Y$-system, we have
\begin{align}
N_+&=
r\left(\sum_{a: \mathrm{even}} n_a+1\right) ,\quad
N_-=
r\left( \sum_{a: \mathrm{odd}} n_a\right).
\end{align}
\end{thm}

For the SG case, the equalities
\eqref{eq:NM1} immediately follow from
Theorem \ref{thm:Npm1}.
On the other hand, for the RSG case,
the equalities \eqref{eq:NM1} follows from
Theorem \ref{thm:Npm1} and the following equality:

\begin{prop} For any positive integer $F$, we have
\label{prop:Ar1}
\begin{align}
\label{eq:Ar1}
A_F =
\frac{r^{(2)}}{r}.
\end{align}
\end{prop}

Therefore, we proved Conjecture \ref{conj:dilog1}.

\begin{cor} Conjecture \ref{conj:dilog1} is true.
\end{cor}

The rest of this section will be devoted to present the proofs of
Theorem \ref{thm:Npm1} and Proposition \ref{prop:Ar1}.

\subsection{Proof of Theorem \ref{thm:Npm1}}

We consider the RSG and SG cases together.

We first decompose the numbers $N_{\pm}$, $N$ in 
\eqref{eq:Npm1}  and \eqref{eq:N1} into the
contribution from each generation $a$ as follows:
\begin{align}
N_{\pm}=\sum_{a=1}^F N_{\pm,a},
\quad
N=\sum_{a=1}^F N_a,
\end{align}
where
\begin{align}
\label{eq:Npm2}
N_{\pm,a}=&\, |\{(m,u) \mid (a,m,u) \in N_{\pm} \}|,\\
\label{eq:Npm3}
N_a:=&\,
N_{+,a} + N_{-,a} = 
\begin{cases}
r(n_{1}-2) & a=1\ \mbox{for RSG}\\
r(n_{1}+1) & a=1\ \mbox{for SG}\\
rn_a & a=2,\dots, F.
\end{cases}
\end{align}

We first count $N_{\pm,a}$ for $a\geq 2$,
which is common for both RSG and SG cases.
 Let $O_a$ be the numbers defined by
\begin{align}
\label{eq:O1}
O_a=
 (r^{(a)}-r^{(a+1)}) p_{a}
+ (r^{(a+1)}-r^{(a+2)}) p_{a+1}.
\end{align}

\begin{prop}
\label{prop:Na1} For $a=2,\dots, F$, the following holds.
\par
(a). For odd $a$, we have $N_{+,a}= 
O_a$,
$N_{-,a}= 
N_a-O_a$.
\par
(b). For even $a$, we have 
$N_{+,a}= 
N_a-O_a$, $N_{-,a}= 
O_a$.
\end{prop}

\begin{figure}
	\begin{center}
		\includegraphics[scale=.45]{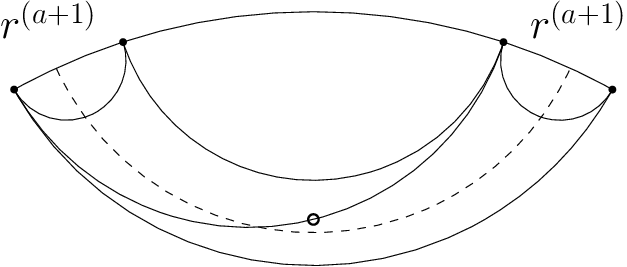}
		\hskip30pt
		\includegraphics[scale=.45]{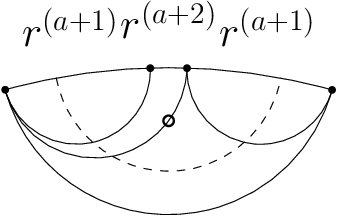}
	\end{center}
	\centerline{\hskip35pt(a) $m\neq n_a$ \hskip85pt
	(b) $m=n_a$}
\caption{Surrounding quadrilateral of forward mutation point.
The dotted line represents a lamination yielding the tropical sign
$\varepsilon=+$.}
\label{fig:quad2}
\end{figure}

\begin{proof}
(a). Let $a\geq 3$ be odd.
For each $m=1,\dots,n_a$,
 we count the total number of  $Y^{(a)}_m(u)\in \mathcal{Y}_+$ ($0\leq u < 2r$)
such that $\varepsilon(Y^{(a)}_m(u))=+$ for each $m$.
Then, summing up over $m$, we obtain $N_{+,a}$.

First, consider the case $m\neq n_a$.
The surrounding quadrilateral of the forward mutation point
for $Y^{(a)}_m(u)$ is given in
Figure \ref{fig:quad2} (a).
By \eqref{eq:lami1},
 $\varepsilon(Y^{(a)}_m(u))=+$ if and only if
 at least one of the initial laminations  
 crosses the surrounding quadrilateral
 ``transversally''
 as in Figure \ref{fig:quad2}.
Then, it is easy to see that such a lamination should have the label
$(a,m)_s$ ($s=1,\dots,p_a$) of the same type as $Y^{(a)}_m(u)$.
Furthermore, from Figure \ref{fig:quad2}, 
each lamination of label $(a,m)_s$ 
 crosses transversally
 the surrounding quadrilateral
 $r^{(a+1)}$ times during the period $0\leq u < 2r$.
Since there are $p_a$ such laminations,
the total number of such $Y^{(a)}_m(u)$ is $r^{(a+1)} p_a$.

Next, consider the case $m= n_a$,
which is a little more complicated.
The surrounding quadrilateral of the forward mutation point
for $Y^{(a)}_{n_a}(u)$ is given in
Figure \ref{fig:quad2} (b).
There are two cases 
giving $\varepsilon(Y^{(a)}_m(u))=+$,
which occur exclusively of each other.
\begin{itemize}
\item[Case 1.] A lamination of label $(a+1,1)_s$ 
crosses transversally  the surrounding quadrilateral.
\item[Case 2.] A lamination of label $(a,n_a)_s$ 
crosses transversally  the surrounding quadrilateral,
but no lamination of label $(a+1,1)_{s'}$ 
crosses transversally  the surrounding quadrilateral.
\end{itemize}
By a similar counting as before, 
the total number of such $Y^{(a)}_m(u)$ is 
$(r^{(a+1)}-r^{(a+2)})p_{a+1}$ for Case 1,
and $r^{(a+2)}p_a$ for Case 2.

Summing up everything, we have
\begin{align}
\begin{split}
N_{+,a}&= (n_a-1)r^{(a+1)} p_a
+(r^{(a+1)}-r^{(a+2)})p_{a+1}
+r^{(a+2)}p_a\\
&=
 (r^{(a)}-r^{(a+1)}) p_{a}
+ (r^{(a+1)}-r^{(a+2)}) p_{a+1}
=O_a
\end{split}
\end{align}
as desired, where we used \eqref{eq:rr4} in the second equality.
\par
(b). Since Figure \ref{fig:quad2} becomes its mirror image,
the same counting as above gives the number $N_{-,a}$.
\end{proof}

The counting of $N_{\pm,1}$ is similar, but a little more tricky
for the RSG case,
and this is where  the RSG and SG cases differ.
Let $O_1$ be the number defined by
\begin{align}
\label{eq:O2}
O_1=
\begin{cases}
2r-6r^{(2)}+ (r^{(2)}-r^{(3)})p_2 & \mbox{for RSG}\\
r+ (r^{(2)}-r^{(3)})p_2 & \mbox{for SG}.
\end{cases}
\end{align}

\begin{prop}
\label{prop:Na2}
We have $N_{+,1}= 
O_1$,
$N_{-,1}= 
N_1-O_1$.
\end{prop}

\begin{proof}
{\em The RSG case.}
For the surrounding quadrilateral of the forward mutation point for 
$Y^{(1)}_m(u)$, see Figure \ref{fig:firstgen}.

First, consider the case $m\neq n_1-2$.
There are three mutually exclusive cases giving
 $\varepsilon(Y^{(1)}_m(u))=+$.
\begin{itemize}
\item[Case 1.] A lamination of label $(1,m)$ 
crosses transversally  the surrounding quadrilateral.
\item[Case 2.] A lamination of label $(1,n_1-2-m)$ 
crosses transversally  the surrounding quadrilateral.
\item[Case 3.] A lamination of label $(1,n_1-1-m)$ 
crosses transversally  the surrounding quadrilateral.
\end{itemize}
The total number of such $Y^{(1)}_m(u)$ is 
$r^{(2)}$ for  Case 1,
$r^{(2)}-r^{(3)}$ for  Case 2,
and 
$r^{(3)}$ for  Case 3.

Next, consider the case $m= n_1-2$.
There are three mutually exclusive cases 
giving $\varepsilon(Y^{(1)}_m(u))=+$.
\begin{itemize}
\item[Case 1.] A lamination of label $(2,1)_s$ 
crosses transversally  the surrounding quadrilateral.
\item[Case 2.] A lamination of label $(1,n_1-2)$ 
crosses transversally  the surrounding quadrilateral,
but no lamination of label $(2,1)_s$   crosses transversally
 the surrounding quadrilateral.
\item[Case 3.] A lamination of label $(1,1)$ 
crosses transversally  the surrounding quadrilateral.
\end{itemize}
The total number of such $Y^{(1)}_m(u)$ is 
$(r^{(2)}-r^{(3)})p_2$ for  Case 1,
and $r^{(3)}$ for  Cases 2 and 3.
Summing up everything, we have
\begin{align}
\begin{split}
N_{+,1}&= 2(n_1-3)r^{(2)} 
+(r^{(2)}-r^{(3)})p_{2}
+2r^{(3)}\\
&=
2r-6r^{(2)}
+ (r^{(2)}-r^{(3)}) p_{2}
=
O_1.
\end{split}
\end{align}

\begin{figure}
	\begin{center}
		\includegraphics[scale=.4]{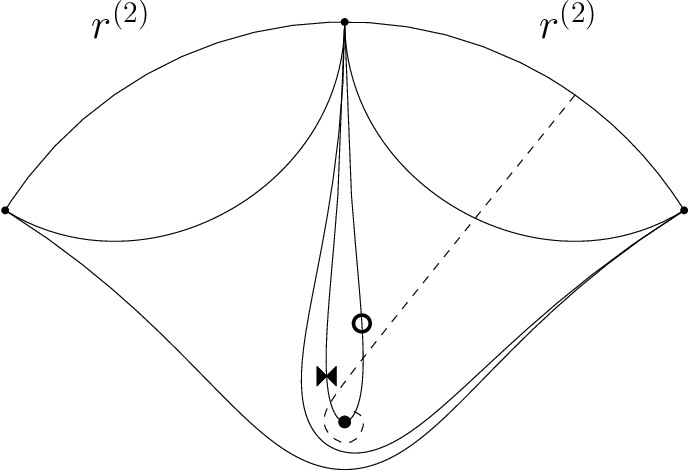}
		\hskip15pt
		\includegraphics[scale=.4]{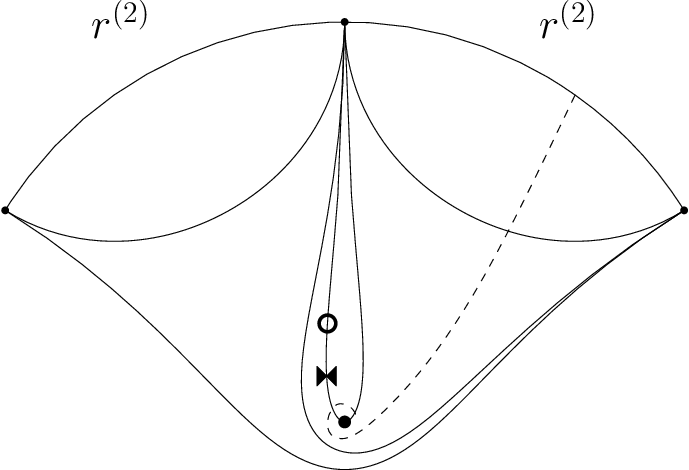}
	\end{center}
	\centerline{(a) $m=\overline{2}$ \hskip105pt
	(b) $m=\overline{1}$\hskip5pt}
\caption{Surrounding ``quadrilateral''
of forward mutation point $(1,\overline{2})$ or $(1,\overline{1})$.
The dotted line represents a lamination yielding the tropical sign
$\varepsilon=+$.
}
\label{fig:quad3}
\end{figure}

{\em The SG case.}
First, consider the case $m=\overline{2},\overline{1}$,
which are the arcs ending at the puncture.
Then,
$\varepsilon(Y^{(1)}_m(u))=+$
if and only if
there is a lamination of label 
$(1,m)$
which
 crosses the surrounding ``quadrilateral''
 as in Figure \ref{fig:quad3}.
The total number of such $Y^{(1)}_m(u)$ is 
$r^{(2)}$.

Next, consider the case $0\leq m\leq n_1-3$.
Then,
$\varepsilon(Y^{(1)}_m(u))=+$
if and only if
there is a lamination of label
$(1,m)$
which crosses
transversally  the surrounding quadrilateral.
The total number of such $Y^{(1)}_m(u)$ is 
again $r^{(2)}$.
Finally, consider the case $m= n_1-2$.
There are two cases 
giving $\varepsilon(Y^{(1)}_m(u))=+$,
which occur exclusively of each other.
\begin{itemize}
\item[Case 1.] A lamination of label $(2,1)_s$ 
crosses transversally  the surrounding quadrilateral.
\item[Case 2.] A lamination of label $(1,n_1-2)$ 
crosses transversally  the surrounding quadrilateral,
but no lamination of label $(2,1)_s$  crosses transversally
 the surrounding quadrilateral.
\end{itemize}
The total number of such $Y^{(1)}_m(u)$ is 
$(r^{(2)}-r^{(3)})p_2$ for  Case 1,
and $r^{(3)}$ for  Case 2.
Summing up everything, we have
\begin{align}
\begin{split}
N_{+,1}&= n_1r^{(2)} 
+(r^{(2)}-r^{(3)})p_{2}
+r^{(3)}
=
 r
+ (r^{(2)}-r^{(3)}) p_{2}
=
O_1.
\end{split}
\end{align}

\end{proof}

Now we are ready to prove 
Theorem \ref{thm:Npm1}.
By Propositions \ref{prop:Na1} and \ref{prop:Na2}
and $r^{(F+1)}=r^{(F+2)}=1$,
we have
\begin{align}
\begin{split}
N_+&=
\sum_{a: \mathrm{even}}
N_{a}
+O_1 + \sum_{a=2}^{F-1} (-1)^{a-1} O_a
=
\sum_{a: \mathrm{even}}
N_{a}
+O_1- (r^{(2)}-r^{(3)})p_2\\
&=
\begin{cases}
\sum_{a: \mathrm{even}}
N_{a}
+2r-6r^{(2)}
& \mbox{for RSG}\\
\sum_{a: \mathrm{even}}
N_{a}
+r
& \mbox{for SG},
\end{cases}
\end{split}
\end{align}
and
\begin{align}
\begin{split}
N_- &=
\sum_{a: \mathrm{odd}}
N_{a}
-O_1 - \sum_{a=2}^{F-1} (-1)^{a-1} O_a
=
\sum_{a: \mathrm{odd}}
N_{a}
-O_1+ (r^{(2)}-r^{(3)})p_2\\
&=
\begin{cases}
\sum_{a: \mathrm{odd}}
N_{a}
-2r+6r^{(2)}
& \mbox{for RSG}\\
\sum_{a: \mathrm{odd}}
N_{a}
-r
& \mbox{for SG}.
\end{cases}
\end{split}
\end{align}
Taking into account   \eqref{eq:Npm3},
 they agree with the formulas in Theorem \ref{thm:Npm1}.

This complete the proof of Theorem \ref{thm:Npm1}.

\subsection{Proof of Proposition \ref{prop:Ar1}}
%

Let us prove the equality \eqref{eq:Ar1}.
Dividing (\ref{eq:pq-qp}) for $k=1$ and $a>1$
by $p_aq_a$, 
then using  the relations $q_{a}=p_{a+1}$ and $q_a^{(2)}=p_{a+1}^{(2)}$,
we have
\begin{align}
	\label{eq:A_F-lin}
	(-1)^{a+1}\frac{1}{p_aq_a}=
\frac{q_a^{(2)}}{q_a}	-	\frac{p_a^{(2)}}{p_a}
	=
	\frac{p_{a+1}^{(2)}}{p_{a+1}}
-	\frac{p_a^{(2)}}{p_a}.
\end{align}
Putting it into
 the right hand side
of \eqref{eq:AF1},
we obtain
\begin{align}
\begin{split}
A_F&=
\frac{1}{p_1q_1}+
\sum_{a=2}^{F-1}\left(\frac{p_{a+1}^{(2)}}{p_{a+1}}- \frac{p_{a}^{(2)}}{p_{a}}\right)
+\frac{(-1)^{F+1}}{p_Fr}\\
&=
\frac{1}{p_1q_1}-
\frac{p_{2}^{(2)}}{p_{2}}+\frac{p_F^{(2)}}{p_F}
+
\frac{(-1)^{F+1}}{p_Fr}.
\end{split}
\end{align}
The first two summands cancel since $p_1=p_2^{(2)}=1$ and $q_1=p_2$,
 and we get
\begin{align}
\begin{split}
A_F&
=\frac{p_F^{(2)}r+(-1)^{F+1}}{p_Fr}
=
\frac{p_F^{(2)}q_F+p_F^{(2)}p_F+(-1)^{F+1}}{p_Fr}.
\end{split}
\end{align}
Finally, using (\ref{eq:pq-qp}) with $k=1$ and $a=F$, we obtain
\begin{align}
A_F=
\frac{p_F(q_F^{(2)}+p_F^{(2)})}{p_Fr}=
\frac{r^{(2)}}{r}.
\end{align}

\section{RSG and SG $T$-systems}

Here we present the RSG and SG $T$-systems,
which are the companions of the RSG and SG $Y$-systems.
These $T$-systems share exactly the same periodicities with 
the corresponding $Y$-systems.
They are new in the literature,
and they might have a representation theoretical
interpretation, possibly by
a certain variation of the Hecke algebras or the quantum groups
at roots of unity.
See \cite{Nakanishi10c} for $T$-systems
and $Y$-systems in a more general setting.

\subsection{RSG $T$-systems}

To the mutation sequence
\eqref{eq:RSGmseq},
one can attach, not only $Y$-variables, but also  {\em $T$-variables\/}
by identifying the cluster variables ($x$-variables)
$x^{(a)}_{m,s}(u)$ with $T^{(a)}_m(u-p_a)$ only at  forward mutation
points at time $u$.
The $T$-systems are a family of algebraic relations
satisfied by these $T$-variables.
One can directly derive these
 $T$-systems by applying the snapshot method
 to $T$-variables.
Alternatively,
using the {\em duality\/} between $T$-systems
and $Y$-systems \cite[Prop.5.6]{Nakanishi10c},
one can also translate the RSG $Y$-systems into
the corresponding $T$-systems.
Skipping the derivation, here we present them
as the definition of the RSG $T$-systems.

Consider the case $n_1\neq 2$.
For a  sequence $(n_1,\dots,n_F)$ 
we introduce the $T$-variables
$ T^{(a)}_m(u)$,
where $u\in \mathbb{Z}$, $a=1,\dots, F$,
and  $m$ runs over the set
 specified by \eqref{eq:RSGam}.

\begin{defn}
For $n_1\neq 2$,
the {\em reduced sine-Gordon (RSG) $T$-system\/}
$\mathbb{T}_{\mathrm{RSG}}(n_1,\dots,n_F)$ is the following system of
relations:
For $(a,m)=(1,1)$,
\begin{align}
\label{eq:RSGT1}
\begin{split}
T^{(1)}_1(u-p_1)T^{(1)}_1(u+p_1)
&= T^{(1)}_2(u)T^{(2)}_1(u) + T^{(2)}_1(u-2)T^{(2)}_1(u+2),
\end{split}
\end{align}
for $(a,m)=(1,m)$ with $m\neq 1, n_1-2$,
\begin{align}
\label{eq:RSGT2}
\begin{split}
T^{(1)}_m(u-p_1)T^{(1)}_m(u+p_1)
&=\prod_{(b,k)\sim (a,m)}T^{(b)}_k(u) \\
&\quad + T^{(2)}_{1}(u-1-m)T^{(2)}_1(u+1+m),
\end{split}
\end{align}
for $(a,m)=(1,n_1-2)$,
\begin{align}
\label{eq:RSGT3}
\begin{split}
T^{(1)}_{n_1-2}(u-p_1)T^{(1)}_{n_1-2}(u+p_1)
&= T^{(1)}_{n_1-3}(u)T^{(3)}_1(u)\\
&\quad + T^{(2)}_1(u-p_2+1)T^{(2)}_1(u+p_2-1),
\end{split}
\end{align}
for $(a,m)$ with $a\geq 2$, $m\neq n_a$,
\begin{align}
\label{eq:RSGT4}
\begin{split}
&T^{(a)}_m(u-p_a)T^{(a)}_m(u+p_a)
=\prod_{(b,k)\sim (a,m)}T^{(b)}_k(u) \\
&+
 T^{(a+1)}_1(u-p_{a+1}+(n_a+1-m)p_a)
T^{(a+1)}_{1}(u+p_{a+1}-(n_a+1-m)p_a),
\end{split}
\end{align}
for $(a,m)=(a,n_a)$ with $a\geq 2$,
\begin{align}
\label{eq:RSGT5}
\begin{split}
T^{(a)}_{n_a}(u-p_a)T^{(a)}_{n_a}(u+p_a)
&=T^{(a)}_{n_a-1}(u)T^{(a+2)}_1(u)\\
&\quad +
 T^{(a+1)}_1(u-p_{a+1}+p_a)T^{(a+1)}_1(u+p_{a+1}-p_a),
\end{split}
\end{align}
where $(b,k)\sim (a,m)$ means $(b,k)$ is adjacent to
$(a,m)$ in the diagram $X_{\mathrm{RSG}}(n_1,\dots,n_F)$
in Figure \ref{fig:XA}.
\end{defn}

When $n_1=2$ with $F\geq 2$,
we reset the $T$-variables
$ T^{(a)}_m(u)
 $,
  where $u\in \mathbb{Z}$, $a=2,\dots, F$,
and  $m=1,\dots,n_a$.

\begin{defn}
For $F\geq 2$,
the {\em reduced sine-Gordon (RSG) $T$-system\/}
$\mathbb{T}_{\mathrm{RSG}}(2,n_2,\dots,n_F)$ is the following system of
relations:
For $(a,m)$ with $m\neq n_a$,
\begin{align}
\label{eq:RSGT6}
\begin{split}
&T^{(a)}_m(u-p_a)T^{(a)}_m(u+p_a)
=\prod_{(b,k)\sim (a,m)}T^{(b)}_k(u) \\
&+
 T^{(a+1)}_1(u-p_{a+1}+(n_a+1-m)p_a)
T^{(a+1)}_{1}(u+p_{a+1}-(n_a+1-m)p_a),
\end{split}
\end{align}
for $(a,m)=(a,n_a)$,
\begin{align}
\label{eq:RSGT7}
\begin{split}
T^{(a)}_{n_a}(u-p_a)&T^{(a)}_{n_a}(u+p_a)
=T^{(a)}_{n_a-1}(u)T^{(a+2)}_1(u)\\
& +
 T^{(a+1)}_1(u-p_{a+1}+p_a)T^{(a+1)}_1(u+p_{a+1}-p_a),
\end{split}
\end{align}
where $(b,k)\sim (a,m)$ means $(b,k)$ is adjacent to
$(a,m)$ in the diagram $X_{\mathrm{RSG}}(2,n_2,\dots,n_F)$
in Figure \ref{fig:XA2}.
\end{defn}

\subsection{SG $T$-systems}


For a sequence $(n_1,\dots,n_F)$,
we introduce the  $T$-variables
$ T^{(a)}_m(u) $,
where $u\in \mathbb{Z}$, $a=1,\dots, F$,
and  $m$ runs over the set
 specified by \eqref{eq:SGam}.

\begin{defn}
The {\em sine-Gordon (SG) $T$-system\/}
$\mathbb{T}_{\mathrm{SG}}(n_1,\dots,n_F)$ is the following system of
relations:
\begin{itemize}
\item[(i).] the relation for $(a,m)=(1,\overline{2}), (1,\overline{1})$,
\begin{align}
\label{eq:SGT1}
\begin{split}
T^{(1)}_{\overline{2}}(u-p_1)T^{(1)}_{\overline{2}}(u+p_1)
&=T^{(1)}_0(u)+T^{(2)}_1(u),\\
T^{(1)}_{\overline{1}}(u-p_1)T^{(1)}_{\overline{1}}(u+p_1)
&=T^{(1)}_0(u)+T^{(2)}_1(u),
\end{split}
\end{align}
\item[(ii).] 
 the relation \eqref{eq:RSGT2} for 
$(a,m)=(1,m)$ with $m\neq \overline{2},\overline{1}, n_1-2$,
\item[(iii).] 
 the relation \eqref{eq:RSGT3} for 
$(a,m)=(1,n_1-2)$ ,
\item[(iv).] 
 the relation \eqref{eq:RSGT4} for 
$(a,m)$ with  $a\geq 2$,  $m\neq  n_a$,
\item[(v).] 
the relation \eqref{eq:RSGT5} for 
$(a,m)=(a,n_a)$ with $a\geq 2$,
\end{itemize}
\par\noindent
where the adjacency diagram in \eqref{eq:RSGT2} 
and \eqref{eq:RSGT4} is  replaced with
 $X_{\mathrm{SG}}(n_1,\dots,n_F)$
in Figure \ref{fig:XD}.
\end{defn}

\subsection{Periodicity}

The following theorem is simultaneously proved
with Theorems \ref{thm:periodRSG} and 
 \ref{thm:periodSG},
because it follows from the periodicity of  seeds,
in particular, the periodicity of  $x$-variables.
 For $F=2$, it was proved by \cite{Nakanishi10b}.
 
\begin{thm}
\label{thm:periodRSGT}
The RSG $T$-systems
share the same periodicity with
the corresponding RSG $Y$-systems in Theorems \ref{thm:periodRSG}.
The SG $T$-systems
shares the same periodicity with
the corresponding SG $Y$-systems in Theorems \ref{thm:periodSG}.
 \end{thm}

\bibliography{../../../biblist/biblist.bib}

\end{document}